\documentclass[12pt,a4paper]{article}

\usepackage{graphicx}   %
\usepackage{enumerate}  %
\usepackage{latexsym}   %
\usepackage{amsmath}    %
\usepackage{amsbsy}     %
\usepackage{amsfonts}   %
\usepackage{mathrsfs}   %
\usepackage{amssymb}    %
\usepackage{array}      %
\usepackage{theorem}
\usepackage{cancel}    %
\usepackage{mathrsfs,bm}
\usepackage{comment}
\usepackage{color}
\usepackage[textsize=small]{todonotes}
\definecolor{winered}{rgb}{0.8,0,0}
\definecolor{bloodorange}{rgb}{0.95,0.05,0}

\newcommand{\jump}[1]{\ensuremath{[\hspace{-0.34ex}[#1]\hspace{-0.34ex}]} }

\usepackage[colorlinks=true]{hyperref}
\hypersetup{
    linkcolor=blue,
    urlcolor=blue,
    citecolor=blue
}

\theorembodyfont{\rmfamily}
\newtheorem{proof}{Proof.}

\newtheorem{defn}{Definition}
\newtheorem{rem}{Remark}
\newtheorem{exmp}{Example}

\theorembodyfont{\itshape}
\newtheorem{thm}{Theorem}
\newtheorem{cor}[thm]{Corollary}

\newtheorem{lem}[thm]{Lemma}

\numberwithin{equation}{section}
\def\bB{{\mathbb B}}

\def\R{{\mathbb R}}
\def\bS{{\mathbb S}}
\def\bB{{\mathbb B}}
\def\N{{\mathbb{N}}}

\def\ds{\displaystyle}
\def\ts{\textstyle}

\def\laplace{\varDelta}

\def\laplace{{\varDelta}}

\usepackage{comment}
\usepackage{cite}

\begin{document}
\label{page:t}

\begin{center}
{\LARGE\textbf{
    Kobayashi--Warren--Carter type systems \\ with nonhomogeneous Dirichlet \\[0.5ex] boundary data for crystalline orientation \footnotemark[1]
  }}
\end{center}
\vspace{7ex}
\begin{center}
{\sc Salvador Moll}\footnotemark[2]\\
Department d'An\`{a}lisi Matem\`{a}tica, Universitat de Val\`{e}ncia\\
C/Dr. Moliner, 50, Burjassot, Spain\\
{\ttfamily j.salvador.moll@uv.es}
\vspace{2ex}

{\sc Ken Shirakawa}\footnotemark[3]\\
Department of Mathematics, Faculty of Education, Chiba University\\
1-33, Yayoi-cho, Inage-ku, Chiba, 263--8522, Japan\\
{\ttfamily sirakawa@faculty.chiba-u.jp}
\vspace{2ex}

{\sc Hiroshi Watanabe}\footnotemark[4]\\
Division of Mathematical Sciences, Faculty of Science and Technology, Oita University\\
700 Dannoharu, Oita, 870--1192, Japan\\
{\ttfamily hwatanabe@oita-u.ac.jp}
\end{center}
\vspace{10ex}

\noindent
{\bf Abstract.}
In this paper we study the Dirichlet problem for the Kobayashi--Warren--Carter system. This system of parabolic PDE's models the grain boundary motion in a polycrystal with a prescribed orientation at the boundary of the domain. We obtain global existence in time of energy-dissipative solutions. The regularity of the solutions as well as the energy-dissipation property permit us to derive the steady-state problem as the asymptotic in time limit of the system. We finally study the $\omega$-limit set of the solutions; we completely characterize it in the one dimensional case, showing, in particular that orientations in the $\omega$-limit set belong to thee space of SBV functions. In the two dimensional case, we give sufficient conditions for existence of radial symmetric piecewise constants solutions.
%\vspace{5cm}

%%%%%%%%%%%%%%%%%%%%%%%%%%%%%%%%%%%%%%
\footnotetext[0]{%\hspace*{-7.4mm}
%%%%%%%%%%%%%%%%%%%%%%%%%%%%%%%%%%%%%%

$\empty^\dag$\,This author is supported by the Spanish MCIU and FEDER project PGC2018-094775-B-I00.

$\empty^\ddag$\,This author is supported by Grant-in-Aid No. 16K05224 and 20K03672, JSPS.

$\empty^\S$\,This author is supported by Grant-in-Aid No. 20K03696 and 21K03312, JSPS.
}
\newpage

\section*{Introduction}
\ \ \vspace{-4ex}

Let $ 0 < T < \infty $ and $ N \in \N $ be fixed constants. Let $ \Omega \subset \R^N $ be a bounded domain with a smooth boundary $ \Gamma $ if $ N > 1 $. Let $ n_\Gamma $ be the unit outer normal on $ \Gamma $. Let $ Q := (0, T) \times \Omega $ be the product set of the time-interval $ (0, T) $ and the spatial domain $ \Omega $, and let $ \Sigma := (0, T) \times \Gamma $.

In this paper, we consider a coupled system, denoted by (S). Our system consists of the following two initial-boundary problems.
\bigskip

(S):
\begin{equation}\label{1st.eq}
\left\{ ~ \parbox{9cm}{
$ \eta_t -{\mit \Delta} \eta +g(\eta) +\alpha'(\eta)|D \theta| = 0 $, \ $ (t, x) \in Q $,
\\[1ex]
    $ \ds \nabla \eta \cdot n_\Gamma +\alpha'(\eta) |\theta -\gamma| = 0 $, \ $ (t, x) \in \Sigma $,
\\[1ex]
$ \eta(0, x) = \eta_0(x) $, \ $ x \in \Omega $;
} \right.
\end{equation}
\begin{equation}\label{2nd.eq}
\left\{ ~ \parbox{9cm}{
$ \ds \alpha_0(\eta) \theta_t -{\rm div} \left( \alpha(\eta) \frac{D \theta}{|D \theta|} \right) = 0 $, \ $ (t, x) \in Q $,
\\[1ex]
    $ \theta = \gamma $, \ $ (t, x) \in \Sigma $,
\\[1ex]
$ \theta(0, x) = \theta_0(x) $, \ $ x \in \Omega $.
} \right.
\end{equation}

The system (S) is called ``Kobayashi--Warren--Carter system'', and as the name suggests, it is based on a mathematical model proposed by Kobayashi et al \cite{kobayashi2000continuum,MR1794359} on planar grain boundary motion. According to the original works \cite{kobayashi2000continuum,MR1794359}, the unknowns $ \eta = \eta(t, x) $ and $ \theta = \theta(t, x) $ represent, respectively, the ``orientation order'' and the ``orientation angle'' in a polycrystal. Based on this, the initial-boundary problems (\ref{1st.eq}) and (\ref{2nd.eq}) are derived as gradient flows of the following {governing energy $ \mathcal{F}_{\gamma} = \mathcal{F}_{\gamma}(\eta, \theta) $}, called ``free-energy'':
\begin{align}\label{freeEgy}
    [\eta, \theta] \in H^1(\Omega) & \times BV(\Omega) \mapsto {\mathcal{F}_{\gamma}(\eta, \theta)} := \ds \frac{1}{2} \int_\Omega |\nabla \eta|^2 \, dx +\int_\Omega G(\eta) \, dx
    \nonumber
    \\
    & +\int_\Omega \alpha(\eta) |D \theta| +\int_{\Gamma} \alpha(\eta) |\theta -{\gamma}| \, d \Gamma.
\end{align}
In the context, $ 0 < \alpha \in C^2(\mathbb{R}) $ and $ 0 < \alpha_0 \in W_{\rm loc}^{1, \infty}(\mathbb{R}) $ are given functions corresponding to the mobilities of grain boundaries, and the superscript ``\,$'$\,'' denotes the differential of a function of one-variable. $ g \in C^1(\mathbb{R}) $ is a given function corresponding a perturbation for the orientation order, and $ G \in C^2(\mathbb{R}) $ is a nonnegative primitive of $ g $. {$ \gamma \in H^1(\Gamma) \cap L^\infty(\Gamma) $} is a given function corresponding to the Dirichlet boundary data for $ \theta $. Besides, $ d \Gamma $ denotes the area element on $ \Gamma $.

For parabolic systems kindred to (S), various mathematical methods have been developed in the literature. For instance, when both boundary conditions for $ \eta $ and $ \theta $ are of {zero-Neumann type}, the study of the corresponding system was treated in the recent works \cite{MR3268865,MR3670006,MR3082861,MR3462536,MR3038131,MR3362773,MR3203495,MR3238848}. Also, when $ \gamma \equiv 0 $, Ito--Kenmochi--Yamazaki \cite{MR2469586,MR2548486,MR2836555} and Kenmochi--Yamazaki \cite{MR2668289} studied relaxed versions of (S) that include an additional term $ -{\mit \Delta} \theta $ of parabolic regularization in the initial-boundary problem (\ref{2nd.eq}). {However, in the original (non-relaxed) problem \eqref{2nd.eq}, we first need to address a problem of ambiguity in the Dirichlet boundary condition. Actually, in rigorous mathematics, the condition $ \theta = \gamma $ on $ \Sigma $, as in \eqref{2nd.eq}, does not make sense, because the singular diffusion $ -\mathrm{div} \bigl( \alpha(\eta) \frac{D \theta}{|D \theta|} \bigr) $ generally allows spatial discontinuity for the solution $ \theta $, including the jumps between the trace of $ \theta $ and the boundary data $ \gamma $. For the problem of ambiguity, we adopt the method to define a weak formulation of the Dirichlet boundary condition in the singular problem \eqref{2nd.eq}, by means of the subdifferential of the $ L^1 $-integral term $ \int_\Gamma \alpha(\eta) |\theta -\gamma| \, d \Gamma $, as in \eqref{freeEgy}. The idea of the weak formulation is based on the general theory, established in \cite{MR1814993,MR3288271,MR2139257}.

In the case that $ \gamma $ is not constant, the nonhomogeneous boundary data will bring various equilibrium patterns of grain boundaries, as solutions to the steady-state problem for our system (S). Therefore, it could be expected that the structural observations for steady-state solutions would provide fruitful information for more advanced researches, such as the stability analysis of polycrystalline structure, the determination of possible (ideal) target profiles in the control problems of grain boundary motions, and so on.
\medskip

In view of these, we set the objective of this paper to deal with the following issues:
\begin{description}
        \hypertarget{sharp1}{}
    \item[{\boldmath issue\,$\sharp1$)}]the mathematical solvability of the system (S) with nonhomogeneous Dirichlet boundary data;
        \hypertarget{sharp2}{}
    \item[{\boldmath issue\,$\sharp2$)}]the derivation of the steady-state problem as an asymptotic governing system for the large-time behavior;
        \hypertarget{sharp3}{}
    \item[{\boldmath issue\,$\sharp3$)}]the structural observation for steady-state solutions.
\end{description}
Besides, we largely divide the content of this paper into two parts:

The former part deals with issues\,\hyperlink{sharp1}{$\sharp1$}) and \,\hyperlink{sharp2}{$\sharp2$}), and the conclusions are the main results of this paper, in short:
\hypertarget{mainTh}{}
\begin{description}
        \hypertarget{mainTh01}{}
    \item[(Main Theorem 1)]the existence of time-global solutions to the system (S) with the dissipation in time of the free-energy;
        \hypertarget{mainTh02}{}
    \item[(Main Theorem 2)]the semicontinuous association between the $ \omega $-limit set of the solution, as time tending to $ \infty $, and the steady-state problem for our system (S);
\end{description}
and some auxiliary results, in short:
\hypertarget{thms}{}
\begin{description}
    \item[(Theorem 1)]the solvability and energy-dissipation property for the approximating systems of time--discretization  type;
    \item[(Theorem 2)]the energy-inequality for a priori estimates of approximating solutions.
\end{description}
The precise statement of the above results is contained in Sections \ref{sec:main} (\hyperlink{mainTh}{Main Theorems}) and \ref{sec:approx} (\hyperlink{thms}{Theorems}), and proved in Sections 5 and 6, respectively, on the basis of the preliminaries prepared in Section 1, and some auxiliary Lemmas obtained in Section 4.
\medskip

The latter part is assigned to the last Section 7, and is devoted to issue\,\hyperlink{sharp3}{$\sharp3$}). Under the following concrete setting:
\begin{equation}\label{concrete-setting}
    \begin{cases}
        \ds \alpha_0(\eta) = \alpha(\eta) = \frac{\eta^2}{2} +\delta_0, \mbox{ with a constant $ \delta_0 \geq 0 $,}
        \\
        \ds g(\eta) = \eta -1, \mbox{ with } G(\eta) = \frac{1}{2}(\eta -1)^2.
    \end{cases}
\end{equation}
the results of structural observations will be reported as conclusions of this study. They concern with:
\begin{description}
        \hypertarget{sharp3-a}{}
    \item[{\boldmath ~~$\sharp3$-a)}]the exact profiles and SBV-regularity of one-dimensional steady-state solutions;
        \hypertarget{sharp3-b}{}
    \item[{\boldmath ~~$\sharp3$-b)}]the sufficient conditions for the existence of radially symmetric two-dimensional steady-state solutions;
\end{description}

}

\section{Preliminaries}
\ \ \vspace{-4ex}

We begin with some notations used throughout this paper.
\medskip

\noindent
\underline{\textit{Basis notations.}}
For arbitrary $ a_0, b_0 \in [-\infty, \infty] $, we define
\begin{equation*}
a_0 \vee b_0 := \max \{ a_0, b_0\} \mbox{ \ and \ } a_0 \wedge b_0 := \min \{ a_0, b_0\},
\end{equation*}
and for arbitrary $ -\infty \leq a \leq b \leq \infty $, we define the truncation function (operator) $ \mathcal{T}_a^b : \R \rightarrow [a, b] $ by letting
\begin{equation*}
r \in \R \mapsto \mathcal{T}_a^b r := a \vee (b \wedge r) \in [a, b],
\end{equation*}
and in particular, we set
\begin{equation*}
[{}\cdot{}]^+ := \mathcal{T}_0^\infty  \mbox{ and } [{}\cdot{}]^- := -\mathcal{T}_{-\infty}^0, \mbox{ on $ \R $.}
\end{equation*}

We denote by $ |x| $ and $ x \cdot y $ the Euclidean norm of $ x \in \mathbb{R}^N $ and the standard scalar product of $ x, y \in \R^{N} $, respectively, as usual, i.e.,
\begin{equation*}
\begin{array}{c}
| x | := \sqrt{x_1^2 +\cdots +x_N^2} \mbox{ \ and \ } x \cdot y  := x_1 y_1 +\cdots +x_N y_N
\\[1ex]
\mbox{ for all $ x = [x_1, \ldots, x_N], ~ y = [y_1, \ldots, y_N] \in \mathbb{R}^N $.}
\end{array}
\end{equation*}
\medskip

\noindent
\underline{\textit{Abstract notations. (cf. \cite[Chapter II]{MR0348562})}}
For an abstract Banach space $X$, we denote by $|\cdot|_{X}$ the norm of $X$, and by $ \langle \cdot, \cdot \rangle_X $ the duality pairing between $ X $ and its dual $ X^* $. In particular, when $X$ is a Hilbert space, we denote by $(\,\cdot\,,\cdot\,)_{X}$ the inner product of $X$.
Besides, for two Banach spaces $ X $ and $ Y $, let $  \mathscr{L}(X; Y)$ be the space of bounded linear operators from $ X $ into $ Y $.
\medskip

For any proper functional $ \Psi : X \rightarrow(-\infty, \infty] $ on a Banach space $ X $, we denote by $ D(\Psi) $ the effective domain of $ \Psi $.
\medskip

For any proper lower semi-continuous (l.s.c., in short) and convex function $\Phi$ defined on a Hilbert space $H$, we denote by $\partial \Phi$ the subdifferential of $\Phi$. The subdifferential $\partial \Phi$ corresponds to a weak differential of $\Phi$, and it is a maximal monotone graph in the product space $ H^2 $ $ (= H \times H) $. More precisely, for each $v_{0} \in H$, the value $\partial \Phi(v_{0})$ of the subdifferential at $v_{0}$ is defined as a set of all elements $v_{0}^{*} \in H$ which satisfy the following variational inequality:
\begin{equation*}
(v_{0}^{*}, v-v_{0})_{H} \le \Phi(v) - \Phi(v_{0})\ \ \mbox{for any}\ v \in D(\Phi).
\end{equation*}
The set $D(\partial \Phi) := \{ z \in H \ |\ \partial\Phi(z) \neq \emptyset\}$ is called the domain of $\partial\Phi$. We often use the notation ``$ [v_{0},v_{0}^{*}] \in \partial\Phi_{0}$ in $ H^2 $\,'', to mean that $`` v_{0}^{*} \in \partial\Phi(v_{0})$ in $H$ for $v_{0} \in D(\partial\Phi)"$, by identifying the operator $\partial\Phi$ with its graph in $H^2$.
\bigskip

\noindent
\underline{\textit{Notion of  $ {\mit \Gamma} $-convergence. (cf. \cite{MR1201152})}}
Let $ X $ be a reflexive Banach space. We say that a net $ \{ \Xi_\nu  \}_{\nu>0}$ of proper functionals $ \Xi_\nu  : X \rightarrow(-\infty, \infty] $, $ \nu > 0 $ (resp. a sequence $ \{ \Xi_n  \}_{n = 1}^\infty $ of proper functionals $ \Xi_n  : X \rightarrow(-\infty, \infty] $, $ n \in \N $), $\Gamma$-converges to a proper functional $ \Xi_0 : X \rightarrow(-\infty, \infty] $ as
$ \nu \downarrow 0 $ (resp. as $n\to+\infty$), iff. the following two conditions hold.
\hypertarget{Mosco}{}
\begin{description}
\item[\textmd{\it 1$^\circ$ Lower bound:}]$ \displaystyle \liminf_{\nu \downarrow 0} \Xi_\nu(v_\nu) \geq \Xi_0(v_0) $, if $ v_0 \in X $, $ \{ v_\nu \}_{\nu > 0} \subset X $, and $ v_\nu\to v_0 $ in $ X $ as $ \nu \downarrow 0 $ (resp. replacing ``$\nu$'' by ``$n$'', and ``$\nu \downarrow 0$'' by ``$n\to \infty$'');
\item[\textmd{\it 2${}^{\,\circ}$ Optimality:}]for any $ v_0 \in D(\Xi_0) $, there exists a net $ \{ v_\nu \}_{\nu > 0} \subset X  $ (resp. a sequence $\{ v_n \}_{n=1}^\infty \subset X$) , such that $ v_\nu \to v_0 $ and $ \Xi_\nu(v_\nu) \to \Xi_0(v_0) $, $ \nu \downarrow 0 $
(resp. replacing ``$\nu$'' by ``$n$'', and ``$\nu \downarrow 0$'' by ``$n\to \infty$'').
\end{description}

\noindent
\underline{\textit{Notations of basic elliptic operators.}}
Let $F\ :\ H^{1}(\Omega) \longrightarrow H^{1}(\Omega)^{\ast}$ be the duality mapping, defined as
\begin{equation}
\langle Fw, v \rangle_{H^1(\Omega)} := (w,v)_{H^{1}(\Omega)} = (w,v)_{L^{2}(\Omega)} + (\nabla w, \nabla v)_{L^{2}(\Omega)^{N}}, \mbox{ for all $v,w \in H^{1}(\Omega)$.}
\end{equation}
As is well-known, $Fu = - \laplace_{\rm N}u + u$ in $L^{2}(\Omega)$, if $u \in H^{1}(\Omega)$ belongs to the domain
\begin{equation*}
D_{\rm N} := \{ v \in H^{2}(\Omega)\ |\ \nabla v \cdot \nu_{\partial\Omega} = 0 \mbox{ in } L^{2}(\partial\Omega)\},
\end{equation*}
of the Laplacian operator
\begin{equation*}
\laplace_{\rm N}\ : \ u \in D_{\rm N} \subset L^{2}(\Omega) \mapsto \laplace u \in L^{2}(\Omega),
\end{equation*}
subject to the Neumann-zero boundary condition.

\noindent
\underline{\textit{Notations in basic measure theory. (cf. \cite{MR1857292, MR3288271})}} We denote by $\mathcal{L}^{N}$ the $ N $-dimensional Le--besgue measure, and we denote by $ \mathcal{H}^{N-1} $ the $ (N-1) $-dimensional Hausdorff measure.  In particular, the measure theoretical phrases, such as ``a.e.'', ``$dt$'' and ``$dx$'', and so on, are all with respect to the Lebesgue measure in each corresponding dimension. Also, in the observations on a Lipschitz surface $ S $, the phrase ``a.e.'' is with respect to the Hausdorff measure in each corresponding Hausdorff dimension.

For a measurable function $ u : B \rightarrow[-\infty, \infty] $ on a Borel set $ B \subset \R^N $, we denote by $ [u]^+ $ and $ [u]^- $ the positive part $ \mathcal{T}_0^\infty \circ u : B \rightarrow[0, \infty] $ and the negative part $ -\mathcal{T}_{-\infty}^0 \circ u : B \rightarrow[0, \infty] $, respectively.

Let $ U \subset \mathbb{R}^N $ be any open set. We denote by $ \mathscr{M}(U) $ (resp. $ \mathscr{M}_{\rm loc}(U) $) the space of all finite Radon measures (resp. the space of all Radon measures) on $ U $. Recall that the space $ \mathscr{M}(U) $ (resp. $ \mathscr{M}_{\rm loc}(U) $) is the dual of the Banach space $ C_0(U) $ (resp. dual of the locally convex space $ C_{\rm c}(U) $), for any open set $ U \subset \mathbb{R}^N $.

We set:
\begin{equation*}
\bB :=  \left\{ \begin{array}{l|l} \xi \in \R^N & |\xi| < 1 \end{array} \right\} \mbox{ and }
\bS^{N -1} := \partial \bB,
\end{equation*}
and for any $ \nu  \in \bS^{N -1} $, we set:
\begin{equation*}
\bB(\nu) :=  \left\{ \begin{array}{l|l} \xi \in \bB & \xi \cdot \nu > 0 \end{array} \right\}.
\end{equation*}

For a function $ u \in L_{\rm loc}^1(U) $ and a point $ x_0 \in U $, we denote by
\begin{equation*}
\mbox{ap\,-\hspace{-0.5ex}}\lim_{x \to x_0} u(x) = [\tilde{u}](x_0),
\end{equation*}
iff. there exists a value $ [\tilde{u}](x_0) \in \R $ such that:
\begin{equation*}
\lim_{r \downarrow 0} \frac{\mathcal{L}^N \left( \left\{ \begin{array}{l|l}
x \in r\bB +x_0 & |u(x) -[\tilde{u}](x_0)| > \sigma
\end{array}  \right\} \right)}{\mathcal{L}^d\left( r \bB  +x_0 \right)} = 0, \mbox{ \ for any $ \sigma > 0 $.}
\end{equation*}
The above $ [\tilde{u}](x_0) \in \R^N $ is called the {\em approximate limit} of the function $ u \in L_{\rm loc}^1(U) $ at the point $ x_0 \in U $, and the real value $ [\tilde{u}](x_0) $ exists, iff. for any $ \nu \in \bS^{N -1} $, a real value $ [\tilde{u}]_\nu(x_0)   \in \R $ such that:
\begin{equation*}
\lim_{r \downarrow 0} \frac{\mathcal{L}^N \left( \left\{ \begin{array}{l|l}
x \in r\bB(\nu) +x_0 & |u(x) -[\tilde{u}]_\nu(x_0)| > \sigma
\end{array}  \right\} \right)}{\mathcal{L}^N\left( r \bB(\nu)  +x_0 \right)} = 0, \mbox{ for any $ \sigma > 0 $,}
\end{equation*}
exists, and it coincides with the value $ [\tilde{u}]_{-\nu}(x_0) \in \R $, i.e.
\begin{equation*}
\lim_{r \downarrow 0} \frac{\mathcal{L}^N \left( \left\{ \begin{array}{l|l}
x \in r\bB(-\nu) +x_0 & |u(x) -[\tilde{u}]_{\nu}(x_0)| > \sigma
\end{array}  \right\} \right)}{\mathcal{L}^d\left( r \bB(-\nu)  +x_0 \right)} = 0, \mbox{ for any $ \sigma > 0 $.}
\end{equation*}
As is well-known (cf. \cite[Section 10.3]{MR3288271}), any function $ u \in L_{\rm loc}^1(U) $ admits an approximate limit $ [\tilde{u}](x) $ at a.e. $ x \in U $, and in this regard, the function $ x \in U \mapsto [\tilde{u}](x) $ is called the {\em continuous representative} of $ u $. The measurable set $ J_u \subset U $ defined as:
\begin{equation}\label{J_u}
J_u := \left\{ \begin{array}{l|l}
{x} \in U & \parbox{10cm}{there exists $ \nu_{u}^{(x)} \in \bS^{N -1} $ such that $ [\tilde{u}]_{\nu_u^{(x)}}({x}) > [\tilde{u}]_{-\nu_u^{(x)}}({x}) $}
\end{array} \right\},
\end{equation}
is called the {\em jump set} of $ u $. Additionally, the unit vector $ \nu_u^{(x)} \in \bS^{N -1} $ as in \eqref{J_u} is uniquely determined for each $ x \in J_u $, and the mapping $ x \in J_u \mapsto \nu_u(x) := \nu_u^{(x)} \in \bS^{N -1} $ forms a measurable function. Similarly, we denote by
\begin{equation*}
\mbox{ap\,-}\limsup_{x \to x_0} u(x) = [\widetilde{u_+}](x_0) ~ \bigl( \mbox{resp. } \mbox{ap\,-}\liminf_{x \to x_0} u(x) = [\widetilde{u_-}](x_0) \bigr),
\end{equation*}
iff. there exists a real value $ [\widetilde{u_+}](x_0) \in \R $ (resp. $ [\widetilde{u_-}](x_0) \in \R $) satisfying:
\begin{align*}
& [\widetilde{u_+}](x_0) = \inf \left\{ \begin{array}{l|l}
\tilde{\sigma} \in \R & \displaystyle
\lim_{r \downarrow 0} \frac{\mathcal{L}^N \left( \left\{ \begin{array}{l|l}
x \in r \bB +x_0 & u(x) > \tilde{\sigma}
\end{array}  \right\} \right)}{\mathcal{L}^N \left( r \bB +x_0 \right)} = 0
\end{array} \right\},
\\
\left( \rule{0pt}{18pt} \mbox{resp. } \right. & [\widetilde{u_-}](x_0) = \sup \left\{ \begin{array}{l|l}
\tilde{\sigma} \in \R & \displaystyle
\lim_{r \downarrow 0} \frac{\mathcal{L}^N \left( \left\{ \begin{array}{l|l}
x \in r \bB +x_0 & u(x) < \tilde{\sigma}
\end{array}  \right\} \right)}{\mathcal{L}^N\left( r \bB +x_0 \right)} = 0
\end{array} \right\} \left. \rule{0pt}{18pt} \right).
\end{align*}
Besides, we set
$$
S_u := \left\{ \begin{array}{l|l}
x \in U &
[\widetilde{u_+}](x) \ne [\widetilde{u_-}](x), ~ \mbox{i.e. $ [\widetilde{u_+}](x) > [\widetilde{u_-}](x) $}
\end{array} \right\}.
$$

\noindent
\underline{\textit{Notations in BV-theory. (cf. \cite{MR1857292, MR3288271, MR3409135, MR0775682})}} Let $U\subset \R^N$ be an open set. A function $ u \in L^1(U) $ (resp. $ u \in L_{\rm loc}^1(U) $)  is called a function of bounded variation, or a BV-function, (resp. a function of locally bounded variation or a BV$\empty_{\rm loc}$-function) on $ U $, iff. its distributional differential $ D u $ is a finite Radon measure on $ U $ (resp. a Radon measure on $ U $), namely $ D u \in \mathscr{M}(U) $ (resp. $ D u \in \mathscr{M}_{\rm loc}(U) $).
We denote by $ BV(U) $ (resp. $ BV_{\rm loc}(U) $) the space of all BV-functions (resp. all BV$\empty_{\rm loc}$-functions) on $ U $. For any $ u \in BV(U) $, the Radon measure $ D u $ is called the variation measure of $ u $, and its  total variation $ |Du| $ is called the total variation measure of $ u $. Additionally, the value $|Du|(U)$, for any $u \in BV(U)$, can be calculated as follows:
\begin{equation*}
|Du|(U) = \sup \left\{ \begin{array}{l|l}
\ds \int_{U} u \ {\rm div} \,\varphi \, dy & \varphi \in C_{\rm c}^{1}(U)^N \ \ \mbox{and}\ \ |\varphi| \le 1\ \mbox{on}\ U
\end{array}
\right\}.
\end{equation*}
The space $BV(U)$ is a Banach space, endowed with the following norm:
\begin{equation*}
|u|_{BV(U)} := |u|_{L^{1}(U)} + |D u|(U),\ \ \mbox{for any}\ u\in BV(U).
\end{equation*}
Also, $ BV(U) $ is a complete metric space, endowed with the following distance:
$$
[u, v] \in BV(U)^2 \mapsto |u -v|_{L^1(U)} +\left| \int_U |Du| -\int_U |Dv| \right|.
$$
The topology provided by this distance is called the \em strict topology \em of $ BV(U) $.

For any $ u \in BV(U) $, it is known that (cf. \cite[Sections 10.3 and 10.4]{MR3288271}), the sets $ S_u $ and $ J_u $ are $ \mathcal{H}^{N -1} $-measurable, $ \mathcal{H}^{N -1}(S_u \cup J_u) < \infty $, $ \mathcal{H}^{N -1}((S_u \cup J_u) \setminus (S_u \cap J_u)) = 0 $, and
\begin{equation*}
[\widetilde{u_+}](x) = [\tilde{u}]_{\nu_u(x)}(x) \mbox{ and } [\widetilde{u_-}](x) = [\tilde{u}]_{-\nu_u(x)}(x), \mbox{ for $ \mathcal{H}^{d -1} $-a.e. $ x \in S_u \cap J_u $.}
\end{equation*}

Now, by Radon-Nikod\'{y}m's theorem \cite[Theorem 1.28]{MR1857292}, the measure $ Du $ for $ u \in BV(U) $ is decomposed in the absolutely continuous part $ D^a u $ for $ \mathcal{L}^d $ and the singular part $ D^s u $, i.e.
$$
Du = D^{a}u +D^{s}u \mbox{ \ in $ \mathscr{M}(U) $.}
$$
Furthermore, (cf. \cite{MR3288271,MR1857292}), this decomposition is precisely expressed as follows:
$$
D^{a} u = \nabla u \, \mathcal{L}^N \, \lfloor \, (U \setminus S_u) \mbox{ and } D^s u = ([\widetilde{u_+}] -[\widetilde{u_-}]) \nu_u \mathcal{H}^{N -1} \, \lfloor \, S_u +D^{s} u \, \lfloor \, (U \setminus S_u)
$$
In this context, $ \nabla u \in L^1(U) $ denotes the Radon-Nikod\'{y}m density of $ D^a u  $ for $ \mathcal{L}^d $. The part $ ([\widetilde{u_+}] -[\widetilde{u_-}]) \nu_u \mathcal{H}^{d -1} \, \lfloor \, S_u $ is called the \em jump part \em of $ Du $, and it provides an exact expression of the variation at the discontinuities of $ u $. The part $ D^s u \, \lfloor \, (U \setminus S_u) $ is called the \em Cantor part \em of $ Du $ and it is denoted by $D^c u$. We denote by SBV the space of special functions of bounded variation; i.e. those BV functions such that $D^c u=0$. We note that a function defined as
$$
x \in U \mapsto [u]^*(x) := \left\{ \begin{array}{ll}
[\tilde{u}](x), & \mbox{if $ x \in U \setminus S_u $,}
\\[1ex]
\ds \frac{[\widetilde{u_+}](x) +[\widetilde{u_-}](x)}{2}, & \mbox{if $ x \in J_u $,}
\\[2ex]
0, & \mbox{if $ x \in S_u \setminus J_u $,}
\end{array} \right.
$$
provides a \em precise representative \em of $ u \in BV(U) $, $ \mathcal{H}^{N -1} $-a.e. on $ U $.

In particular, if $ U $ is bounded and the boundary  $\partial U$ is Lipschitz, then the space $BV(U)$ is continuously embedded into $L^{N/(N -1)}(U)$ and compactly embedded into $L^{q}(U)$ for any $1 \le q < N/(N-1)$ (cf. \cite[Corollary 3.49]{MR1857292} or \cite[Theorems 10.1.3 and 10.1.4]{MR3288271}). Also, there exists a bounded linear operator $ \mathfrak{tr}_{{\partial U}} : BV(U) \mapsto L^1(\partial U) $, called the \em trace operator, \em such that $ \mathfrak{tr}_{{\partial U}} \varphi = \varphi|_{\partial U} $ on $ \partial U $ for any $ \varphi \in C^1(\overline{U}) $. The value of trace $ \mathfrak{tr}_{{\partial U}} u \in L^1(\partial U) $ is often abbreviated by $ u $, by identifying $ \mathfrak{tr}_{{\partial U}} u $ as the extension of $ u \in BV(U) $ onto $ \partial U $.
Additionally, if $1 \le r < \infty$, then the space $C^{\infty}(\overline{U})$ is dense in $BV(U) \cap L^{r}(U)$ for the {\em intermediate convergence} (cf. \cite[Definition 10.1.3. and Theorem 10.1.2]{MR3288271}), i.e. for any $u \in BV(U) \cap L^{r}(U)$, there exists a sequence $\{u_{n} \}_{n = 1}^\infty \subset C^{\infty}(\overline{U})$ such that $u_{n} \to u$ in $L^{r}(U)$ and $\int_{U}|\nabla u_{n}|dx \to |Du|(U)$ as $n \to \infty$.

\begin{rem}\label{Rem.trace}
\textbf{(cf. \cite[Theorem 3.88]{MR1857292})} Let $ U \subset \R^N $ be a bounded open set with Lipschitz boundary $ \partial U $. Then it is known that:
\begin{equation*}
\int_U u \, {\rm div} \, \psi \, dx = \int_U \psi \cdot Du -\int_{\partial U} u \, (\psi \cdot n_{\Gamma}) \, d \mathcal{H}^{N -1}, \mbox{ for any $ \psi \in C^1(\R^N)^N $.}
\end{equation*}
Moreover, the trace operator $ \mathfrak{tr}_{{\partial U}} : BV(U) \mapsto L^1(\partial U) $ is continuous with respect to the strict topology of $ BV(U) $. Namely,  $ \mathfrak{tr}_{{\partial U}} u_n \to \mathfrak{tr}_{{\partial U}} u $ in $ L^1(\partial U) $ as $ n \to \infty $, if $ u \in BV(U) $, $ \{ u_n \}_{n = 1}^\infty \subset BV(U) $ and $ u_n \to u $ in the strict topology of $ BV(U) $ as $ n \to \infty $.
\end{rem}

\noindent
\underline{\textit{Extensions of functions. (cf. \cite{MR1857292,MR3288271})}} Let $ \mu $ be a positive measure on $ \R^N $, and let $ B \subset \R^N $ be a $ \mu $-measurable Borel set. For a $ \mu $-measurable function $ u : B \rightarrow \R $, we denote by $ [u]^{\rm ex} $ an extension of $ u $ over $ \R^N $, i.e. $ [u]^{\rm ex} : \R^d \rightarrow \R $ is a measurable function such that $ [u]^{\rm ex} = u $ $ \mu $-a.e. in $ B $. In general, the choices of extensions are not necessarily unique.

\begin{rem}\label{Rem.ext} Let $ U \subset \R^N $ be an open set, with compact $C^1$-boundary (in case $N>1$) $\partial U$. Then,
\begin{description}
    \hypertarget{Fact0}{}
\item[(Fact\,0)](cf. \cite[Proposition 3.21]{MR1857292})There exists a bounded linear operator $ \mathfrak{ex}_{U} : BV(U) \rightarrow BV(\R^N) $, such that:
\begin{itemize}
\item[--] $ \mathfrak{ex}_{U} $ maps any function $ u \in BV(U) $ to an extension $ [u]^{\rm ex} \in BV(\R^N) $;
\item[--] for any $ 1 < q < \infty $,  $ \mathfrak{ex}_{U} ({W^{1, q}(U)}) \subset W^{1, q}(\R^N) $, and the restriction $ \mathfrak{ex}_{U}  |_{W^{1, q}(U)} : W^{1, q}(U) \rightarrow W^{1, q}(\R^N) $ forms a bounded and linear operator.
\end{itemize}
\hypertarget{Fact1}{}
\item[(Fact\,1)](cf. \cite[{[1.II] Teorema}]{MR0102739}) There exists a bounded linear operator $ \mathfrak{ex}_{{\partial U}}^1 : L^{1}(\partial U) \rightarrow W^{1, 1}(\R^N) $, which maps any function $ \varrho \in L^1(\partial U) $ to an extension $ [\varrho]^{\rm ex} \in W^{1, 1}(\R^N) \subset BV(\R^N) $.
\end{description}
\end{rem}

{
\begin{rem}[Harmonic extension]\label{Rem.ex_hm}
    Under the same notations as in Remark \ref{Rem.ext}, we denote by $ \mathfrak{hm}_{\partial U} $ the operator of harmonic extension $ \mathfrak{hm}_{{\partial U}} : H^{\frac{1}{2}}(\partial U) \to H^1({U}) $, which maps any $ \varrho \in H^{\frac{1}{2}}(\partial U) $ to a (unique) minimizer $ [\varrho]^{\rm hm} \in H^1(U) $ of the following proper l.s.c. and convex function on $ L^2(U) $:
\begin{equation*}
    z \in {L^2(U)} \mapsto \left\{ \begin{array}{ll}
\multicolumn{2}{l}{\ds \frac{1}{2} \int_U |\nabla z|^2 \, dx, \mbox{ if $ z \in H^1(U) $ and $ z = \varrho $ in $ H^{\frac{1}{2}}(\partial U) $,}}
\\[2ex]
\infty, & \mbox{otherwise.}
\end{array} \right.
\end{equation*}
    Indeed, by the general theories of linear PDEs (cf. \cite[Chapter 3]{MR0241822}, \cite[Chapter 8]{MR0244627} and \cite[Theorem 8.3 in Chapter 2]{MR0350177}), we can check that $ \mathfrak{hm}_{{\partial U}} H^{\frac{1}{2}}(\partial U) \subset H_\mathrm{loc}^2(U) $, and:
\begin{equation}\label{harmonic01}
    (\mathit{\Delta} [\varrho]^{\rm hm}, \varphi)_{L^2(U)} = (\nabla [\varrho]^{\rm hm}, \nabla \varphi)_{L^2(U)^N} = 0, \mbox{ for all $ \varrho \in H^{\frac{1}{2}}(\Gamma) $ and {$ \varphi \in C_\mathrm{c}^1(U) $},}
\end{equation}
    i.e. the extension $ [\varrho]^{\rm hm} $ solves the harmonic equation $ \mathit{\Delta} [\varrho]^{\rm hm} = 0 $ on $ U $, having the Dirichlet boundary data $ \varrho \in {H^{\frac{1}{2}}(\partial U)} $. Additionally, the  variational form \eqref{harmonic01} implies that:
\begin{description}
        \hypertarget{ex0}{}
    \item[\textmd{(ex.0)}]the operator $ \mathfrak{hm}_\Gamma : H^{\frac{1}{2}}(\partial U) \to H^1(U) $ is a bounded linear isomorphism, and in particular, there exists a positive constant $ C_{U}^{\rm hm} $ satisfying
\begin{equation}\label{hm_const}
    |[\varrho]^{\rm hm}|_{H^1(U)} \leq C_{U}^{\rm hm} |\varrho|_{H^{\frac{1}{2}}(\partial U)}, \mbox{ for any $ \varrho \in H^{\frac{1}{2}}(\Gamma) $;}
\end{equation}
        \hypertarget{ex1}{}
\item[\textmd{(ex.1)}]if $ r \in \R $, and $ r \leq \varrho \in H^{\frac{1}{2}}(\partial U) $ (resp. $ r \geq \varrho  \in H^{\frac{1}{2}}(\partial U) $), then $ r \leq [\varrho]^{\rm hm} \in H^1(U) $ (resp. $ r \geq [\varrho]^{\rm hm} \in H^1(U) $).
\end{description}
Moreover, by using the composition $ \mathfrak{hm}_{{\partial (\overline{U}^{\rm C})}} \circ \mathfrak{tr}_{\partial U} : H^{1}(U) \to H^1(\overline{U}^{\rm C}) $, we may suppose that:
\begin{description}
        \hypertarget{ex2}{}
    \item[\textmd{(ex.2)}]if $ r \in \R $, and $ r \leq u \in H^{1}(U) $ (resp. $ \sigma \geq r \in  H^{1}(U) $), then $ r \leq [u]^{\rm ex} \in H^{\frac{1}{2}}(\R^N) $ (resp. $ r \geq [u]^{\rm ex} \in H^{\frac{1}{2}}(\R^N) $);
\end{description}
\end{rem}
}

\noindent
\underline{\textit{Specific notations. (cf. \cite{MR2306643,MR2139257,MR3268865})}}
Throughout this paper, $ \Omega \subset \R^N $ is a bounded domain with $ \Gamma $ being its boundary. We assume sufficient smoothness {(at least $ C^1 $-regularity)} for $ \Gamma $ when $ N > 1 $. Also, we fix a large open ball $ \bB_\Omega \subset \R^N $ such that $ \bB_\Omega \supset \overline{\Omega}$.

Next, for any $ u \in BV(\Omega) \cap L^2(\Omega) $ and any $ \gamma \in L^1(\Gamma) $, we define a function $ [u]_\gamma^{\rm ex} \in BV(\R^N) \cap W^{1, 1}(\R^N \setminus \overline{\Omega}) $, by letting
\begin{equation*}
x \in \R^N \mapsto [u]_\gamma^{\rm ex}(x) := \left\{ \begin{array}{l}
[u]^*(x), \mbox{ if $ x \in \Omega $,}
\\[1ex]
[\gamma]^{\rm ex}(x), \mbox{ if $ x \in \R^N \setminus \overline{\Omega} $,}
\end{array} \right.
\end{equation*}
For any $ \beta \in H^1(\Omega) \cap L^\infty(\Omega) $, any $ \gamma \in L^1(\Gamma) $ and any $ u \in BV(\Omega) \cap L^2(\Omega) $, we define a Radon measure $ [\beta |D u|]_\gamma \in \mathscr{M}_{\rm loc}(\R^N) $ by letting:
\begin{equation}\label{01_[bt|Du|]_gm}
\begin{array}{c}
\displaystyle
[\beta |Du|]_\gamma(B) := \int_{B \cap \Omega} [\beta]^* d |Du| +\int_{B \cap \Gamma} \beta |u -\gamma| \, d \mathcal{H}^{N -1} +\int_{B \setminus \overline{\Omega}} [\beta]^{\rm ex}|D [\gamma]^{\rm ex}| \, dx,
\\[3ex]
\mbox{for any bounded Borel set $ B \subset \R^N $.}
\end{array}
\end{equation}
Referring to the previous works, e.g. \cite[Theorem 6.1]{MR2306643}, we can reduce the measure $ [\beta|Du|]_\gamma $ in the following simple form:
\begin{equation}\label{02_[bt|Du|]_gm}
[\beta |Du|]_\gamma = [\beta]^{\rm ex} \, |D[u]_\gamma^{\rm ex}| \mbox{ \ in $ \mathscr{M}_{\rm loc}(\R^N) $.}
\end{equation}
In addition, as a consequence of \eqref{01_[bt|Du|]_gm} and Remark \ref{Rem.trace}, we can see the following facts.
\begin{description}
        \hypertarget{Fact2}{}
    \item[{(Fact\,2)}]If $ \{ u_n \}_{n = 1}^\infty \subset BV(\Omega) \cap L^2(\Omega) $, $ u \in BV(\Omega) \cap L^2(\Omega) $ and $ u_n \to u $ in $ L^2(\Omega) $ and strictly in $ BV(\Omega) $ as $ n \to \infty $, then $ [\beta|Du_n|]_\gamma(B) \to [\beta |Du|]_\gamma(B) $ as $ n \to \infty $, for any bounded Borel set $ B \subset \R^N $.
        \hypertarget{Fact3}{}
    \item[{(Fact\,3) (cf. \cite[Theorem 6.1]{MR2306643} and \cite[Section 2]{MR3268865})}]If $ A \subset \R^N $ is a bounded and open, and $ \beta \in H^1(\Omega) \cap L^\infty(\Omega) $ with $ \beta \geq 0 $ a.e. in $ \Omega $, then
\begin{equation}\label{exp01}
[\beta|Du|]_\gamma(A) = \sup \left\{ \begin{array}{l|l}
\ds \int_A u \, {\rm div} \, \varphi \, dx & \parbox{4.75cm}{
$ \varphi \in L^\infty(A)^N $, such that $ {\rm supp} \varphi $ is compact and $ |\varphi| \leq [\beta]^{\rm ex} $ a.e. in $ A $
}
\end{array} \right\},
\end{equation}
and in particular, if $ \log \beta \in L^\infty(\Omega) $, then
\begin{equation}\label{exp02}
[\beta|Du|]_\gamma(A) = \sup \left\{ \begin{array}{l|l}
\ds \int_A u \, {\rm div}(\beta \varphi) \, dx & \parbox{4.75cm}{
$ \varphi \in L^\infty(A)^N $, such that $ {\rm supp} \, \varphi $ is compact and $ |\varphi| \leq 1 $ a.e. in $ A $
}
\end{array} \right\}.
\end{equation}
\end{description}

Now, we set:
\begin{equation*}
\begin{array}{c}
\left\{ \begin{array}{l}
\ds X_{\rm c}(\Omega) := \left\{ \begin{array}{l|l}
\varphi \in L^{\infty}(\Omega)^{N} & {\rm div} \, \varphi \in L^{2}(\Omega) \mbox{ and supp $\varphi$ is compact in } \Omega
\end{array} \right\},
\\[1ex]
\ds W_{0}(\Omega) := \left\{ \begin{array}{l|l}
\beta \in H^{1}(\Omega) \cap L^{\infty}(\Omega) &
\parbox{2.75cm}{$ \beta \ge 0 $ a.e. in $ \Omega $}
\end{array} \right\},
\\[1ex]
\ds W_{\rm c}(\Omega) := \left\{ \begin{array}{l|l}
\beta \in H^{1}(\Omega) \cap L^{\infty}(\Omega) &
\parbox{2.75cm}{$ \log \beta \in L^\infty(\Omega) $}
\end{array} \right\}.
\end{array} \right.
\ \\[-2ex]
\end{array}
\end{equation*}

Besides, for any $ \beta \in W_0(\Omega) $, we define a functional {$\Phi_\gamma(\beta; {}\cdot{})$} on $L^{2}(\Omega)$, by letting
\begin{equation}\label{Phi_gm(bt.)}
    \ds u \in L^{2}(\Omega) \mapsto {\Phi_\gamma(\beta; u)} := \left\{ \begin{array}{ll}
\multicolumn{2}{l}{\ds \int_{\overline{\Omega}} d [\beta |Du|]_\gamma = \int_\Omega [\beta]^* \, d |Du| +\int_{\Gamma} \beta |u -\gamma| \, d\mathcal{H}^{N -1},}
\\[2ex]
& \mbox{if $ u \in BV(\Omega) $,}
\\[1ex]
\infty, & \mbox{otherwise.}
\end{array} \right.
\end{equation}
Note that the value of {$\Phi_\gamma(\beta; {}\cdot{})$} is determined, independently on the choice of the extensions $ [\beta]^{\rm ex} $ and $ [\gamma]^{\rm ex} $. Therefore, taking into account (\ref{01_[bt|Du|]_gm})--(\ref{Phi_gm(bt.)}), it is inferred that the functional $\Phi_\gamma(\beta; {}\cdot{})$ is proper l.s.c. and convex on $ L^2(\Omega) $.
In addition, the following holds:

\begin{description}
{
\hypertarget{Fact4}{}
\item[{(Fact\,4)}]\textbf{(cf. \cite{MR0102739,MR2139257,MR3751650})} If $ \beta \in W_{\rm c}(\Omega) $, $ \gamma \in H^{1}(\Gamma) $ and $ u \in BV(\Omega) \cap L^2(\Omega) $, then there exists a sequence $ \{ u_i \}_{i = 1}^\infty \subset H^1(\Omega) $ such that
\begin{equation}\label{reg01}
\left\{ \hspace{-4ex} \parbox{11.75cm}{
\vspace{-1ex}
\begin{itemize}
\item[]$ u_i = \gamma $, $ \mathcal{H}^{N -1} $-a.e. on $ \Gamma $, for any $ i \in \N $,
\item[]$ u_i \to u $ in $ L^2(\Omega) $ and $ \ds \int_\Omega \beta |\nabla u_i| \, dx \to {\Phi_\gamma(\beta; u)} $, as $ i \to \infty $.
\vspace{-1ex}
\end{itemize}
} \right.
\end{equation}
}
Hence, the functional $ {\Phi_\gamma(\beta;{}\cdot{})} $ coincides with the lower semi-continuous envelope of the functional:
$$
\varphi \in L^2(\Omega) \mapsto \left\{ \begin{array}{ll}
        \multicolumn{2}{l}{\ds \int_\Omega \beta |\nabla \varphi| \, dx, \mbox{ \ if {$ \varphi \in H^{1}(\Omega) $} and $ \varphi = \gamma $, $ \mathcal{H}^{N -1} $-a.e. on $ \Gamma $,}}
\\[2ex]
\infty, & \mbox{otherwise,}
\end{array} \right.
$$
i.e. for any $ v \in L^2(\Omega) $, it holds that
$$
{\Phi_\gamma(\beta; v)} = \inf \left\{ \begin{array}{l|l}
\ds \liminf_{i \to \infty} \int_\Omega \beta |\nabla \varphi_i| \, dx & \parbox{7cm}{
        {$ \{ \varphi_i \}_{i = 1}^\infty \subset H^{1}(\Omega) $,} $ \varphi_i = \gamma $, $ \mathcal{H}^{N -1} $-a.e. on $ \Gamma $, for any $ i \in \N $, and $ \varphi_i \to v $ in $ L^2(\Omega) $ as $ i \to \infty $
}
\end{array} \right\}.
$$
\end{description}

{
\begin{rem}\label{Rem.Fact4}
    In the construction of the sequence $ \{ u_i \}_{i = 1}^\infty \subset H^1(\Omega) $ as in \eqref{reg01}, the key point is to prepare an auxiliary sequence $ \{ u_i^\perp \}_{i = 1}^\infty \cap W^{1, 1}(\Omega) \cap L^2(\Omega) $ such that:
\begin{equation*}
\left\{ \hspace{-4ex} \parbox{12.75cm}{
\vspace{-1ex}
\begin{itemize}
\item[]$ u_i^\perp = \gamma $, $ \mathcal{H}^{N -1} $-a.e. on $ \Gamma $, for any $ i \in \N $,
\item[]$ u_i^\perp \to 0 $ in $ L^2(\Omega) $ and $ \ds \int_\Omega \beta |\nabla u_i^\perp| \, dx \to \int_\Gamma \beta|u| \, d \mathcal{H}^{N -1} $, as $ i \to \infty $,
\vspace{-1ex}
\end{itemize}
} \right.
\end{equation*}
    and such a sequence is easily obtained by referring to some appropriate general theories, e.g. \cite{MR0102739,MR2139257}. In addition, if we suppose $ \gamma \in H^1(\Gamma) $, then we can take the auxiliary sequence to satisfy $ \{ u_i^\perp \}_{i = 1}^\infty \subset H^1(\Omega) $ (cf. \cite[Lemma 2]{MR3751650}), and thereby, we can also suppose $ \{ u_i \}_{i = 1}^\infty \subset H^1(\Omega) $ for the sequence as in \eqref{reg01}.
\end{rem}
}
\section{Main Theorems}\label{sec:main}
\ \ \vspace{-4ex}

We begin with the assumptions for the Main Theorem.
\begin{enumerate}
    \hypertarget{A1}{}
    \item[\textmd{(A1)}]$ g \in W_{\rm loc}^{1, \infty}(\R) $, $ g(0) \leq 0 $ and $ g(1) \geq 1 $.
    \hypertarget{A2}{}
    \item[\textmd{(A2)}]$ \alpha_0 \in W_{\rm loc}^{1, \infty}(\R) $ and $ \alpha \in C^2(\R) $ are given functions (mobilities) such that $ \alpha $ is convex, $ \alpha'(0) = 0 $, and $ \delta_\alpha := \inf \alpha_0(\R) \cup \alpha(\R) > 0 $.
    {\hypertarget{A3}{}
    \item[\textmd{(A3)}]$ \gamma \in H^1(\Gamma) \cap L^\infty(\Gamma) $ is a fixed boundary datum, and the initial pair $ [\eta_0, \theta_0] \in L^2(\Omega) $ fulfills that:
\begin{equation*}
[\eta_0, \theta_0] \in D_0 := \left\{ \begin{array}{l|l}
[\tilde{\eta}, \tilde{\theta}] \in L^\infty(\Omega)^2 & \parbox{7cm}{
$ 0 \leq \tilde{\eta} \leq 1 $ and $ |\tilde{\theta}| \leq |\gamma|_{L^\infty(\Gamma)} $, a.e. in $ \Omega $
}
\end{array} \right\}.
\end{equation*}
}
\end{enumerate}
\begin{rem}[Possible choice for (\hyperlink{A1}{A1}) and (\hyperlink{A2}{A2})]
Referring to \cite{kobayashi2000continuum, MR1794359}, the setting in \eqref{concrete-setting} provides a possible choice of given functions that fulfills the above (\hyperlink{A1}{A1}) and (\hyperlink{A2}{A2}).
\end{rem}

Next, for simplicity of descriptions, we prescribe some additional notations.
\medskip

\noindent
\underline{\textit{Additional notations.}}
We define a functional $ \Psi_{0} : L^2(\Omega) \longrightarrow (-\infty, \infty] $, by letting:
\begin{equation*}
z \in L^2(\Omega) \mapsto \Psi_{0}(z) := \left\{ \begin{array}{ll}
\multicolumn{2}{l}{\ds \frac{1}{2} \int_\Omega |\nabla z|^2 \, dx +\int_\Omega G(z) \, dx, \mbox{ if $ z \in H^1(\Omega) $,}}
\\[2ex]
\infty, & \mbox{otherwise.}
\end{array} \right.
\end{equation*}
Then, the free-energy $ {\mathcal{F}_\gamma = \mathcal{F}_\gamma(\eta, \theta)} $ given in \eqref{freeEgy} can be described in the following simple formula:
\begin{align*}
    [\eta, \theta] \in & \times L^2(\Omega)^2 \mapsto \mathcal{F}_\gamma(\eta, \theta) := \Psi_0(\eta) +\Phi_\gamma(\alpha(\eta); \theta) \in (-\infty, \infty].
\end{align*}

Based on these, the solution to the system (S) is defined as follows.
\begin{defn}[Solutions to (S)]\label{Def.Sol}
A pair $ [\eta, \theta] \in L_{\rm loc}^2(0, T; L^2(\Omega))^2 $ is called a solution to (S), iff. the components $ \eta = \eta(t, x) $ and $ \theta = \theta(t, x) $ fulfill the following properties.
\begin{description}
    \hypertarget{S0}{}
    \item[\textmd{(S0)}]$ \eta \in C([0, \infty); L^2(\Omega)) \cap W_{\rm loc}^{1, 2}((0, \infty); L^2(\Omega)) \cap L_\mathrm{loc}^2([0, \infty); H^1(\Omega)) \cap L_{\rm loc}^\infty((0, \infty); H^1(\Omega)) $,
$ \theta \in C([0, \infty); L^2(\Omega)) \cap W_{\rm loc}^{1, 2}((0, \infty); L^2(\Omega)) $, $ |D \theta({}\cdot{})|(\Omega) \in L_\mathrm{loc}^1(0, \infty) \cap L_{\rm loc}^\infty((0, \infty)) $, $ [\eta(t), \theta(t)] \in D_0 $ for any $  t \geq 0 $, and $ [\eta(0), \theta(0)] = [\eta_0, \theta_0] $ in $ L^2(\Omega)^2 $.
\hypertarget{S1}{}
\item[\textmd{(S1)}]$ \eta $ solves the following variational identity:
\begin{align*}
    \ds \int_\Omega \bigl( \eta_t(t) +g(\eta(t)) & \bigr) w \, dx +\int_\Omega \nabla \eta(t) \cdot \nabla w \, dx {+\int_{\overline{\Omega}} d \bigl[  w \alpha'(\eta(t)) |D\theta(t)| \bigr]_{\gamma}} = 0,
\\
& \mbox{for any $ w \in H^1(\Omega) \cap L^\infty(\Omega) $ and a.e. $ t > 0 $.}
\end{align*}
\hypertarget{S2}{}
\item[\textmd{(S2)}]$ \theta $ solves the following variational inequality:
\begin{align*}
    \ds \int_\Omega \alpha(\eta(t)) \theta_t(t) (\theta(t) & -v) \, dx {+\int_{\overline{\Omega}} d \bigl[\alpha(\eta(t)) |D\theta(t)| \bigr]_{\gamma}}
    \ds \leq {\int_{\overline{\Omega}} d \bigl[\alpha(\eta(t)) |Dv| \bigr]_{\gamma}},
\\
& \mbox{for any $ v \in BV(\Omega) \cap L^2(\Omega) $ and a.e. $ t > 0 $.}
\end{align*}
\end{description}
\end{defn}
\begin{rem}\label{Rem.Sol01}
    In the light of \eqref{01_[bt|Du|]_gm}, the variational identity in Definition \ref{Def.Sol} (\hyperlink{S1}{S1}), can be rewritten into the following weak formulation:
\begin{equation*}
    \eta_{t}(t) + F\eta(t) - \eta(t) + g(\eta(t)) {+[\alpha'(\eta(t))|D \theta(t)|]_{\gamma}} = 0 \mbox{ in } H^{s}(\Omega)^*, \mbox{ a.e. } t > 0,
\end{equation*}
where $ F : H^1(\Omega) \to H^1(\Omega)^* $ is the duality map between $ H^1(\Omega) $ and $ H^1(\Omega)^* $, and $s > N/2$ is a large constant such that the embedding $ H^s(\Omega) \subset C(\overline{\Omega}) $ holds true. Thus, we note that the inhomogeneous boundary condition in \eqref{1st.eq} is implicitly built in the expressions of the $ H^s(\Omega)^* $-valued function $ {[\alpha'(\eta(\cdot)) |D \theta(\cdot)|]_{\gamma}} \in L^2(0, T; H^s(\Omega)^*) $, and specifically, $ {[\alpha'(\eta(t)) |D \theta(t)|]_{\gamma}} \in \mathscr{M}(\Omega) $ for a.e. $ t \in (0, T) $.

    In the meantime, the variational inequality in Definition \ref{Def.Sol} (\hyperlink{S2}{S2}) is equivalent to the following evolution equation:
$$
\alpha_{0}(\eta(t))\theta_{t}(t) {+\partial\Phi_\gamma(\alpha(\eta(t));\theta(t))} \ni 0 \mbox{ in } L^{2}(\Omega),  \mbox{ a.e. } t > 0,
$$
which is governed by the $ L^2 $-subdifferential $ {\partial \Phi_\gamma(\alpha(\eta(t));\cdot\,)} $ of the time-dependent convex function $ {\Phi_\gamma(\alpha(\eta(t)); \cdot\,)} $.
From this reformulation, we can see that the mathematical meaning of the singular diffusion $-\mbox{div}(\alpha(\eta)\frac{D\theta}{|D\theta|})$ in \eqref{2nd.eq} is given in terms of the time-dependent subdifferential $ {\partial \Phi_\gamma(\alpha(\eta(t));{}\cdot{})} $ of the weighted total variation, including the Dirichlet type boundary condition.
\end{rem}
\medskip

Now, the Main Theorems of this paper are stated as follows.
\paragraph{Main Theorem 1 (Existence of solution with energy-dissipation).}
{Let us assume (\hyperlink{A1}{A1})--(\hyperlink{A3}{A3}). Then, the system (\hyperlink{S}{S}) admits at least one solution $ [\eta, \theta] $ which fulfills the following properties.
\begin{description}
        \hypertarget{S4}{}
    \item[\textmd{(S4)}](Energy-dissipation) A real function $ \mathcal{J}_\gamma \in L_\mathrm{loc}^1([0, \infty)) \cap BV_\mathrm{loc}((0, \infty)) $, defined as:
    {
        \begin{align*}
            t \in (0, \infty) \mapsto & \mathcal{J}_\gamma(t) := \mathcal{F}_\gamma(\eta(t), \theta(t))
            \in [0, \infty),
        \end{align*}
    }
fulfills the following inequality:
{
\begin{align*}
\int_s^t \bigl( |\eta_t(\tau)|_{L^2(\Omega)}^2 & +|{\textstyle \sqrt{\alpha_0(\eta(\tau))} \theta_t(\tau)}|_{L^2(\Omega)}^2 \bigr) \, d \tau +\mathcal{J}_\gamma(t) \leq \mathcal{J}_\gamma(s)
\\
& \mbox{for a.e. $ s > 0 $ and any $ t \geq s $.}
\end{align*}
}
        \hypertarget{S5}{}
\item[\textmd{(S5)}](Properties in regular cases of initial data) If $ [\eta_0, \theta_0] \in D_0 \cap (H^1(\Omega) \times BV(\Omega)) $
then
\begin{equation*}
\begin{cases}
\eta \in W_\mathrm{loc}^{1, 2}([0, \infty); L^2(\Omega)) \cap L_\mathrm{loc}^\infty([0, \infty); H^1(\Omega)),
\\
\theta \in W_\mathrm{loc}^{1, 2}([0, \infty); L^2(\Omega)), ~~ |D \theta({}\cdot{})|(\Omega) \in L_\mathrm{loc}^\infty([0, \infty)),

\end{cases}
\end{equation*}
and
{
\begin{align*}
\int_0^t \bigl( |\eta_t(\tau)|_{L^2(\Omega)}^2 & +|{\textstyle \sqrt{\alpha_0(\eta(\tau))} \theta_t(\tau)}|_{L^2(\Omega)}^2 \bigr) \, d \tau +\mathcal{J}_\gamma(t) \leq \mathcal{J}_\gamma(0), \mbox{ \ for all $ t \geq 0 $.}
\end{align*}
}
\end{description}
}

\begin{rem}\label{Rem.sw-sols} Property (\hyperlink{S5}{S5}), though not used in the paper, is characteristic of solutions to parabolic PDEs or parabolic systems of PDEs. In this special case, $ [\eta_0, \theta_0] \in D_0 \cap (H^1(\Omega) \times BV(\Omega)) $, we will call the solution {\em a strong solution}, in view of the smoothing effect as in (\hyperlink{S0}{S0}) and (\hyperlink{S5}{S5}).
\end{rem}

\paragraph{Main Theorem 2 (Large-time behavior).}
{
    Under the assumptions (\hyperlink{A1}{A1})--(\hyperlink{A3}{A3}), let $ [\eta, \theta] $ be the solution to (S) with the property (\hyperlink{S4}{S4}) of energy-dissipation, that is obtained in Main Theorem 1. We denote by $ \omega_\infty(\eta, \theta) $ the $ \omega $-limit set of $ [\eta, \theta] $ in the large-time, i.e.:
\begin{equation*}
\omega_\infty(\eta, \theta) := \left\{ \begin{array}{l|l}
[\eta_\infty, \theta_\infty] \in L^2(\Omega)^2
& \parbox{6.6cm}{
$ [\eta(t_n), \theta(t_n)] \to [\eta_\infty, \theta_\infty] $ in $ L^2(\Omega)^2 $ for some $ 0 < t_1 < \cdots < t_n \uparrow \infty $, as $ n \to \infty $
}
\end{array} \right\}.
\end{equation*}
Then, $ \omega_\infty(\eta, \theta) $ is nonempty and compact in $ L^2(\Omega)^2 $, and any $ \omega $-limit point $ [\eta_\infty, \theta_\infty] $ fulfills the following conditions.
\begin{description}
    \hypertarget{S0infty}{}
    \item[\textmd{(S0)$_\infty$}]$ [\eta_\infty, \theta_\infty] \in D_0 \subset H^1(\Omega) \times BV(\Omega) $.
        \hypertarget{S1infty}{}
\item[\textmd{(S1)$_\infty$}]$ \eta_\infty $ solves the following variational identity:
    {
        \begin{align*}
            \int_\Omega \nabla \eta_\infty \cdot \nabla w \, dx & + \int_\Omega g(\eta_\infty) w \, dx +\int_{\overline{\Omega}} d \bigl[ w \alpha'(\eta_\infty) |D \theta_\infty| \bigr]_{\gamma} = 0,
            \\
            & \mbox{for any $ w \in H^1(\Omega) \cap L^\infty(\Omega) $. }
        \end{align*}
    }
        \hypertarget{S2infty}{}
\item[\textmd{(S2)$_\infty$}]$ \theta_\infty $ is the minimizer of {$ \Phi_\gamma(\alpha(\eta_\infty);{}\cdot{}) $} in $ L^2(\Omega) $, more precisely:
    {
        \begin{align*}
            \int_{\overline{\Omega}} d \bigl[ \alpha( & \eta_\infty) |D \theta_\infty| \bigr]_{\gamma} \leq  \int_{\overline{\Omega}} d \bigl[ \alpha(\eta_\infty) |D v| \bigr]_{\gamma} \,,\quad \mbox{for any $ v \in BV(\Omega) \cap L^2(\Omega) $. }
        \end{align*}
    }
\end{description}

}
\section{Approximating problem}\label{sec:approx}
~~ \vspace{-3ex}

In this section, the approximating problems to the system (S) are prescribed, and some key-properties of the approximating solutions are verified.

To this end, we first consider a class $ \{ |{}\cdot{}|_\nu \}_{\nu \in (0, 1)} $ of functions, which was adopted in \cite{MR3670006} as a ``suitable approximation'' to the Euclidean norm $ |{}\cdot{}| $.

\begin{defn}\label{euclregul}
We say that a collection of functions $\{|\cdot|_\nu\}_{\nu\in(0,1)}$ is a suitable approximation to the {Euclidean} norm if the following properties hold.
\begin{description}
\hypertarget{AP1}{}
\item [\textmd{(AP1)}] $|{}\cdot{}|_\nu:\R^N\mapsto [0,+\infty[$ is a convex $C^1$ function such that $|0|_\nu = 0 $, for all $ \nu \in (0, 1) $.
\hypertarget{AP2}{}
\item [\textmd{(AP2)}]There exist bounded functions $ a : (0,1) \to (0,1] $, $ b : (0, 1) \to [0, 1] $, $ c : (0,1) \to [1,\infty) $, such that:
\begin{equation*}
a(\nu) \to 1, ~ b(\nu) \to 0 \mbox{ and } c(\nu) \to 1, \mbox{ as $ \nu \downarrow 0 $,}
\end{equation*}
and
\begin{equation*}
\begin{array}{c}
\ds
|\xi|_\nu \geq a(\nu) |\xi| -b(\nu) \mbox{ and } |[\nabla |\cdot|_\nu](\xi)| \leq c(\nu)
\\[1ex]
\mbox{for any $ \xi \in \R^N $ and $ \nu \in (0, 1) $.}
\end{array}
\end{equation*}
\end{description}
\end{defn}

\begin{rem}\label{Rem.Euc}
Note that (\hyperlink{AP1}{AP1}) and (\hyperlink{AP2}{AP2}) lead to the following fact:
\begin{center}
$ |\xi|_\nu \leq [\nabla |\cdot|_\nu](\xi) \cdot \xi \leq c(\nu)|\xi| $, \ for all $ \xi \in \R^N $ and $ \nu \in (0, 1) $. \end{center}
Also, we note that the class of possible regularizations verifying  (\hyperlink{AP1}{AP1}) and (\hyperlink{AP2}{AP2}) covers a number of standard type regularizations, while it is a restricted version of the class adopted in \cite{MR3670006}. For instance,
\begin{itemize}
\item Hyperbola type, i.e. $\xi \in \R^N \mapsto |\xi|_\nu:=\sqrt{|\xi|^2 +\nu^2} -\nu$, for $ \nu \in (0, 1) $,
\item Yosida's regularization, i.e.
$ \xi \in \R^N \mapsto |\xi|_\nu := \ds \inf_{\varsigma \in \R^N} \left\{ |\xi| +\frac{\nu}{2} |\varsigma -\xi|^2 \right\} $, for $ \nu \in (0, 1) $,
\vspace{-1ex}
\item Hyperbolic-tangent type, i.e.
$ \xi \in \R^N \mapsto |\xi|_\nu := \ds \int_0^{|\xi|} \tanh \frac{\tau}{\nu} \, d \tau $, for $ \nu \in (0, 1) $,
\vspace{-1ex}
\item Arctangent type, i.e.
$ \xi \in \R^N \mapsto |\xi|_\nu := \ds \frac{2}{\pi} \int_0^{|\xi|} {\rm Tan}^{-1} \frac{\tau}{\nu} \, d \tau $, for $ \nu \in (0, 1) $,
\end{itemize}
Such flexibility is the reason for the ``suitability'' as in Definition \ref{euclregul}, and indeed, the essential thing in the proof of Main Theorem is not in the precise expression of $ \{ |{}\cdot{}|_\nu \}_{\nu \in (0, 1)} $, but in its suitability.
\end{rem}

Next, for any  $ 0 \leq \beta \in L^2(\Omega) $ and $ \gamma \in H^{\frac{1}{2}}(\Gamma) $, we define a class {$ \{ \Phi^{\nu}_{\gamma}(\beta;{}\cdot{}) \}_{\nu \in (0, 1)} $} of proper l.s.c. and convex functions on $ L^2(\Omega) $, by putting:
    \begin{equation*}
        z \in L^2(\Omega)  \mapsto \Phi^{\nu}_{\gamma}(\beta; z)
        \\
        := \left\{ \begin{array}{ll}
        \multicolumn{2}{l}{\ds \int_\Omega \beta |\nabla z|_\nu \, dx +\frac{\nu^2}{2} \int_\Omega |\nabla (z -[\gamma]^{\rm hm})|^2 \, dx,}
        \\[1ex]
        & \mbox{if $ z \in H^1(\Omega) $ and $ z = \gamma $ in $ H^{\frac{1}{2}}(\Gamma) $,}
        \\[2ex]
        \infty, & \mbox{otherwise,}
    \end{array} \right.
\end{equation*}
where $ [\gamma]^{\rm hm} \in H^2(\Omega) $ denotes the harmonic extension of $ \gamma \in H^{\frac{1}{2}}(\Gamma) $ as in Remark \ref{Rem.ex_hm}. This convex function corresponds to a relaxed version of the convex function {$ \Phi_{\gamma}(\beta;{}\cdot{}) $} given in \eqref{Phi_gm(bt.)}.
\begin{rem}\label{Rem.subdif_rx}
    For every $ 0 \leq \beta \in L^2(\Omega) $ and $ \gamma \in H^{\frac{1}{2}}(\Gamma) $, we denote by {$ \partial \Phi^{\nu}_{\gamma}(\beta;{}\cdot{}) $} the $ L^2 $-subdifferential of the convex function {$ \Phi^{\nu}_{\gamma}(\beta;{}\cdot{}) $}. Here, let us define:
\begin{center}
$ \jump{\beta \nabla u}_{\gamma}^{\nu} := \beta [\nabla |{}\cdot{}|_\nu](\nabla u) +\nu^2 \nabla (u -[\gamma]^{\rm hm}) $ in $ L^2(\Omega)^N $, for any $ u \in H^1(\Omega) $.
\end{center}
Then, as is easily checked (cf. \cite[Section 1]{MR0372419}), {$ [u, u^*] \in \partial \Phi^{\nu}_{\gamma}(\beta;{}\cdot{}) $} in $ L^2(\Omega)^2 $, iff. {$ u \in D(\Phi^{\nu}_{\gamma}(\beta;{}\cdot{})) $}, and
\begin{equation*}
u^* = -\mathrm{div} \jump{\beta \nabla u}_{\gamma}^{\nu} \mbox{ in the distribution sense,}
\end{equation*}
so that:
\begin{equation*}
    {D(\partial \Phi^{\nu}_{\gamma}(\beta;{}\cdot{}))} = \left\{ \begin{array}{l|l}
z \in H^1(\Omega) & \mathrm{div} \jump{\beta \nabla u}_{\gamma}^{\nu} \in L^2(\Omega) \mbox{ and $ u = \gamma $ in $ H^{\frac{1}{2}}(\Gamma) $}
\end{array} \right\}.
\end{equation*}
Additionally, for each {$ [u, u^*] \in \partial \Phi^{\nu}_{\gamma}(\beta;{}\cdot{}) $} in $ L^2(\Omega)^2 $, the directional derivative:
\begin{center}
$ \jump{\beta \nabla u}_{\gamma}^{\nu} \cdot n_\Gamma = \beta \bigl( [\nabla |{}\cdot{}|_\nu](\nabla u) +\nu^2  (u -[\gamma]^{\rm hm}) \bigr) \cdot n_\Gamma $,
\end{center}
can be identified as an element of $ H^{-\frac{1}{2}}(\Gamma) $, via the following variational form:
\begin{equation*}
\begin{array}{c}
\ds \bigl\langle  \jump{\beta \nabla u}_{\gamma}^{\nu} \cdot n_\Gamma, \varphi \bigr\rangle_{H^{\frac{1}{2}}(\Gamma)} = (u^*, \varphi)_{L^2(\Omega)} -\bigl(\jump{\beta \nabla u}_{\gamma}^{\nu}, \nabla \varphi \bigr)_{L^2(\Omega)^N}
\\[1ex]
\mbox{for any $ \varphi \in  H^1(\Omega) $.}
\end{array}
\end{equation*}
\end{rem}

Now, let us fix any $ \nu \in (0, 1) $, and for any $ \gamma \in H^{\frac{1}{2}}(\Gamma) $, we define:
{
    \begin{equation}\label{rxFE}
        \begin{array}{c}
            [\eta, \theta] \in L^2(\Omega)^2 \mapsto \mathcal{F}^{\nu}_{\gamma}(\eta, \theta) := \Psi_0(\eta) +\Phi^{\nu}_{\gamma}(\eta; \theta).
        \end{array}
    \end{equation}
}
Then, as a gradient system of this free-energy, we derive the following relaxed system, denoted by (S$_\nu$).
\bigskip

\noindent
$(\mbox{S$_\nu$})$:
\begin{equation*}
\ds \eta_{t}^{\nu}(t) -{\laplace}_{\rm N}\eta^{\nu}(t) + g(\eta^{\nu}(t)) + \alpha'(\eta^{\nu}(t))|\nabla\theta^{\nu}(t)|_\nu = 0
\ds\ \mbox{ in}\ L^{2}(\Omega),\ \mbox{ a.e. } t\in(0, T),
\vspace{1mm}
\end{equation*}
{
    \begin{equation*}
        \ds \alpha_{0}(\eta^{\nu}(t)) \theta_{t}^{\nu}(t) + \partial\Phi^{\nu}_\gamma(\alpha(\eta^{\nu}(t));\theta^{\nu}(t)) \ni 0 \mbox{ \ in \ } L^{2}(\Omega),\ \mbox{ a.e. } t\in(0, T),
        \vspace{1mm}
    \end{equation*}
}
\begin{equation*}
\ds [\eta^{\nu}(0), \theta^{\nu}(0)] = [\eta_{0}, \theta_{0}] \ \mbox{ in}\ L^{2}(\Omega).
\vspace{1mm}
\end{equation*}
In the context, the initial data $[\eta_{0}, \theta_{0}]$ satisfy
\begin{equation}\label{initial-2}
    [\eta_{0}, \theta_{0}] \in {D_0^\mathrm{rx} :=
    \left\{ \begin{array}{l|l}
        [\tilde{\eta}, \tilde{\theta}] \in D_0 \cap H^1(\Omega)^2 & \parbox{3cm}{$ \tilde{\theta} = \gamma $ in $ H^{\frac{1}{2}}(\Gamma) $}
    \end{array} \right\}.}
\end{equation}

On account of Remark \ref{Rem.subdif_rx}, the exact statement of the approximating problem (\hyperlink{AP_h^nu}{AP}$_h^\nu$) is prescribed as follows.

\begin{description}
        \hypertarget{AP_h^nu}
\item[\textmd{$ \mbox{(AP$_h^\nu $)} $:}]for any $ [\eta_0, \theta_0] \in D_{0}^\mathrm{rx} $, find a sequence $ \{ [\eta_{h, i}^{\nu}, \theta_{h, i}^{\nu}] \}_{i = 1}^\infty \in H^1(\Omega)^2 $ $ (\subset H^1(\Omega)^2) $, such that
\begin{equation}\label{kenApp00}
    [\eta_{h, i}^{\nu}, \theta_{h, i}^{\nu}] \in {D_{0}^\mathrm{rx},} \mbox{ for every $ i \in \N $,}
\end{equation}
\begin{equation}\label{kenApp01}
\begin{array}{ll}
\multicolumn{2}{l}{\ds
\frac{1}{h} (\eta_{h, i}^{\nu} -\eta_{h, i -1}^{\nu}, w)_{L^2(\Omega)} +(\nabla \eta_{h, i}^{\nu}, \nabla w)_{L^2(\Omega)^N} +(g(\eta_{h, i}^{\nu}), w)_{L^2(\Omega)}
}
\\[1ex]
\qquad \qquad & \ds +(\alpha'(\eta_{h, i}^{\nu}) |\nabla \theta_{h, i -1}^{\nu}|_\nu, w)_{L^2(\Omega)} = 0, \mbox{ for every $ w \in H^1(\Omega) $ and $ i \in \N $,}
\end{array}
\vspace{1ex}
\end{equation}
\begin{equation}\label{kenApp02}
    \frac{1}{h} \alpha_0(\eta_{h, i}^{\nu}) (\theta_{h, i}^{\nu} -\theta_{h, i -1}^{\nu}) {-\mathrm{div} \jump{\alpha(\eta_{h,i}^\nu) \nabla \theta_{h, i}^{\nu}}_{\gamma}^{\nu}} = 0 \mbox{ in $ L^2(\Omega) $, for every $ i \in \N $,}
\end{equation}
\end{description}
On this basis, we call the above sequence $ \{ [\eta_{h, i}^{\nu}, \theta_{h, i}^{\nu}] \}_{i = 1}^\infty \subset H^1(\Omega)^2 $ a solution to the approximating problem $ \mbox{(\hyperlink{AP_h^nu}{AP}$_h^\nu $)} $, or an approximating solution in short.
{
\begin{rem}\label{noteRx}
    We note that the range $ D_0^\mathrm{rx} $ of approximating solutions coincides with the effective domains of the relaxed convex functions $ D(\Phi_\gamma^\nu(\beta;{}\cdot{})) $ in $D_0$, for all $ \nu \in (0, 1) $, $ 0 \leq \beta \in L^2(\Omega) $, and $ \gamma \in H^{\frac{1}{2}}(\Gamma) $.
    Furthermore, invoking (\hyperlink{A3}{A3}) and Remark \ref{Rem.Fact4}, we can see that $ D_0^\mathrm{rx} $ is dense in $ D_0 $, in the topology of $ L^2(\Omega)^2 $. In fact, for any $ [\tilde{\eta}, \tilde{\theta}] \in D_0 $, we can construct an approximating sequence $ \{ [\tilde{\eta}_{n}, \tilde{\theta}_{n}] \}_{n = 1}^\infty \subset D_0^\mathrm{rx} $ of $ [\tilde{\eta}, \tilde{\theta}] $ in $ L^2(\Omega) $, as follows:
    \begin{equation*}
        [\tilde{\eta}_{n}, \tilde{\theta}_{n}] := [\tilde{\eta}, \tilde{\theta}_n^\circ +\tilde{\theta}_n^\perp] \mbox{ in $ H^1(\Omega) $, for $ n = 1, 2, 3, \dots $,}
    \end{equation*}
    where $ \{ \tilde{\theta}_n^\circ \}_{n = 1}^\infty \subset H_0^1(\Omega) $ is a standard approximating sequence of $ \tilde{\theta} $ in $ L^2(\Omega)^2 $, and $ \{ \theta_n^\perp \}_{n = 1}^\infty \subset H^1(\Omega) $ is a sequence, satisfying:
    \begin{equation*}
        \begin{cases}
            \tilde{\theta}_n^\perp = \gamma \mbox{ in $ H^{\frac{1}{2}}(\Gamma) $, for $ n = 1, 2, 3, \dots $,}
            \\[1ex]
            \tilde{\theta}_n^\perp \to 0 \mbox{ in $ L^2(\Omega) $, and } |D \tilde{\theta}_n^\perp|(\Omega) \to |\gamma|_{L^1(\Gamma)}, \mbox{ as $ n \to \infty $.}
        \end{cases}
    \end{equation*}
    Note that we can take the above sequence $ \{ \tilde{\theta}_n^\perp \}_{n = 1}^\infty $ by referring to Remark \ref{Rem.Fact4}.
\end{rem}
}

{Now, for the mathematical analysis of the approximating problem, we need to prove the following theorems, concerned with energy-inequalities for the approximating solutions.}
\begin{thm}[{\boldmath Energy-inequality for time-global estimates}]\label{thAp1}
    Under the assumptions {(\hyperlink{(A1)}{A1})--(\hyperlink{(A4)}{A4}), there exists a constant $h_* \in (0, 1)$, independent of the initial data $ [\eta_0^\nu, \theta_0^\nu] $ and  $ \{ \gamma \} $}, such that for any $ h \in (0, h_*) $ and $[\eta_0^\nu,\theta_0^\nu]\in D_{0}^\mathrm{rx} $, $(\hyperlink{AP_h^nu}{AP})_{h}^{\nu}$ admits a unique solution $ \{ [\eta_{h, i}^{\nu}, \theta_{h, i}^{\nu}]\}_{i = 1}^\infty $, such that:
\begin{align}\label{ene-inq}
    \ds \frac{1}{2h} |\eta_{h, i}^{\nu} & -\eta_{h, i-1}^{\nu}|_{L^{2}(\Omega)}^{2} +\frac{1}{h} \left| {\textstyle \sqrt{\alpha_{0}(\eta_{h, i}^{\nu})}}(\theta_{h, i}^{\nu} - \theta_{h, i-1}^{\nu}) \right|_{L^{2}(\Omega)}^{2}
\nonumber
\\
& \ds {+\mathcal{F}^{\nu}_{\gamma}(\eta_{h, i}^{\nu}, \theta_{h, i}^{\nu}) \leq \mathcal{F}^{\nu}_{\gamma}(\eta_{h, i -1}^{\nu},\theta_{h, i -1}^{\nu}), \mbox{ for $ i = 1, 2, 3, \dots $,}}
\end{align}
and
\begin{align}\label{lem2-inq}
    \ds \frac{1}{2} \sum_{i = 1}^m i |\eta_{h, i}^{\nu} & -\eta_{h, i -1}^{\nu}|^{2}_{L^{2}(\Omega)} +\sum_{i = 1}^m i {\left| {\textstyle \sqrt{\alpha_{0}(\eta_{h,i}^{\nu})}}(\theta_{h,i}^{\nu} -\theta_{h,i -1}^{\nu}) \right|_{L^{2}(\Omega)}^{2} }
\nonumber
\\
& \ds {+mh \mathcal{F}^{\nu}_{\gamma}(\eta_{h, m}^{\nu}, \theta_{h, m}^{\nu})} \le h \sum_{i = 1}^m \mathcal{F}^{\nu}_{\gamma}(\eta_{h, i -1}^{\nu}, \theta_{h, i -1}^{\nu}), \mbox{ for any $ m \in \N $.}
\end{align}
\end{thm}

\begin{thm}[{\boldmath Energy-inequality for short-time estimates}]\label{thAp2}
Under the assumptions and notations as in Theorem \ref{thAp1}, there exists constants $ \nu_* $ and $ A_k $, $ k = 1, 2, 3 $, such that if $ h \in (0, h_*) $ and $ \nu \in (0, \nu_*) $, then:
\begin{align}\label{thAp2-00}
    \frac{1}{2} \bigl( |{\eta_{h, m}^{\nu}} & - w_0 |_{L^2(\Omega)}^2 -|{\eta_{h, 0}^{\nu}} -w_0|_{L^2(\Omega)}^2 \bigr)
    \nonumber
    \\
    & +\frac{A_1}{2} \bigl( |{\theta_{h, m}^{\nu}} -v_0|_{L^2(\Omega)}^2 -|{\theta_{h, 0}^\nu} -v_0|_{L^2(\Omega)}^2 \bigr)
\nonumber
\\
& {+A_2 h \sum_{i = 1}^{m} \mathcal{F}^{\nu}_{\gamma}(\eta_{h, i -1}^{\nu}, \theta_{h, i -1}^{\nu}) \leq {\frac{h}{A_2}} \mathcal{F}^{\nu}_{\gamma}(\eta_0^\nu, \theta_0^\nu)}
\\
& +A_3 {mh} (1 +|w_0|_{H^1(\Omega)}^2 +|v_0|_{H^1(\Omega)}^2 +|{\gamma}|_{H^{\frac{1}{2}}(\Gamma)}^2),
\nonumber
\\
& \qquad \mbox{for any $ [w_0, v_0] \in {D(\mathcal{F}_{\gamma}^{\nu}(\cdot))} $ and any $ m \in \N $.}
\nonumber
\end{align}
\end{thm}
\section{Auxiliary lemmas}
\ \ \vspace{-4ex}

Throughout this section, we fix a bounded and open time-interval $ I \subset (0, \infty) $, and we prove some auxiliary lemmas concerned with the mathematical treatments of time-dependent $ BV(\Omega) $-valued functions on $ I $.

{
    We begin with recalling the following fact, obtained in \cite[Lemma 5]{MR3268865}, \cite[(Fact\,5)]{MR3670006}, and \cite[Remark 2]{MR3462536}.
\begin{description}
    \hypertarget{Fact5}{}
    \item[(Fact\,5) (Strict approximation)]For any $v \in  L^{2}(I;L^{2}(\Omega))$ satisfying $|D v(\cdot)|(\Omega) \in L^{1}(I)$, there exists a sequence of smooth functions $\{ v_{n} \}_{n = 1}^\infty \subset C^{\infty}(\overline{I\times\Omega})$, such that:
\begin{equation*}
    \begin{array}{c}
        \ds v_{n} \to v \mbox{ in } L^{2}(I;L^{2}(\Omega)) \mbox{ and } \int_{I} \left| \int_{\Omega} |\nabla v_{n}(t)| dx -\int_{\Omega}d|D v(t)| \right| dt \to 0\,,\ \
        \mbox{as $ n \to \infty $,}
    \end{array}
\end{equation*}
and moreover,
\begin{equation*}
    \begin{array}{c}
        \ds v_n(t) \to v(t) \mbox{ in } L^2(\Omega) \mbox{ and } \int_{\Omega} |\nabla v_{n}(t)| dx \to \int_{\Omega}d|D v(t)|, \mbox{ a.e. $ t \in I $\,},~
        \mbox{as $ n \to \infty $,}
    \end{array}
\end{equation*}
\end{description}
}

{
    As an advanced version of the above (\hyperlink{Fact5}{Fact\,5}), we next prove the following Lemma.
\begin{lem}\label{axLem_reg}
    Let $ \gamma \in H^{\frac{1}{2}}(\Gamma) $ be a fixed function, and let $ \{ \gamma_n \}_{n = 1}^\infty \subset H^{1}(\Gamma) $ be a sequence of functions, such that:
    \begin{equation}\label{convGamma}
        \gamma_n \to \gamma \mbox{ weakly in $ H^{\frac{1}{2}}(\Gamma) $, as $ n \to \infty $.}
    \end{equation}
    Then, for any $ v \in L^2(I; L^2(\Omega))  $ satisfying $  |Dv({}\cdot{})|(\Omega) \in L^1(I) $, there exists a sequence $ \{ v_n \}_{n = 1}^\infty \subset L^\infty(I; H^1(\Omega)) $, such that:
\begin{equation}\label{v_n^circ01}
v_n(t) = \gamma_n \mbox{ in $ H^{\frac{1}{2}}(\Gamma) $, for a.e. $ t \in I $,}
\end{equation}
\begin{equation}\label{v_n^circ02}
\begin{array}{c}
\ds v_n \to v \mbox{ in $ L^2(I; L^2(\Omega)) $ and }
    \int_I \left| \int_\Omega |\nabla v_n(t)| \, dx -\int_{\overline{\Omega}} d \bigl[ |Dv({}\cdot{})| \bigr]_{\gamma} \right| \, dt \to 0,
\\[2ex]
\mbox{ as $ n \to \infty $.}
\end{array}
\end{equation}
\end{lem}
}
{
\noindent
\textbf{Proof.}
The proof of this Lemma is a modified version of those of \cite[Proposition 2.7.1]{Kenmochi81} and \cite[Lemma 5]{MR3268865}.

First, invoking the assumption $ v \in L^2(I; L^2(\Omega)) $ with $ |Dv({}\cdot{})|(\Omega) \in L^1(I) $, (\hyperlink{Fact5}{Fact\,5}), \cite[Theorem 3.88]{MR1857292}, and \cite[Theorem III.6.10]{MR0117523}, we can see that the functions in time:
\begin{equation*}
    t \in I \mapsto v(t) \in L^2(\Omega),~ t \in I \mapsto |Dv(t)|(\Omega) \in \R, \mbox{ and } t \in I \mapsto v(t)_{|_\Gamma} \in L^1(\Gamma),
\end{equation*}
are Borel measurable (cf. \cite[Section 2]{MR2072104}) from  $ I $ into the respective separable Banach spaces. Additionally, \eqref{Phi_gm(bt.)} allows us to say that
\begin{equation*}
    t \in I \mapsto [|Dv(t)|]_{\gamma_n}(\overline{\Omega}) = |Dv(t)|(\Omega) +|v_{|_\Gamma} -\gamma_n|_{L^1(\Gamma)} \in \R, ~ n = 1, 2, 3, \dots,
\end{equation*}
are Borel measurable.

Now, with the assumption \eqref{convGamma} in mind, we define
\begin{equation*}
    \bar{\gamma}_* := \sup \left\{ \begin{array}{l|l}
        |\gamma|_{H^{\frac{1}{2}}(\Gamma)}, |\gamma_n|_{H^{\frac{1}{2}}(\Gamma)} & n = 1, 2, 3, \dots
    \end{array} \right\} < \infty.
\end{equation*}
Then, applying Lusin's theorem, we find a sequence of compact sets $ \{ K_n \}_{n = 1}^\infty \subset 2^I $, such that:
\begin{equation*}
    \begin{array}{c}
    \left\{ \hspace{-2ex} \parbox{12cm}{
    \vspace{-2ex}
    \begin{itemize}
    \item $ K_n \subset K_{n +1} \subset I $,
    \item $ v \in C(K_n; L^2(\Omega)) $, $ [|Dv({}\cdot{})|]_{\gamma} \in C(K_n) $,
    \item $ |v|_{L^2(I \setminus K_n; L^2(\Omega))} \leq 2^{-(n +1)} $, $ \bigl| [|Dv({}\cdot{})|]_{\gamma}(\overline{\Omega}) \bigr|_{L^1(I \setminus K_n)} \leq 2^{-(n +1)} $,
    \\[1ex]
            $ \bigl( 1 +\mathcal{L}^N(\Omega) \bigr)\bigl( 1 +C_*^\mathrm{hm} \, \bar{\gamma}_* \bigr) \mathcal{L}^1(I \setminus K_n) \leq 2^{-(n +1)} $,
    \vspace{-2ex}
    \end{itemize}
    } \right.
        \ \\[-2ex]
        \ \\
    \mbox{for $ n = 1, 2, 3, \dots $.}
    \end{array}
\end{equation*}
Besides, we can find sequences $ \{ \delta_n \}_{n = 1}^\infty \subset (0, 1) $ and $ \{ m_n \}_{n = 1}^\infty \subset \N $, such that:
\begin{equation*}
\left\{ \hspace{-2ex} \parbox{9cm}{
\vspace{-2ex}
\begin{itemize}
\item $ 0 < \delta_{n +1} < \delta_n < 2^{-(n +1)} $,
\item $ |v(\tau) -v(\varsigma)|_{L^2(\Omega)} \leq 2^{-(n +1)} $ and
    \\[1ex]
        $ \bigl| [|Dv(\tau)|]_{\gamma}(\overline{\Omega}) -[|Dv(\varsigma)|]_{\gamma}(\overline{\Omega}) \bigr| \leq 2^{-(n +1)} $,
    \\[1ex]
    if $ \tau, \varsigma \in K_n $ and $ |\tau -\varsigma| \leq \delta_n $,
\item $ 2^n \leq m_n < m_{n +1}  $ and $ 2^{-m_n} \mathcal{L}^1(I) \leq \delta_n $,
\vspace{-2ex}
\end{itemize}
} \right. \mbox{ for $ n = 1, 2, 3, \dots $.}
\end{equation*}
Here, let us define a partition of time-interval by putting:
\begin{equation*}
t_\ell^n := \inf I +\ell 2^{-m_n} \mathcal{L}^1(I), \mbox{ for $ \ell = 0, 1, 2, \dots, 2^{m_n} $, $ n = 1, 2, 3, \dots $.}
\end{equation*}
Also, for any $ n \in \N $, let us define:
\begin{align*}
& K_{\ell}^n := K_n \cap [t_{\ell -1}^n, t_\ell^n), \mbox{ for $ \ell = 1, \dots, 2^{m_n} $,}
\\[1ex]
& L_n := \left\{ \begin{array}{l|l}
    \ell \in \N & 1 \leq \ell \leq 2^{m_n} \mbox{ and } K_{\ell}^n \ne \emptyset
\end{array} \right\},
    \\[1ex]
    & \bar{\sigma}_\ell^n := \min K_\ell^n \mbox{ for $ \ell \in L_n $.}
\end{align*}
Then, since
\begin{center}
    $ \gamma_n \in H^1(\Gamma) $ and $ \bar{\sigma}_\ell^n \in K_\ell^n $, for all $ n \in \N $ and $ \ell \in L_n $,
\end{center}
we can apply (\hyperlink{Fact4}{Fact\,4}). Thus, we find a sequence of functions $ \bigl\{ \bar{\varphi}_{\ell}^n \bigl| n \in \N, ~ \ell = 1, \dots, 2^{m_n}  \bigr\} \subset H^1(\Omega) $, such that:
\begin{equation*}
\begin{array}{c}
\left\{ \begin{array}{l}
    \ds  |\bar{\varphi}_{\ell}^n -v(\bar{\sigma}_{\ell}^n)|_{L^2(\Omega)} \leq 2^{-(n +1)},
\\[1ex]
    \ds \left| \int_\Omega |\nabla \bar{\varphi}_{\ell}^n| \, dx -\int_{\overline{\Omega}} d \bigl[ |D v(\bar{\sigma}_{\ell}^n)| \bigr]_{\gamma_n} \right| \leq 2^{-(n +1)},
\end{array} \right.
\quad
\mbox{for any $ n \in \N $ and any $ \ell \in L_n $.}
\end{array}
\end{equation*}

In view of these, we define a sequence $ \{ v_n \}_{n = 1}^\infty \subset L^\infty(I; H^1(\Omega)) $, as follows:
\begin{equation*}
v_n(t) := \left\{ \begin{array}{l}
\bar{\varphi}_{\ell}^n, \mbox{ if $ t \in K_{\ell}^n $ with some $ \ell \in L_n $,}
\\[1ex]
[\gamma_n]^{\rm hm}, \mbox{ otherwise,}
\end{array} \right. \mbox{ for $ n = 1, 2, 3, \dots $.}
\end{equation*}
Then, we immediately observe that this sequence $ \{ v_n \}_{n = 1}^\infty $ satisfies the assertion \eqref{v_n^circ01}. Meanwhile, by \eqref{convGamma}, the compactness of embedding $ H^{\frac{1}{2}}(\Gamma) \subset L^1(\Gamma) $ leads to:
\begin{equation*}
    \gamma_n \to \gamma \mbox{ in $ L^1(\Gamma) $, as $ n \to \infty $.}
\end{equation*}
the remaining assertion \eqref{v_n^circ02} can be verified as follows:
\begin{align*}
|v_n - & v|_{L^2(I; L^2(\Omega))}^2
 \\ \leq &  2 \sum_{\ell = 1}^{m_n} \int_{K_{\ell}^n} \bigl( |\bar{\varphi}_{j, \ell}^n -v(\bar{\sigma}_{j, \ell}^n)|_{L^2(\Omega)}^2 +|v(\bar{\sigma}_{j, \ell}^n) -v(t)|_{L^2(\Omega)}^2 \bigr) \, dt
\\
 & +  2 \int_{I \setminus K_n} \bigl(|v(t)|_{L^2(\Omega)}^2  +|[\gamma_n]^{\rm hm}|^2 \bigr) \, dt
\\
 \leq & 2^{-2n} \mathcal{L}^1(I)  +2 |v|_{L^2(I \setminus K_n; L^2(\Omega))}^2
\\  & + 2 |[\gamma_n]^\mathrm{hm}|_{H^{1}(\Omega)}^2 \cdot \frac{2^{-2(n +1)}}{\bigl( 1 +\mathcal{L}^N(\Omega) \bigr)^2 \bigl( 1 +C_*^\mathrm{hm}\,\bar{\gamma}_* \bigr)^2}
\\
    \leq & 2^{-2n} (1 +\mathcal{L}^1(I)) \to 0, \mbox{ as $ n \to \infty $,}
\end{align*}
and
\begin{align*}
    \int_I \left| \int_\Omega \right. & |\nabla v_n(t)| \, dx -\left. \int_{\overline{\Omega}} d \bigl[ |D v(t)| \bigr]_{\gamma}  \right| \, dt
\\
    \leq & \sum_{\ell = 1}^{m_n} \int_{K_{\ell}^n} \left| \int_\Omega |\nabla \bar{\varphi}_{\ell}^n| \, dx -\int_{\overline{\Omega}} d \bigl[ |D v(\bar{\sigma}_{\ell}^n)| \bigr]_{\gamma_n} \right| \, dt
\\
    & +\sum_{\ell = 1}^{m_n} \int_{K_{\ell}^n} \left| \bigl[ |Dv(\bar{\sigma}_{\ell}^n)| \bigr]_{\gamma_n}(\overline{\Omega}) -\bigl[ |Dv(\bar{\sigma}_{\ell}^n)| \bigr]_{\gamma}(\overline{\Omega}) \right| \, dt
\\
    & +\sum_{\ell = 1}^{m_n} \int_{K_{\ell}^n} \left| \bigl[ |Dv(\bar{\sigma}_{\ell}^n)| \bigr]_{\gamma}(\overline{\Omega}) -\bigl[ |Dv(t)| \bigr]_{\gamma}(\overline{\Omega}) \right| \, dt
\\
    & +\int_{I \setminus K_n} \left( \bigl[ |D v(t)| \bigr]_{\gamma}(\overline{\Omega}) +\int_\Omega |\nabla [\gamma]^{\rm hm}| \, dx \right) \, dt
\end{align*}
\begin{align*}
    \leq & 2^{-n} \mathcal{L}^1(I) +\sum_{\ell = 1}^{m_n} \left| |v(\bar{\sigma}_\ell^n)_{|_\Gamma} -\gamma_n|_{L^1(\Gamma)} -|v(\bar{\sigma}_\ell^n)_{|_\Gamma} -\gamma|_{L^1(\Gamma)} \right| \, dt
    \\
    &  +2^{-(n +1)} +\mathcal{L}^N(\Omega)^{\frac{1}{2}} |[\gamma]^\mathrm{hm}|_{H^1(\Omega)} \cdot \frac{2^{-(n +1)}}{\bigl( 1 +\mathcal{L}^N(\Omega) \bigr) \bigl( 1 +C_*^\mathrm{hm}\,\bar{\gamma}_* \bigr)}
\\
    \leq &  \bigl( 1 +\mathcal{L}^1(I) \bigr) \bigl( 2^{-n} +|\gamma_n -\gamma|_{L^1(\Gamma)}\bigr) \to 0, \mbox{ as $ n \to \infty $.}
\end{align*}

Thus, we conclude this Lemma.
\hfill $ \Box $
}
{
\begin{rem}\label{axRem_reg}
     Concerning the approximating sequence $ \{ v_n \}_{n = 1}^\infty \subset L^\infty(I; H^1(\Omega)) $ obtained in Lemma \ref{axLem_reg}, by \eqref{v_n^circ01} and \eqref{v_n^circ02} we can suppose the following pointwise convergences:
    \begin{equation}\label{v_n^circ03}
        \begin{array}{c}
            v_n(t) \to v(t) \mbox{ in $ L^2(\Omega) $ and } [|Dv_n(t)|]_{\gamma_n}(\overline{\Omega}) \to [|Dv(t)|]_{\gamma}(\overline{\Omega}),
            \\
            \mbox{for a.e. $ t \in I $, as $ n \to \infty $,}
        \end{array}
    \end{equation}
    by taking a subsequence if necessary.
\end{rem}
}
\bigskip

In what follows, we set, for any open set $ \Omega \subset \R^N $.
\begin{equation*}
\begin{array}{c}
\ds \mathscr{W}_{0}(I; \Omega) := \left\{ \begin{array}{l|l}
\beta \in L^2(I; L^2(\Omega)) &
\parbox{4.25cm}{
$ \beta(t) \in W_0(\Omega) $ a.e. $ t \in I $
}
\end{array} \right\},
\\[1ex]
\ds \mathscr{W}_{\rm c}(I; \Omega) := \left\{ \begin{array}{l|l}
\beta \in L^2(I; L^2(\Omega)) &
\parbox{4.25cm}{
$ \beta(t) \in W_{\rm c}(\Omega) $ a.e. $ t \in I $
}
\end{array} \right\}.
\end{array}
\end{equation*}

{On this basis, for any $\beta \in \mathscr{W}_{0}(I;\Omega)$ and any $ \gamma \in H^{\frac{1}{2}}(\Gamma) $, we define a functional $ \Phi_{\gamma}^I(\beta;{}\cdot{}) $ on $L^{2}(I;L^{2}(\Omega))$ by letting:
\begin{equation}\label{2.4-1}
    v \in L^{2}(I;L^{2}(\Omega)) \mapsto \Phi_{\gamma}^I(\beta; v) := \left \{
\begin{array}{ll}
    \ds \int_{I} \Phi_{\gamma}(\beta(t) ; v(t))\, dt,
    & \mbox{if } \Phi_{\gamma}(\beta; v) \in L^{1}(I),
\\[2ex]
\ds \infty, & \mbox{otherwise,}
\end{array}
\right.
\end{equation}
and additionally, we define a class $ \{ \Phi_{\nu, \gamma}^I(\beta;{}\cdot{}) \}_{\nu \in (0, 1)} $ of relaxed functionals of $ \Phi_{\gamma}^I(\beta;{}\cdot{}) $, by letting:
\begin{equation*}
    v \in L^{2}(I;L^{2}(\Omega)) \mapsto \Phi_{\nu, \gamma}^I(\beta; v) := \left \{
\begin{array}{ll}
    \ds \int_{I} \Phi_{\gamma}^\nu(\beta(t) ; v(t))\, dt,
    & \mbox{if } v \in L^{2}(I; H^1(\Omega)),
\\[2ex]
\ds \infty, & \mbox{otherwise.}
\end{array}
\right.
\end{equation*}
As is easily checked, the functionals $ \Phi_{\nu, \gamma}^I(\beta;{}\cdot{}) $, for every $ \beta \in \mathscr{W}_0(I; \Omega) $, $ \gamma \in H^{\frac{1}{2}}(\Gamma) $, and $ \nu \in (0, 1) $, are proper l.s.c. and convex on $ L^2(I; L^2(\Omega)) $.
}
\medskip

{
Now, for mathematical analysis under time-dependent situations, we prepare some key-Lemmas.
}

{
\begin{lem}\label{baseA}
    Let $ \gamma \in H^{\frac{1}{2}}(\Gamma) $ and $ \beta \in \mathscr{W}_0(I; \Omega) $. Then, the following items hold:
\begin{description}
        \hypertarget{baseA(I)}
    \item[\textmd{\it ({\,I\,})}]If $ v \in L^2(I; L^2(\Omega)) $ and $ |D v({}\cdot{})|(\Omega) \in L^1(I) $, then the function $ t \in I \mapsto $ \linebreak $ [\beta(t) |Dv(t)|]_{\gamma}(\overline{\Omega}) \in [0, \infty) $ is measurable.
        \hypertarget{baseA(II)}
\item[\textmd{\it ({II})}]If $ \beta \in \mathscr{W}_{\rm c}(I; \Omega) $ and  $ \log \beta \in  L^\infty(I \times \Omega) $, then $ \Phi^{I}_\gamma(\beta;{}\cdot{}) $ is a proper, l.s.c. and convex function on $ L^2(I; L^2(\Omega)) $, and
\begin{equation}\label{dom}
    D(\Phi^{I}_{\gamma}(\beta;{}\cdot{})) = \left\{ \begin{array}{l|l}
\tilde{v} \in L^2(I; L^2(\Omega)) & |D \tilde{v}({}\cdot{})|(\Omega) \in L^1(I)
\end{array} \right\}.
\end{equation}
\end{description}
\end{lem}
}
{
\noindent
\textbf{Proof.}
First, we verify (I). We take approximating sequences $ \{ \gamma_n \}_{n = 1}^\infty \subset H^1(\Gamma) $ and $ \{ v_n \}_{n = 1}^\infty \subset L^\infty(I; H^1(\Omega)) $, as in Lemma \ref{axLem_reg} and Remark \ref{axRem_reg}. Then, from \eqref{v_n^circ03}, we have
\begin{align*}
    \ds [\beta(t) |Dv(t)|]_{\gamma}(\overline{\Omega})
&~ \ds = \lim_{n \to \infty}  \int_\Omega \beta(t) |\nabla v_n(t)| \, dx, \mbox{ \ for a.e. $ t \in I $.}
\end{align*}
This implies the validity of (\,I\,).

Meanwhile, in the light of \eqref{01_[bt|Du|]_gm}, \eqref{02_[bt|Du|]_gm}, and (\ref{Phi_gm(bt.)}), we immediately verify the assertions in (II) except for the lower semi-continuity of $ \Phi^{I}_{\gamma}(\beta;{}\cdot{}) $.

Now, for the verification of lower semi-continuity, we take any $ w \in D(\Phi^{I}_{\gamma}(\beta;{}\cdot{})) $ and any sequence $ \{ w_n \}_{n = 1}^\infty \subset D(\Phi^{I}_{\gamma}(\beta;{}\cdot{})) $ such that $ w_n \to w $ in $ L^2(I; L^2(\Omega)) $ as $ n \to \infty $. Then, applying  \cite[Lemma 4.2]{MR3268865}, we immediately see the lower semi-continuity of the functional:
\begin{equation*}
\tilde{v} \in L^2(I; L^2(\Omega)) \mapsto \int_I [\beta]^{\rm ex}(t)|D[\tilde{v}(t)]_\gamma^{\rm ex}|(\bB_{\Omega}) \, dt. \
\end{equation*}
Having these in mind, the lower semi-continuity of $ \Phi^{I}_{\gamma}(\beta;{}\cdot{}) $ is verified as follows.
\begin{eqnarray*}
    \liminf_{n \to \infty} \Phi^{I}_{\gamma}(\beta; w_n) & = & \liminf_{n \to \infty} \int_I [\beta(t)]^{\rm ex}|D[
\omega_n(t)]_\gamma^{\rm ex}|(\bB_\Omega) \, dt
\\
&& \qquad - \int_I [\beta(t)]^{\rm ex} d |\nabla [\gamma]^{\rm ex}|(\bB_\Omega \setminus \Omega) \, dt
\\
& \geq & \int_I [\beta(t)]^{\rm ex}|D[w(t)]_\gamma^{\rm ex}|(\bB_\Omega) \, dt
\\
&& \qquad - \int_I [\beta(t)]^{\rm ex} d |\nabla [\gamma]^{\rm ex}|(\bB_\Omega \setminus \Omega) \, dt
\\
    & = & \Phi^{I}_{\gamma}(\beta; w).
\end{eqnarray*}
$ \Box $
}
{
\begin{lem}\label{Lem.LB}
    Let $ \beta \in \mathscr{W}_{0}(I; \Omega)\cap C(\overline I;L^2(\Omega)) $, $ \{ \beta_n \}_{n = 1}^\infty \subset \mathscr{W}_0(I; \Omega)\cap C(\overline I;L^2(\Omega)) $, $ \gamma \in H^{\frac{1}{2}}(\Gamma) $ and $ \{ \gamma_n \}_{n = 1}^\infty \subset H^{\frac{1}{2}}(\Gamma) $ be such that
\begin{equation}\label{core01}
\left\{ \parbox{10cm}{
\begin{tabular}{ll}
$ \beta_n \to \beta $ & in $ L^2(I; L^2(\Omega)) $,  weakly in $ W^{1, 2}(I; L^2(\Omega)) $,
\\
& and weakly in $ L^2(I; H^1(\Omega)) $,
\\[1ex]
$ \gamma_n \to \gamma $ & in $ H^{\frac{1}{2}}(\Gamma) $,
\end{tabular}
} \right.
\mbox{as $ n \to \infty $.}
\end{equation}
Let $ v \in L^2(I; L^2(\Omega)) $ and $ \{ v_n \}_{n = 1}^\infty \subset L^2(I; L^2(\Omega)) $ be such that
\begin{equation}\label{core02}
\left\{ ~ \parbox{12.75cm}{
$ |Dv({}\cdot{})|(\Omega) \in L^1(I) $, $ \{ |D v_n({}\cdot{})|(\Omega) \}_{n = 1}^\infty \subset L^1(I) $,
\\[1ex]
$ v_n(t) \to v(t) $ in $ L^2(\Omega) $ and weakly-$*$ in $ BV(\Omega) $, a.e. $ t \in I $, as $ n \to \infty $.
} \right.
\end{equation}
Here, if:
\begin{equation}\label{core03}
\beta \in \mathscr{W}_{\rm c}(I; \Omega), \mbox{ or } L_* := \sup_{n \in \N}  \bigl| |D v_n({}\cdot{})|(\Omega) \bigr|_{L^1(I)} < \infty,
\end{equation}
then:
\begin{equation*}
\liminf_{n \to \infty} \int_I [ \beta_n(t) |Dv_n(t)| ]_{\gamma_n}(\overline{\Omega}) \, dt \geq \int_I [ \beta(t) |Dv(t)| ]_{\gamma}(\overline{\Omega}) \, dt.
\end{equation*}
\end{lem}
}
{
\noindent
\textbf{Proof. }
It is enough to consider only the case when:
\begin{equation*}
\Phi_* := \liminf_{n \to \infty} \int_I [ \beta_n(t) |Dv_n(t)| ]_{\gamma_n}(\overline{\Omega}) \, dt < \infty,
\end{equation*}
because the other case is trivial.

In this case, we find a subsequence $ \{ n_k \}_{k = 1}^\infty \subset \{n\}_{n = 1}^\infty $, such that:
\begin{equation*}
\Phi_*  = \lim_{k \to \infty} \int_I [ \beta_{n_k}(t) |Dv_{n_k}(t)| ]_{\gamma_{n_k}}(\overline{\Omega}) \, dt.
\end{equation*}
Additionally, by Remark \ref{Rem.ext} and the compactness theory of Aubin's type \cite{MR0916688},  assumptions \eqref{core01}--\eqref{core03} enable us to suppose the following properties for the subsequences:
\begin{equation}\label{core06}
\left\{ \parbox{8.75cm}{
\begin{tabular}{lll}
$ [\beta_{n_k}]^{\rm ex} \to [\beta]^{\rm ex} $ & \hspace{-1ex}in& \hspace{-1ex}$ L^2(I; L^2(\bB_\Omega)) $,
\\
\hfill weakly  & \hspace{-1ex}in& \hspace{-1ex}$ W^{1, 2}(I; L^2(\R^N)) $,
\\
\hfill and weakly & \hspace{-1ex}in & \hspace{-1ex}$ L^2(I; H^1(\R^N) $,
\\[1ex]
    $ [\gamma_{n_k}]^{\rm ex} \to [\gamma]^{\rm ex} $ & \hspace{-1ex}in & \hspace{-1ex}$ W^{1, 1}(\bB_\Omega) $, and in $ H^1(\R^N) $,
\end{tabular}
} \right. \mbox{as $ k \to \infty $,}
\end{equation}
\begin{equation}\label{core07}
\left\{ \parbox{7.25cm}{
\begin{tabular}{lll}
$ [\beta_{n_k}(t)]^{\rm ex} \to [\beta(t)]^{\rm ex} $ & \hspace{-1ex}in& \hspace{-1ex}$ L^2(\bB_\Omega) $
\\
\hfill and weakly & \hspace{-1ex}in& \hspace{-1ex}$ H^1(\R^N) $,
\\[1ex]
$ [v_{n_k}(t)]_{\gamma_{n_k}}^{\rm ex} \to [v(t)]_{\gamma}^{\rm ex} $ & \hspace{-1ex}in& \hspace{-1ex}$ L^2(\bB_\Omega) $,
\end{tabular}
} \right. \mbox{as $ k \to \infty $, for a.e. $ t \in I $,}
\end{equation}
and
\begin{equation}\label{core08}
\left\{ \hspace{-2.5ex} \parbox{10.25cm}{
\vspace{-2ex}
\begin{itemize}
\item $ [\beta]^{\rm ex} \in \mathscr{W}_{\rm c}(I; \R^N) $, if $ \beta \in \mathscr{W}_{\rm c}(I; \Omega) $,
\item $ L_*^{\rm ex} := \ds \sup_{n \in \N} \, \bigl| |D[v_{n}(\cdot)]_{\gamma_{n}}^{\rm ex}|(\bB_\Omega) \bigr|_{L^1(I)} < \infty $, if $ L_* < \infty $.
\vspace{-2ex}
\end{itemize}
} \right.
\end{equation}
Taking into account \eqref{02_[bt|Du|]_gm} and \eqref{core06}--\eqref{core08}, we can apply \cite[Lemma 4.3]{MR3670006} to deduce that:
\begin{eqnarray*}
\Phi_* & = & \liminf_{k \to \infty} \int_I [\beta_{n_k}(t)]^{\rm ex}\bigl| D[{v}_{n_k}(t)]_{\gamma_{n_k}}^{\rm ex} \bigr|(\bB_\Omega) \, dt
\\
&& \qquad -\lim_{k \to \infty} \int_I [\beta_{n_k}(t)]^{\rm ex} d |\nabla [\gamma_{n_k}]^{\rm ex}|(\bB_\Omega \setminus \Omega) \, dt
\\
& \geq & \int_I [\beta(t)]^{\rm ex}\bigl| D[{v}(t)]_{\gamma}^{\rm ex} \bigr| (\bB_\Omega) \, dt
\\
&& \qquad - \int_I [\beta(t)]^{\rm ex} d |\nabla [\gamma]^{\rm ex}|(\bB_\Omega \setminus \Omega) \, dt
\\
& = & \int_I [\beta(t) |Dv(t)|]_{\gamma}(\overline{\Omega}) \, dt.
\end{eqnarray*}
$ \Box $
}

{
\begin{lem}\label{Lem.core}
    Let $ \beta \in \mathscr{W}_{\rm c}(I; \Omega)\cap C(\overline I;L^2(\Omega)) $, $ \{ \beta_n \}_{n = 1}^\infty \subset \mathscr{W}_{c}(I; \Omega)\cap C(\overline I;L^2(\Omega)) $, $ \gamma \in H^{\frac{1}{2}}(\Gamma) $, $ \{ \gamma_n \}_{n = 1}^\infty \subset H^{\frac{1}{2}}(\Gamma) $, $ v \in L^2(I; L^2(\Omega)) $ and $ \{ v_n \}_{n = 1}^\infty \subset L^2(I; L^2(\Omega)) $ be such that the conditions \eqref{core01} and \eqref{core02} are fulfilled, and
\begin{equation}\label{core10}
\beta \geq \delta_0 \mbox{ and } \beta_n \geq \delta_0 \mbox{ a.e. in $ I \times \Omega $, for some constant $ \delta_0 > 0 $.}
\end{equation}
Let $ \varrho \in L^\infty(I; H^1(\Omega)) \cap L^\infty(I \times \Omega) $ and $ \{ \varrho_n \}_{n = 1}^\infty \subset L^\infty(I; H^1(\Omega)) \cap L^\infty(I \times \Omega) $ be such that
\begin{equation}\label{core11}
\begin{array}{ll}
\varrho_n \to \varrho & \mbox{in $ L^2(I; L^2(\Omega)) $, weakly in $ W^{1, 2}(I; L^2(\Omega)) $,}
\\
& \mbox{and weakly in $ L^2(I; H^1(\Omega)) $, \ as $ n \to \infty $,}
\end{array}
\end{equation}
and
\begin{equation*}
M_0 := \sup \left\{ \begin{array}{l|l}
|\varrho|_{L^\infty(I \times \Omega)}, |\varrho_n|_{L^\infty(I \times \Omega)} & n \in \N
\end{array} \right\} < \infty.
\end{equation*}
Besides, let us assume that
\begin{equation}\label{core13}
\int_I [\beta_n(t)|D v_n(t)|]_{\gamma_n}(\overline{\Omega}) \, dt \to \int_I [\beta(t)|D v(t)|]_{\gamma}(\overline{\Omega}) \, dt, \mbox{ as $ n \to \infty $.}
\end{equation}
Then, it holds that:
\begin{equation}\label{core14}
\int_I [\varrho_n(t)|D v_n(t)|]_{\gamma_n}(\overline{\Omega}) \, dt \to \int_I [\varrho(t)|D v(t)|]_{\gamma}(\overline{\Omega}) \, dt, \mbox{ as $ n \to \infty $.}
\end{equation}
\end{lem}
}
{
\noindent
\textbf{Proof. } This lemma is proved by using the following elementary fact (\cite[Proposition 1.80]{MR1857292})
\begin{enumerate}
        \hypertarget{(*)}{}
\item[$(*)$]if $ a_0 \in \R $, $ b_0 \in \R $, $ \{ a_n \}_{n = 1}^\infty \subset \R $ and $ \{ b_n \}_{n = 1}^\infty \subset \R $ fulfill that:
\begin{equation}\label{*1}
\liminf_{n \to \infty} a_n \geq a_0 \mbox{ \ and \ } \liminf_{n \to \infty} b_n \geq b_0,
\end{equation}
and
\begin{equation}\label{*2}
\limsup_{n \to \infty} (a_n +b_n) \leq a_0 +b_0,
\end{equation}
then \ $ a_n \to a_0 $ and $ b_n \to b_0 $ as $ n \to \infty $.
\end{enumerate}

With the above fact in mind, let us put:
\begin{equation*}
\left\{ \begin{array}{l}
\ds a_0 := \int_I [\varrho(t)|D v(t)|]_{\gamma}(\overline{\Omega}) \, dt,
\\[2ex]
\ds b_0 := \int_I \bigl[ \bigl( {\ts \frac{M_0}{\delta_0} \beta(t) -\varrho(t)} \bigr)|D v(t)|\bigr]_{\gamma} \hspace{0ex} (\overline{\Omega}) \, dt,
\end{array} \right.
\end{equation*}
and
\begin{equation*}
\left\{ \begin{array}{l}
\ds a_n := \int_I [\varrho_n(t)|D v_n(t)|]_{\gamma_n}(\overline{\Omega}) \, dt,
\\[2ex]
\ds b_n := \int_I \bigl[ \bigl( {\ts \frac{M_0}{\delta_0} \beta_n(t) -\varrho_n(t)} \bigr)|D v(t)|\bigr]_{\gamma_n} \hspace{0ex} (\overline{\Omega}) \, dt,
\end{array} \right.
\mbox{for $ n = 1, 2, 3, \dots $.}
\end{equation*}
Then, by \eqref{core01}, \eqref{core02}, \eqref{core10}, and \eqref{core11}, we verify the condition \eqref{*1} by applying Lemma \ref{Lem.LB}. Also, on account of \eqref{Phi_gm(bt.)} and \eqref{core13}, condition \eqref{*2} can be verified as follows:
\begin{eqnarray*}
\limsup_{n \to \infty} (a_n +b_n) & = & \frac{M_0}{\delta_0} \lim_{n \to \infty} \int_I [\beta_n(t)|D v_n(t)|]_{\gamma_n}(\overline{\Omega}) \, dt
\\
& = & \frac{M_0}{\delta_0} \int_I [\beta(t)|D v(t)|]_{\gamma}(\overline{\Omega}) \, dt = a_0 +b_0.
\end{eqnarray*}
In view of these, \eqref{core14} is obtained as a straightforward consequence of  the fact $ (\hyperlink{(*)}{*}) $. \hfill $ \Box $
}
\bigskip

{
As a consequence of Lemmas \ref{Lem.LB} and \ref{Lem.core}, we can show the following Corollaries.
}
{
\begin{cor}\label{Cor.Gamma-conv}
If $ \beta \in \mathscr{W}_{\rm c} (I; \Omega)\cap C(\overline I;L^2(\Omega)) $, $ \{ \beta_n \}_{n = 1}^\infty \subset \mathscr{W}_{\rm c}(I; \Omega)\cap C(\overline I;L^2(\Omega)) $, $ \gamma \in H^{\frac{1}{2}}(\Gamma) $, and $ \{ \gamma_n \}_{n = 1}^\infty \subset H^{\frac{1}{2}}(\Gamma) $ fulfill that:
\begin{align}\label{coreG01}
    \beta_n \to & \beta \mbox{ in $ L^2(I; L^2(\Omega)) $, }
    \mbox{weakly in $ W^{1, 2}(I; L^2(\Omega)) $,}
    \nonumber
    \\
    & \mbox{weakly in $ L^2(I; H^1(\Omega)) $, }
    \mbox{and weakly-$*$ in $ L^\infty(I \times \Omega) $,  as $ n \to \infty $.}
\end{align}
and
\begin{equation}\label{coreG02}
    \gamma_n \to \gamma \mbox{ in $ H^{\frac{1}{2}}(\Gamma) $, as $ n \to \infty $,}
\end{equation}
then the sequence of convex functions $ \{ \Phi^{I}_{\gamma_n}(\beta_n;{}\cdot{}) \}_{n = 1}^\infty $ $ \mathit{\Gamma} $-converges to the convex function $ \Phi^{I}_{\gamma}(\beta;{}\cdot{}) $ on $ L^2(I; L^2(\Omega)) $, as $ n \to \infty $.
\end{cor}
}
{
\noindent
\textbf{Proof.}
The \hyperlink{Mosco}{lower bound} condition is a straightforward consequence of Lemma \ref{Lem.LB}. Also, for any $ {v} \in D(\Phi^I_{\gamma}(\beta;{}\cdot{})) $, we can take $ \{v_n:= {v} \} $ as the sequence in the \hyperlink{Mosco}{optimality} condition, by applying Lemma \ref{Lem.core}, under the replacements:
\begin{equation*}
    \begin{cases}
        \parbox{10cm}{
            $ \beta $ by the constant $ 1 $, $ \{ \beta_n \}_{n = 1}^\infty $ by the singleton $ \{ 1 \} $,
            \\[1ex]
            $ \varrho $ by $ \beta $, and $ \{ \varrho_n \}_{n = 1}^\infty  $ by $ \{ \beta_{n} \}_{n = 1}^\infty $.
        }
    \end{cases}
\end{equation*}
$ \Box $
}
{
\begin{cor}\label{Cor.Gamma-conv_nu}
    In addition to the assumptions and notations as in Corollary \ref{Cor.Gamma-conv}, let $ \{ \nu_n \}_{n = 1}^\infty$ $\subset (0, 1) $ be a sequence such that $ \nu_n \downarrow 0 $ as $ n \to \infty $. Then, the following items hold.
\begin{description}
        \hypertarget{Cor.Gamma-conv_nu(I)}
    \item[\textmd{\em (I)}]$ \liminf_{n \to \infty} \Phi^I_{\nu_n, \gamma_n}(\beta_n; v_n) \geq \Phi^I_{\gamma}(\beta; v) $, if $ v \in L^2(I; L^2(\Omega)) $, $ \{ v_n \}_{n = 1}^\infty \subset L^2(I; L^2(\Omega)) $, and $ v_n \to v $ in $ L^2(I; L^2(\Omega)) $ as $ n \to \infty $.
        \hypertarget{Cor.Gamma-conv_nu(II)}
    \item[\textmd{\em (II)}]If $ \gamma \in H^{1}(\Gamma) $, then the sequence $ \{ \Phi^{I}_{\nu_n, \gamma}(\beta_n;{}\cdot{}) \}_{n = 1}^\infty $ of convex functions $ \mathit{\Gamma} $-converges to the convex function $ \Phi^{I}_{\gamma}(\beta;{}\cdot{}) $ on $ L^2(I; L^2(\Omega)) $, as $ n \to \infty $.
\end{description}
\end{cor}
}
{
\noindent
\textbf{Proof.}
By Corollary \ref{Cor.Gamma-conv}, item (I) is immediately obtained as follows:
\begin{equation*}
\begin{array}{c}
    \ds \liminf_{n \to \infty} \Phi^I_{\nu_n, \gamma_n}(\beta_n; v_n) \geq \liminf_{n \to \infty} \Phi^I_{\gamma_n}(\beta_n; v_n) \geq \Phi^I_{\gamma}(\beta; v),
\\[1ex]
\mbox{whenever $ v \in L^2(I; L^2(\Omega)) $, $ \{ v_n \}_{n = 1}^\infty \subset L^2(I; L^2(\Omega)) $}
    \\[1ex]
\mbox{and $ v_n \to v $ in $ L^2(I; L^2(\Omega)) $ as $ n  \to \infty $.}
\end{array}
\end{equation*}

Next, we show (II). Due to (I), it is sufficient to verify the condition of \hyperlink{Mosco}{optimality} in the $ \Gamma $-convergence of $ \{ \Phi^I_{\nu_n, \gamma}(\beta_n;{}\cdot{}) \}_{n = 1}^\infty $. Given any $ \omega \in D(\Phi_I(\gamma, \beta;{}\cdot{})) $, owing to \eqref{dom}, $ \gamma \in H^1(\Gamma) $, and Lemma \ref{axLem_reg}, we can take  $ \{ \omega_n \}_{n = 1}^\infty \subset L^2(I; H^1(\Omega)) $, such that:
\begin{equation}\label{Cor.Gamma01-00}
    \omega_n(t) = \gamma \mbox{ in $ H^1(\Gamma) $, $ n = 1, 2, 3, \dots $,}
\end{equation}
and
\begin{equation}\label{Cor.Gamma00}
    \begin{array}{c}
        \ds \omega_n \to \omega \mbox{ in $ L^2(I; L^2(\Omega)) $ and } \int_I \int_\Omega d |D {\omega}_n| \, dt \to \int_I \int_\Omega d |D \omega| \, dt,
        \quad
        \mbox{as $ n \to \infty $.}
    \end{array}
\end{equation}

On this basis, let us take a subsequence $ \{ n_k \}_{k = 1}^\infty \subset \{ n \} $ such that:
\begin{equation}\label{Cor.Gamma02}
\frac{\nu_{n_k}}{2} \int_I \int_\Omega |\nabla \omega_{k}(t)|^2 \, dx dt \leq 2^{-k}, \mbox{ \ for $ k = 1, 2, 3, \dots $,}
\end{equation}
and define a sequence $ \{ \check{\omega}_n \}_{n = 1}^\infty \subset L^2(I; H^1(\Omega)) $ as follows:
\begin{equation}\label{Cor.Gamma03}
\check{\omega}_n(t) := \begin{cases}
\omega_{k}(t) \mbox{ in $ H^1(\Omega) $, if $ n_{k} \leq n < n_{k +1} $,}
\\
\omega_{1}(t) \mbox{ in $ H^1(\Omega) $, if $ 1 \leq n < n_{1} $,}
\end{cases} \mbox{for $ n = 1, 2, 3\dots $.}
\end{equation}
Then, from \eqref{Cor.Gamma03} and Lemma \ref{axLem_reg}, it follows that:
\begin{equation}\label{Cor.Gamma04}
\check{\omega}_n \to \omega \mbox{ in $ L^2(I; L^2(\Omega)) $ and $ \nu_n |\nabla \check{\omega}_n|_{L^2(\Omega)^N}^2 \to 0 $, as $ n \to \infty $.}
\end{equation}
Also, by \eqref{coreG01}, \eqref{Cor.Gamma00}--\eqref{Cor.Gamma03}, we can apply Lemma \ref{Lem.core} in the case that
\begin{equation*}
    \begin{array}{c}
        \beta = 1, ~ \{ \beta_n \}_{n = 1}^\infty = \{1\}, ~ \gamma = \gamma, ~ \{ \gamma_n \}_{n = 1}^\infty = \{ \gamma \},
                \varrho = \beta, \mbox{ and } \{ \varrho_n \}_{n = 1}^\infty = \{ \beta_n \}_{n = 1}^\infty,
    \end{array}
\end{equation*}
and we see that
\begin{align}\label{Cor.Gamma01}
    & \int_I \int_\Omega d \bigl[ \beta_n(t) |D \check{\omega}_n(t)| \bigr]_{\gamma} \to \int_I \int_\Omega d \bigl[ \beta(t) |D \omega(t)| \bigr]_{\gamma}, \mbox{ as $ n \to \infty $.}
\end{align}
Furthermore, from \eqref{Cor.Gamma01-00}, it immediately follows that:
\begin{align}\label{Cor.Gamma05}
    \check{v}_n \in D(\Phi^I_{\nu_n, \gamma}(\beta_n;{}\cdot{})), \mbox{ for $ n = 1, 2, 3, \dots $.}
\end{align}
On account of \eqref{Cor.Gamma04}--\eqref{Cor.Gamma05}, we verify that the sequence $ \{ \check{\omega}_n \}_{n = 1}^\infty $ satisfies the \hyperlink{Mosco}{optimality} condition. Thus we conclude this Lemma. \hfill $ \Box $
}
{
\begin{rem}\label{Rem.G-conv01}
Applying Corollaries \ref{Cor.Gamma-conv} and \ref{Cor.Gamma-conv_nu} to the case when $ I = (0, 1) $, and all functions $ \beta $, $ \beta_n $, $ n = 1, 2, 3, \dots $, are independent of time, we easily deduce the following facts.
    \begin{description}
        \hypertarget{Fact6}{}
        \item[{(Fact\,6)}]Let $ \beta \in W_c(\Omega) $ and $ \{ \beta_n \}_{n = 1}^\infty \subset W_0(\Omega) $ be such that:
        \begin{equation}\label{beta^circ}
           \beta_n \to \beta \mbox{ in $ L^2(\Omega) $, weakly in $ H^1(\Omega) $, and weakly-$*$ in $ L^\infty(\Omega) $, as $ n \to \infty $,}
        \end{equation}
            and let $ \gamma \in H^{\frac{1}{2}}(\Gamma) $ and $ \{ \gamma_n \}_{n = 1}^\infty \subset H^{\frac{1}{2}}(\Gamma) $ be as in \eqref{coreG02}. Then the sequence of convex functions $ \{ \Phi_{\gamma_n}(\beta_n;{}\cdot{}) \}_{n = 1}^\infty $ $ \Gamma $-converges to the convex function $ \Phi_\gamma(\beta;{}\cdot{}) $ on $ L^2(\Omega) $, as $ n \to \infty $.
        \hypertarget{Fact7}{}
        \item[{(Fact\,7)}]If $ \beta \in W_c(\Omega) $ and $ \{ \beta_n \}_{n = 1}^\infty \subset W_0(\Omega) $ satisfy \eqref{beta^circ}, $ \gamma  \in H^1(\Gamma) $, and $ \{ \nu_n \}_{n = 1}^\infty \subset (0, 1) $ satisfies $  \nu_n \downarrow 0 $ as $ n \to \infty $,  then the sequence of relaxed convex functions $ \{ \Phi_{\gamma}^{\nu_n}(\beta_n;{}\cdot{}) \}_{n = 1}^\infty $ $ \Gamma $-converges to the convex function $ \Phi_\gamma(\beta;{}\cdot{}) $ on $ L^2(\Omega) $, as $ n \to \infty $.
    \end{description}
\end{rem}
}
\section{Proofs of Theorems \ref{thAp1} and \ref{thAp2}}
\ \ \vspace{-3ex}

In this section, Theorems \ref{thAp1} and \ref{thAp2}, concerned with the key-properties of approximating solutions, will be proved via some auxiliary Lemmas.

The first result deals with solvability of the elliptic problem \eqref{kenApp01}. The proof is analogous to that of \cite[Lemmas 9--11]{MR3268865} and we omit it.

\begin{lem}\label{axLemD1}
    There exist
a (small) positive constant $ h_0 \in (0, 1) $, depending on $ |g|_{W^{1, \infty}(0, 1)} $, $ |\alpha|_{C^1[0, 1]} $ and $ \mathscr{L}^N(\Omega) $ such that, for every $ \nu \in (0, 1) $, $ h \in (0, h_0] $ and $[\eta,\theta]\in H^1(\Omega)^2$, there exists a unique function $ \eta_{h}^{\nu} \in H^1(\Omega) $, such that
\begin{equation*}
0 \leq \eta_{h}^{\nu} \leq 1 \mbox{ a.e. in $ \Omega $.}
\end{equation*}
    and it satisfies
\begin{equation*}
\begin{array}{ll}
\multicolumn{2}{l}{\ds
\frac{1}{h} (\eta_{h}^{\nu} -\eta, w)_{L^2(\Omega)} +(\nabla \eta_{h}^{\nu}, \nabla w)_{L^2(\Omega)^N} +(g(\eta_{h}^{\nu}), w)_{L^2(\Omega)}
}
\\[1ex]
\qquad \qquad & \ds +(\alpha'(\eta_{h}^{\nu}) |\nabla \theta|_\nu, w)_{L^2(\Omega)} = 0, \mbox{ for every $ w \in H^1(\Omega) $.}
\end{array}
\vspace{1ex}
\end{equation*}
\end{lem}

Next, we deal with solvability of \eqref{kenApp02}:
\begin{lem}\label{axLemD2}
For every $ \nu, h \in (0, 1) $ and  $[\eta, \theta] \in {D_{0} \cap H^{1}(\Omega)^{2}}$, there exists a unique function $ \theta_{h}^{\nu} \in H^1(\Omega) $, such that
\begin{equation}\label{axLemD2-02}
|\theta_{h}^{\nu}| \leq {|\gamma|_{L^\infty(\Gamma)}} \mbox{ a.e. in $ \Omega $.}
\end{equation} and satisfying \begin{equation}\label{kenApp02_aux}
    \frac{1}{h} \alpha_0(\eta) (\theta_{h}^{\nu} -\theta) {-\mathrm{div} \jump{\alpha(\eta) \nabla \theta_{h}^{\nu}}_{\gamma}^{\nu}} = 0 \mbox{ in $ L^2(\Omega) $.}
    \end{equation}
\end{lem}

\noindent
\textbf{Proof.} By the regularity of $[\eta,\theta]$, a solution to \eqref{kenApp02_aux} is a minimizer to the following proper, l.s.c., coercive and strictly convex function on $ L^2(\Omega) $:
\begin{equation}\label{axLem2-03}
z \in L^2(\Omega) \mapsto \frac{1}{2h} |\sqrt{{\ts \alpha(\eta)}}(z -\theta)|_{L^2(\Omega)}^2 +{\Phi^{\nu}_{\gamma}(\alpha(\eta); z)} \in [0, \infty].
\end{equation}
Therefore, the existence and uniqueness of $\theta_{h}^{\nu} \in D({\Phi^{\nu}_{\gamma}(\alpha(\eta);{}\cdot{})})$ follows immediately.

Thus, all we have to do is reduced to verify the range constraint property \eqref{axLemD2-02}. To this end, we first note that the constant $ |\gamma|_{L^\infty(\Gamma)} $ is a supersolution to \eqref{kenApp02_aux}; i.e. it satisfies:
\begin{equation}\label{axLem2-04}
{\frac{1}{h} \alpha_0(\eta) (|\gamma|_{L^\infty(\Gamma)} -\theta) -\mathrm{div} \, \jump{\alpha(\eta) \nabla |\gamma|_{L^\infty(\Gamma)}}_{\gamma}^{\nu}} \geq 0 \mbox{ a.e. in $ \Omega $.}
\end{equation}
Besides, we take the difference from \eqref{kenApp02_aux} to \eqref{axLem2-04}, and multiply  both sides by $ [\theta_{h}^{\nu} - {|\gamma|_{L^\infty(\Gamma)}}]^+ $. Then, with (\hyperlink{A2}{A2}) and the trace condition $ \theta_{h, i}^{\nu} = {\gamma} $ in $ H^{\frac{1}{2}}(\Gamma) $ in mind, one can observe that:
\begin{equation}\label{axLem2-05}
|[\theta_{h}^{\nu} - {|\gamma|_{L^\infty(\Gamma)}}]^+|_{L^2(\Omega)}^2 \leq 0, \mbox{ i.e. } \theta_{h}^{\nu} \leq {|\gamma|_{L^\infty(\Gamma)}} \mbox{ a.e. in $ \Omega $,}
\end{equation}
via the following computations:
\begin{equation}\label{axLem2-06}
    [\theta_{h|_\Gamma}^{\nu} - {|\gamma|_{L^\infty(\Gamma)}}]^+ = [{\gamma -|\gamma|_{L^\infty(\Gamma)}}]^+ = 0 \mbox{ in $ H^{\frac{1}{2}}(\Gamma) $,}
\end{equation}
and
{
\begin{align}\label{axLem2-07}
\frac{\delta_\alpha}{h} |[\theta_{h, i}^{\nu} & -|\gamma|_{L^\infty(\Gamma)}]^+|_{L^2(\Omega)}^2
\nonumber
\\
\leq & \int_\Omega \mathrm{div} \, \bigl( \jump{\alpha(\eta_{h, i}^{\nu}) \nabla \theta_{h, i}^{\nu}}_{\gamma}^{\nu} -\jump{\alpha(\eta_{h, i}^{\nu}) \nabla |\gamma|_{L^\infty(\Gamma)}}_{\gamma}^{\nu} \bigr) [\theta_{h, i}^{\nu} -|\gamma|_{L^\infty(\Gamma)}]^+ \, dx
\nonumber
\\
= & -\int_\Omega \alpha(\eta_{h, i}^{\nu}) \bigl( [\nabla |{}\cdot{}|_\nu] ( \nabla [\theta_{h, i}^{\nu} -|\gamma|_{L^\infty(\Gamma)}]^+  ) \bigl) \cdot \nabla [\theta_{h, i}^{\nu} -|\gamma|_{L^\infty(\Gamma)}]^+ \, dx
\nonumber
\\
& -\nu^2 \int_\Omega |\nabla [\theta_{h, i}^{\nu} -|\gamma|_{L^\infty(\Gamma)}]^+|^2 \, dx \leq 0.
\end{align}
}
Analogously, the following {inequality} holds:
\begin{equation}\label{axLem2-09}
|[-{|\gamma|_{L^\infty(\Gamma)}} -\theta_{h}^{\nu}]^+|_{L^2(\Omega)}^2 \leq 0, \mbox{ i.e. } \theta_{h}^{\nu} \geq - {|\gamma|_{L^\infty(\Gamma)}} \mbox{ a.e. in $ \Omega $,}
\end{equation}
by noting that $-{|\gamma|_{L^\infty(\Gamma)}}$ is a subsolution to \eqref{kenApp02_aux} and by similar computations to those in \eqref{axLem2-06} and \eqref{axLem2-07}.

Finally, the range constraint property \eqref{axLemD2-02} is a direct consequence from \eqref{axLem2-05} and \eqref{axLem2-09}. \hfill $ \Box $
\bigskip

\noindent
\textbf{Proof of Theorem \ref{thAp1}.}
Existence and uniqueness of the approximating solution is obtained, immediately, by applying recursively Lemmas \ref{axLemD1} and \ref{axLemD2}. We denote it by $ \{ [\eta_{h, i}^{\nu},\theta_{h, i}^{\nu}] \}_{i = 1}^\infty {\subset D_{0}^\mathrm{rx}} $.
\bigskip

In this light, we here demonstrate only the inequalities \eqref{ene-inq} and \eqref{lem2-inq}. To this end, {we first take $w = \eta_{h, i}^{\nu} -\eta_{h, i -1}^{\nu} $ as test function in \eqref{kenApp01}}. Then, from (\hyperlink{A1}{A1}) and (\hyperlink{A2}{A2}) and Taylor's theorem on $ G $, it is seen that:
\begin{align}\label{thAp1-01}
\frac{1}{h} |\eta_{h, i}^{\nu} - & \eta_{h, i -1}^{\nu}|_{L^2(\Omega)}^2 +\frac{1}{2} \int_\Omega |\nabla \eta_{h, i}^{\nu}|^2 \, dx -\frac{1}{2} \int_\Omega |\nabla \eta_{h, i -1}^{\nu}|^2 \, dx
\nonumber
\\
+ & \int_\Omega \left( G(\eta_{h, i}^{\nu}) -G(\eta_{h, i -1}^{\nu}) -\frac{|g'|_{L^\infty(0, 1)}}{2} |\eta_{h, i}^{\nu} -\eta_{h, i -1}^{\nu}|^2 \right) \, dx
\nonumber
\\
+ & \int_\Omega \bigl( \alpha(\eta_{h, i}^{\nu}) -\alpha(\eta_{h, i -1}^{\nu}) \bigr) {|\nabla \theta_{h, i -1}^{\nu}|_{\nu}} \, dx \leq 0,
\\
& \qquad \mbox{for all $ h \in (0, h_0) $ and $ i = 1, 2, 3, \dots $.}
\nonumber
\end{align}
So, setting:
\begin{equation*}
h_* := \min \left\{ h_0, ~ \frac{1}{1 +|g|_{W^{1, \infty}(0, 1)}} \right\},
\end{equation*}
the inequality \eqref{thAp1-01} can be reduced to:
\begin{align}\label{thAp1-02}
\frac{1}{2h} |\eta_{h, i}^{\nu} - & \eta_{h, i -1}^{\nu}|_{L^2(\Omega)}^2 +\left( \frac{1}{2} \int_\Omega |\nabla \eta_{h, i}^{\nu}|^2 \, dx +\int_\Omega G(\eta_{h, i}^{\nu}) \, dx \right)
\nonumber
\\
- &   \left( \frac{1}{2} \int_\Omega |\nabla \eta_{h, i -1}^{\nu}|^2 \, dx +\int_\Omega G(\eta_{h, i -1}^{\nu})  \, dx \right) {+ \int_\Omega \alpha(\eta_{h, i}^{\nu}) |\nabla \theta_{h, i -1}^{\nu}|_\nu} \, dx
\\
\leq & \int_\Omega \alpha(\eta_{h, i -1}^{\nu}) {|\nabla \theta_{h, i -1}^{\nu}|_{\nu}} \, dx, \mbox{ for all $ h \in (0, h_*) $ and $ i = 1, 2, 3, \dots $.}
\nonumber
\end{align}

Next, multiplying both sides of \eqref{kenApp02} by $ {\theta_{h, i}^{\nu} - \theta_{h, i -1}^{\nu}} \in H_0^1(\Omega) $, one can see that:
\begin{align*}
\frac{1}{h} |{\textstyle \sqrt{\alpha_0(\eta_{h, i}^{\nu})}} + &  {\int_\Omega \jump{\alpha(\eta_{h, i}^{\nu})\nabla \theta_{h, i}^{\nu}}_{\gamma}^{\nu} \cdot \nabla (\theta_{h, i}^{\nu} - \theta_{h, i -1}^{\nu}) \, dx = } 0,
\\
& \qquad \mbox{for all $ h \in (0, 1) $ and $ i = 1, 2, 3, \dots $.}
\end{align*}
Subsequently, invoking {(\hyperlink{AP1}{AP1}), Remark \ref{Rem.subdif_rx} and Young's inequality}, it is inferred that:
{
\begin{align}\label{thAp1-03}
\frac{1}{h} | & {\textstyle \sqrt{\alpha_0(\eta_{h, i}^{\nu})}} (\theta_{h, i}^{\nu} -\theta_{h, i -1}^{\nu})|_{L^2(\Omega)}^2 + \left( \int_\Omega \alpha(\eta_{h, i}^{\nu}) |\nabla \theta_{h, i}^{\nu}|_\nu dx + \frac{\nu^{2}}{2} \int_\Omega |\nabla (\theta_{h,i}^{\nu} - [\gamma]^{\rm hm})|^2 \, dx  \right)
\nonumber
\\
& \leq \int_\Omega \alpha(\eta_{h, i}^{\nu}) |\nabla \theta_{h, i-1}^{\nu}|_\nu dx + \frac{\nu^{2}}{2} \int_\Omega  |\nabla (\theta_{h, i -1}^{\nu} - [\gamma]^{\rm hm})|^2 \, dx
\\
& \qquad \mbox{for all $ h \in (0, h_*) $ and $ i = 1, 2, 3, \dots $.}
\nonumber
\end{align}
}
Now, we deduce \eqref{ene-inq} as the result of the sum of \eqref{thAp1-02} and \eqref{thAp1-03}. Finally, \eqref{lem2-inq} is obtained by multiplying  both sides of \eqref{ene-inq} by $ ih $ and taking the sum with respect to $ i \in \{ 1, \dots, m \} $. \hfill $ \Box $
\bigskip

\noindent
\textbf{Proof of Theorem \ref{thAp2}.}
First, we fix, arbitrary, the value of index $ i \in \{ 1, \dots, m \} $ with the maximum $ m \in \N $, and fix the pair of function $ [w_0, v_0] \in {D(\mathcal{F}^{\nu}_{\gamma}(\cdot{}))} $, {i.e.
\begin{equation*}
[w_{0}, v_{0}] \in H^{1}(\Omega)^{2} \ \ \mbox{ and }\ \ v_{0} = \gamma \mbox{ in } H^{\frac{1}{2}}(\Gamma).
\end{equation*}
}
Besides, we define a large constant $ R_* $ as follows:
\begin{align*}
R_* := \frac{1}{\delta_\alpha^4} \cdot & (1 +|\alpha_0|_{W^{1, \infty}(0, 1)})^2
\cdot (1 +|\alpha|_{C^{1}[0, 1]})^2
\cdot (1 +|G|_{C^{1}[0, 1]})^2 \cdot
\\
\cdot & (1 +\mathcal{L}^N(\Omega))^2
\cdot (1 +|c|_{L^{\infty}(0, 1)})^2
\cdot (1 +{|\gamma|_{L^{\infty}(\Gamma)}})^2
    \cdot {(1 +{C_{\Omega}^{\rm hm}})^{2}},
\end{align*}
{where $ C_{\Omega}^{\rm hm} $ is the constant as in \eqref{hm_const}, in the case when $ U = \Omega $.}

On this basis, we take $ w = \eta_{h, i}^{\nu} -w_0 \in H^1(\Omega) \cap L^\infty(\Omega) $ in \eqref{kenApp01}. Then, by {(\hyperlink{A1}{A1}) and (\hyperlink{A2}{A2}) and Young's inequality}, we have:
\begin{align*}
\frac{1}{2h} \bigl( |\eta_{h, i}^{\nu} & -w_0|_{L^2(\Omega)}^2 -|{\eta_{h, i -1}^{\nu}} -w_0|_{L^2(\Omega)}^2 \bigr) +\frac{1}{2} \int_\Omega |\nabla \eta_{h, i}^{\nu}|^2 \, dx
\\
&  +\int_\Omega \bigl( \alpha(\eta_{h, i}^{\nu}) -\alpha(w_0) \bigr) |\nabla \theta_{h, i -1}^{\nu}|_{\nu} \, dx  \leq \frac{1}{2} \int_\Omega |\nabla w_0|^2 \, dx +R_*.
\end{align*}
Subsequently, by (\hyperlink{A2}{A2}) and \eqref{kenApp00}, we can compare $ \alpha(\eta_{h, i}^{\nu}) \geq \frac{\delta_\alpha}{|\alpha|_{C[0, 1]}} \alpha(\eta_{h, i -1}^{\nu}) $, and can compute that:
\begin{align}\label{thAp2-01}
\frac{1}{2h} \bigl( |\eta_{h, i}^{\nu} & -w_0|_{L^2(\Omega)}^2 -|{\eta_{h, i -1}^{\nu}} -w_0|_{L^2(\Omega)}^2 \bigr)
\nonumber
\\
& +\frac{1}{2} \int_\Omega |\nabla \eta_{h, i}^{\nu}|^2 \, dx  +\frac{\delta_\alpha}{|\alpha|_{C[0, 1]}} \int_\Omega \alpha(\eta_{h, i -1}^{\nu}) |\nabla \theta_{h, i -1}^{\nu}|_\nu \, dx
\\
& -|\alpha|_{C[0, 1]} \int_\Omega |\nabla \theta_{h, i -1}^{\nu}|_{\nu} \, dx \leq R_* (1 +|w_0|_{H^1(\Omega)}^2).
\nonumber
\end{align}

Next, we multiply both sides of \eqref{kenApp02} by ${ [(\theta_{h, i}^{\nu} -v_0)/ \alpha_{0}(\eta_{h,i}^{\nu})]} \in H_0^1(\Omega) $. Then, applying \eqref{kenApp00} and Green's formula, we can see that:
{
\begin{align}\label{thAp2-02}
\frac{1}{h} \bigl( \theta_{h, i}^{\nu} - \theta_{h, i-1}^{\nu}, \theta_{h, i}^{\nu} - v_0 \bigr)_{L^2(\Omega)} +\int_\Omega \jump{\alpha(\eta_{h, i}^{\nu}) \nabla \theta_{h, i}^{\nu}}_{\gamma}^{\nu} \cdot \nabla \left( \frac{(\theta_{h, i}^{\nu} -v_0)}{\alpha_0(\eta_{h, i}^{\nu})} \right) \, dx = 0.
\end{align}
}
Also, invoking (\hyperlink{A2}{A2}), (\hyperlink{A3}{A3}), (\hyperlink{AP1}{AP1}), and (\hyperlink{AP2}{AP2}), one can check that:
{
\begin{align}\label{thAp2-03}
\frac{1}{h} \bigl( \theta_{h, i}^{\nu} - \theta_{h, i-1}^{\nu}, \theta_{h, i}^{\nu} -v_0 \bigr)_{L^2(\Omega)} \geq \frac{1}{2h} \bigl( |\theta_{h, i}^{\nu} -v_0|_{L^2(\Omega)}^2 -|\theta_{h, i -1}^{\nu} -v_0|_{L^2(\Omega)}^2 \bigr)
\end{align}
}
and
{
\begin{align}\label{kenAp2-04}
\int_\Omega \jump{\alpha(\eta_{h, i}^{\nu})\nabla \theta_{h, i}^{\nu}}_{\gamma}^{\nu} & \cdot \nabla \left( \frac{\theta_{h, i}^{\nu} -v_0}{\alpha_0(\eta_{h, i}^{\nu})} \right) \, dx
\nonumber
\\
& \hspace{-8ex} =  \int_\Omega \alpha(\eta_{h, i}^{\nu}) [\nabla |{}\cdot{}|_\nu](\nabla \theta_{h,i}^{\nu}) \cdot \left( \frac{\nabla\theta_{h, i}^{\nu}}{\alpha_0(\eta_{h, i}^{\nu})} \right) \, dx
\nonumber
\\
& +\nu^{2} \int_\Omega \nabla(\theta_{h, i}^{\nu} - [\gamma]^{\rm hm}) \cdot \left( \frac{\nabla (\theta_{h, i}^{\nu} - [\gamma]^{\rm hm})}{\alpha_0(\eta_{h, i}^{\nu})} \right) \, dx
\\
& -I_1 -I_2 -I_3 -I_4,
\nonumber
\end{align}
}
where
{
\begin{equation}\label{thAp2-06}
I_1 := \int_\Omega \frac{\alpha(\eta_{h, i}^{\nu})}{\alpha_0(\eta_{h, i}^{\nu})} [\nabla |{}\cdot{}|_\nu](\nabla \theta_{h, i}^{\nu}) \cdot \nabla v_0 \, dx,
\end{equation}
\begin{equation}\label{thAp2-07}
I_2 := \nu^{2} \int_\Omega \frac{1}{\alpha_0(\eta_{h, i}^{\nu})} \nabla (\theta_{h, i}^{\nu} -[\gamma]^{\rm hm}) \cdot \nabla (v_0 -[\gamma]^{\rm hm}) \, dx,
\end{equation}
\begin{equation}\label{thAp2-08}
I_3 := \nu^{2} \int_\Omega \frac{\alpha_0'(\eta_{h, i}^{\nu})}{\alpha_0(\eta_{h, i}^{\nu})^{2}} \bigl(\theta_{h, i}^{\nu} -v_0 \bigr) \nabla (\theta_{h, i}^{\nu} -[\gamma]^{\rm hm}) \cdot \nabla \eta_{h, i}^{\nu} \, dx,
\end{equation}
\begin{equation}\label{thAp2-09}
I_4 := \int_\Omega \frac{\alpha(\eta_{h, i}^{\nu}) \alpha_0'(\eta_{h, i}^{\nu})}{\alpha_0(\eta_{h, i}^{\nu})^2} \bigl(\theta_{h, i}^{\nu} -v_0 \bigr) [\nabla |{}\cdot{}|_\nu](\nabla \theta_{h, i}^{\nu}) \cdot \nabla \eta_{h, i}^{\nu} \, dx.
\end{equation}
}
Here, invoking (\hyperlink{A2}{A2}), (\hyperlink{A3}{A3}), (\hyperlink{AP1}{AP1}), (\hyperlink{AP2}{AP2}), \eqref{hm_const}, \eqref{kenApp00}, and H\"{o}lder's  and Young's inequalities, we can estimate the four terms $ I_k $, $ k = 0, 1, 2, 3 $, as follows:
{
\begin{align}\label{thAp2-11}
|I_1| \leq & \frac{|\alpha|_{C[0, 1]}|c|_{L^\infty(0, 1)}}{\delta_\alpha} \int_\Omega |\nabla v_0| \, dx
\nonumber
\\
\leq & \frac{|\alpha|_{C[0, 1]}|c|_{L^\infty(0, 1)} (1 +\mathcal{L}^N(\Omega))}{\delta_\alpha} \cdot |v_0|_{H^1(\Omega)}
\nonumber
\\
\leq & R_* \left( 1 +|v_0|_{H^1(\Omega)}^2 \right),
\end{align}
\begin{align}\label{thAp2-12}
|I_2| \leq & \frac{\nu^{2}}{4|\alpha_0|_{C[0, 1]}} \int_\Omega |\nabla (\theta_{h, i}^{\nu} -[\gamma]^{\rm hm})|^2 \, dx + \nu^{2} \cdot \frac{|\alpha_0|_{C[0, 1]}}{\delta_\alpha^2} \bigl( |\nabla (v_0 -[\gamma]^{\rm hm})|_{L^{2}(\Omega)^N}^2 \bigr)
\nonumber
\\
\leq & \frac{\nu^{2}}{4|\alpha_0|_{C[0, 1]}} \int_\Omega |\nabla (\theta_{h, i}^{\nu} -[\gamma]^{\rm hm})|^2 \, dx +2 \nu^{2} R_* \bigl( |v_0|_{H^{1}(\Omega)}^2 +|\gamma|_{H^{\frac{1}{2}}(\Gamma)}^2  \bigr),
\end{align}
\begin{align}\label{thAp-13}
|I_3| \leq & \frac{\nu^{2}}{4|\alpha_0|_{C[0, 1]}} \int_\Omega |\nabla (\theta_{h, i}^{\nu} -[\gamma]^{\rm hm})|^2 \, dx
\nonumber
\\
& +\nu^{2} \cdot |\alpha_0|_{C[0, 1]} \cdot \frac{4 |\alpha_0'|_{L^\infty(0, 1)}^2 |\gamma|_{L^\infty(\Gamma)}^2}{\delta_\alpha^4} \int_\Omega |\nabla \eta_{h, i}^{\nu}|^2 \, dx
\nonumber
\\
\leq & \frac{\nu^{2}}{4|\alpha_0|_{C[0, 1]}} \int_\Omega |\nabla (\theta_{h, i}^{\nu} -[\gamma]^{\rm hm})|^2 \, dx +4 \nu^{2} |\alpha_0|_{C[0, 1]} R_* \int_\Omega |\nabla \eta_{h, i}^{\nu}|^2 \, dx.
\end{align}
}
Meanwhile, putting:
\begin{equation}\label{thAp2-20}
A_1 := \frac{|\alpha_0|_{C[0, 1]} |\alpha|_{C[0, 1]}}{\delta_\alpha} ~ (\leq \delta_\alpha R_*^{\frac{1}{2}}),
\end{equation}
the last term $ I_4 $ can be estimated as:
\begin{align}\label{thAp2-21}
|I_4| \leq & \frac{1}{8A_1} \int_\Omega |\nabla \eta_{h, i}^{\nu}|^2 \, dx +8 A_1 R_*.
\end{align}
In addition, multiplying both sides of \eqref{thAp2-02} by $ A_1 $ and applying \eqref{thAp2-03}--\eqref{thAp2-21} yields
{
\begin{align}\label{thAp2-22}
\frac{A_1}{2h} & \left(|\theta_{h, i}^{\nu} -v_0|_{L^2(\Omega)}^2 |\theta_{h, i -1}^{\nu} -v_0|_{L^2(\Omega)}^2 \right)
    \nonumber
\\
& +|\alpha|_{C[0, 1]} \int_\Omega |\nabla \theta_{h, i}^{\nu}|_{\nu} \, dx +\frac{|\alpha|_{C[0, 1]}}{\delta_\alpha} \cdot \frac{\nu^{2}}{2} \int_\Omega |\nabla (\theta_{h, i}^{\nu} -[\gamma]^{\rm hm})|^2 \, dx
\nonumber
\\
\leq & \left( 4 \nu^{2} |\alpha_0|_{C[0, 1]} A_1 R_* + \frac{1}{8} \right) \int_\Omega |\nabla \eta_{h, i}^{\nu}|^2 \, dx + 9 A_{1}^{2} R_* \left( 1 +|v_0|_{H^{1}(\Omega)}^2 +|\gamma|_{H^{\frac{1}{2}}(\Gamma)}^2  \right).
\end{align}
}

Now, we take the sum of \eqref{thAp2-01} and \eqref{thAp2-22}. Then,
{
\begin{align}\label{thAp2-23}
\frac{1}{2h} \bigl( |\eta_{h, i}^{\nu} & -w_0|_{L^2(\Omega)}^2 -|\eta_{h, i -1}^{\nu} -w_0|_{L^2(\Omega)}^2 \bigr) +\frac{A_1}{2h} \bigl( |\theta_{h, i}^{\nu} -v_0|_{L^2(\Omega)}^2 -|\theta_{h, i -1}^{\nu} -v_0|_{L^2(\Omega)}^2 \bigr)
\nonumber
\\
& + \left( \frac{3}{8} - 4 \nu^{2} |\alpha_0|_{C[0, 1]} A_1 R_* \right) \int_\Omega |\nabla \eta_{h, i}^{\nu}| \, dx +\frac{\delta_\alpha}{|\alpha|_{C[0, 1]}} \int_\Omega \alpha(\eta_{h, i -1}^{\nu}) |\nabla \theta_{h, i -1}^{\nu}|_\nu \, dx
\nonumber
\\
& +\frac{\nu^{2}}{2}\cdot \frac{|\alpha|_{C([0,1])}}{\delta_{\alpha}} \int_\Omega |\nabla (\theta_{h, i}^{\nu} -[\gamma]^{\rm hm})|^2 \, dx
\nonumber
\\
& +|\alpha|_{C[0, 1]} \left( \int_\Omega |\nabla \theta_{h, i}^{\nu}|_\nu \, dx -\int_\Omega |\nabla \theta_{h, i -1}^{\nu}|_\nu \, dx  \right)
\nonumber
\\
\leq & 10 A_1^{2} R_* \bigl( 1 +|w_0|_{H^1(\Omega)}^2 +|v_0|_{H^1(\Omega)}^2 +|\gamma|_{H^{\frac{1}{2}}(\Gamma)}^2 \bigr).
\end{align}
}
Besides, letting
{
\begin{equation*}
\begin{array}{c}
    \min \left\{ \frac{1}{2}, \frac{\delta_{\alpha}}{|\alpha|_{C([0,1])}} \right\} \in {\textstyle \left( 0, \frac{1}{2} \right]}, \mbox{and \ } A_3 := 11A_{1}^{2}R_{\ast},
\end{array}
\end{equation*}
}
inequality \eqref{thAp2-23} can be reduced to
{
\begin{align}\label{thAp2-24}
& \hspace{-4ex} \frac{1}{2h} \bigl( |\eta_{h, i}^{\nu} -w_0|_{L^2(\Omega)}^2 -|\eta_{h, i -1}^{\nu} -w_0|_{L^2(\Omega)}^2 \bigr) +\frac{A_1}{2h} \bigl( |\theta_{h, i}^{\nu} -v_0|_{L^2(\Omega)}^2 -|\theta_{h, i -1}^{\nu} -v_0|_{L^2(\Omega)}^2 \bigr)
\nonumber
\\
& +A_2 \left( \frac{1}{2} \int_\Omega |\nabla \eta_{h, i}^{\nu}|^2 \, dx +\int_\Omega \alpha(\eta_{h, i -1}^{\nu}) |\nabla \theta_{h, i -1}^{\nu}|_\nu \, dx +\frac{\nu^{2}}{2} \int_\Omega |\nabla (\theta_{h, i}^{\nu} -[\gamma]^{\rm hm})|^2 \, dx \right)
\nonumber
\\
& +A_{2} \int_{\Omega} G(\eta_{h,i-1}^{\nu}) dx + |\alpha|_{C[0, 1]} \left( \int_\Omega |\nabla \theta_{h, i}^{\nu}|_\nu \, dx -\int_\Omega |\nabla \theta_{h, i -1}^{\nu}|_\nu \, dx  \right)
\nonumber
\\
\leq & A_3 \bigl( 1 +|w_0|_{H^1(\Omega)}^2 +|v_0|_{H^1(\Omega)}^2 +|\gamma|_{H^{\frac{1}{2}}(\Gamma)}^2 \bigr).
\end{align}
}

The conclusion \eqref{thAp2-00} of this theorem is obtained by multiplying both sides of \eqref{thAp2-24} by $h$ and taking the sum from the case of $ i = 1 $ up to $ i = m $. \hfill $ \Box $
\section{Proof of Main Theorems}
\subsection{Proof of Main Theorem 1}
\ \ \vspace{-3ex}

\noindent
{
In this section, we assume that $[\eta_0, \theta_0] \in D_0$.}
Though the proof of Main Theorem 1 is very similar to that of \cite[Main Theorem 1]{MR3268865}, we give it for the sake of completeness. The first goal in this Section is to find a solution $ [\eta, \theta] $ to $(\mbox{S})$, satisfying {(S4)}. {First of all, by Remark \ref{noteRx}, we can find a sequence $ \{ \tilde{\eta}_{0}^{n}, \tilde{\theta}_{0}^{n} \}_{n = 1}^\infty $  such that
\begin{equation}\label{kenV-01}
\left\{ \begin{array}{l}
[\tilde{\eta}_{0}^{n}, \tilde{\theta}_{0}^{n}] \in D_0^\mathrm{rx},\ \ \mbox{ \ for $ n = 1, 2, 3, \dots $,}
\\[1ex]
[\tilde{\eta}_{0}^{n}, \tilde{\theta}_{0}^{n}] \to [\eta_{0}, \theta_0] \mbox{ in $ L^2(\Omega)^{2} $ \ as $ n \to \infty $.
}
\end{array} \right.
\end{equation}
}

Let $ \nu_*, h_* \in (0, 1) $ be the constants as in Theorems \ref{thAp1} and \ref{thAp2}, and let $ \{ \nu_n \}_{n = 1}^\infty \subset (0, \nu_*) $ and $ \{ h_n \}_{n = 1}^\infty $ be sequences such that
{
\begin{equation}\label{6-1}
\left\{ \begin{array}{l}
0 < \nu_{n+1} < \nu_{n} < \nu_{\ast}2^{-n},\ \ 0 < h_{n+1} < h_{n} < h_{\ast} 2^{-n}, \vspace{3mm}\\
\ds 0 \le h_{n} \mathcal{F}^{\nu_{n}}_{\gamma}(\tilde{\eta}_{0}^{n}, \tilde{\theta}_{0}^{n}) < 2^{-n},
\end{array}
\right. \mbox{ for all } n \in \mathbb{N}.
\end{equation}
}

\medskip
For any $ n \in \N $, we denote by {$ \{ [\tilde{\eta}_{h_{n},i}^{\nu_{n}}, \tilde{\theta}_{h_{n},i}^{\nu_{n}}] \}_{i = 1}^\infty $} the unique solution to $ \mbox{(\hyperlink{AP_h^nu}{AP}$_{h_n}^{\nu_n} $)} $ with initial data {$ [\tilde{\eta}_0^{n}, \tilde{\theta}_{0}^{n}] $}, provided by Theorem \ref{thAp1}.
{
We define three different kinds of time-interpolations $ [\overline{\eta}_n, \overline{\theta}_n] \in L_{\rm loc}^\infty([0, \infty); H^1(\Omega))^2 $, $ [\underline{\eta}_n, \underline{\theta}_n] \in L_{\rm loc}^\infty([0, \infty); H^1(\Omega))^2 $ and $ [\widehat{\eta}_n, \widehat{\theta}_n] \in W_{\rm loc}^{1, \infty}([0, \infty); H^1(\Omega))^2 $, by letting
\begin{equation}\label{kenV-20}
\left \{
\begin{array}{ll}
\ds [\overline{\eta}_{n}(t), \overline{\theta}_{n}(t)] := [\tilde{\eta}_{h_{n},i}^{\nu_{n}}, \tilde{\theta}_{h_{n},i}^{\nu_{n}}], & \mbox{if $ t \in ((i -1)h_{n}, ih_{n}] \cap [0, \infty) $ with some $ i \in \mathbb{Z} $,}
\\[2ex]
[\underline{\eta}_{n}(t), \underline{\theta}_{n}(t)] := [\tilde{\eta}_{h_{n},i-1}^{\nu_{n}}, \tilde{\theta}_{h_{n},i-1}^{\nu_{n}}], & \mbox{if $ t \in [(i -1)h_{n}, ih_{n}) $ with some $ i \in \N $,}
\\[1ex]
\multicolumn{2}{l}{
\ds [\widehat{\eta}_{n}(t), \widehat{\theta}_{n}(t)] := \frac{ih_{n} -t}{h_{n}}[\tilde{\eta}_{h_{n},i-1}^{\nu_{n}}, \tilde{\theta}_{h_{n},i-1}^{\nu_{n}}] + \frac{t-(i -1)h_{n}}{h_{n}}[\tilde{\eta}_{h_{n},i}^{\nu_{n}}, \tilde{\theta}_{h_{n},i}^{\nu_{n}}],
}
\\[1ex]
& \mbox{if $ t \in [(i -1)h_{n}, ih_{n}) $ with some $ i \in \N $,}
\end{array}
\right.
\end{equation}
for all $ t \geq 0 $.
}
Then, in the light of (\ref{kenV-01}) and Lemmas \ref{axLemD1} and \ref{axLemD2}, we immediately see that
\begin{equation*}
    \bigl\{ [\overline{\eta}_n(t), \overline{\theta}_n(t)], [\underline{\eta}_n(t), \underline{\theta}_n(t)], [\widehat{\eta}_n(t), \widehat{\theta}_n(t)] \bigr\} \subset { D_{0}^{\rm rx}}, \mbox{ for all $ t \geq 0 $.}
\end{equation*}
In particular,
\begin{equation}\label{app-rist}
\begin{array}{c}
\begin{cases}
\ds 0 \leq \overline{\eta}_n(t) \leq 1, ~ 0 \leq \underline{\eta}_n(t) \leq 1, ~ 0 \leq \widehat{\eta}_n(t) \leq 1,
\\[0.5ex]
\ds |\overline{\theta}_n(t)| \leq {|\gamma|_{L^\infty(\Gamma)}}, ~ |\underline{\theta}_n(t)| \leq {|\gamma|_{L^\infty(\Gamma)}},~|\widehat{\theta}_n(t)| \leq {|\gamma|_{L^\infty(\Gamma)}},
\end{cases}
\\
\\[-2ex]
\mbox{a.e. in $ \Omega $, for all $ t \geq 0 $ and $ n = 1, 2, 3, \dots $.}
\end{array}
\end{equation}
Here, from the energy-inequality (\ref{ene-inq}) in Theorem \ref{thAp1}, it follows that
\begin{align*}
\displaystyle \frac{1}{2} \int_{t_k}^{t_\ell} & |(\widehat{\eta}_n)_t(\tau)|_{L^2(\Omega)}^2 \, d \tau +\int_{t_k}^{t_\ell} \bigl| {\textstyle \sqrt{\alpha_0(\overline{\eta}_n(\tau))}} (\widehat{\theta}_n)_t(\tau) \bigr|_{L^2(\Omega)}^2 \, d \tau
\\
& { + \mathcal{F}^{\nu_n}_{\gamma}(\overline{\eta}_n(t_\ell), \overline{\theta}_n(t_\ell)) \le \mathcal{F}^{\nu_n}_{\gamma}(\underline{\eta}_n(t_k), \underline{\theta}_n(t_k)), }
\\
& \mbox{{where $t_{k} := kh$,} for all $ 0 \leq k \in \mathbb{Z} $ and $ k \leq  \ell \in \N $.}
\end{align*}
and hence, it is observed that
\begin{align}\label{kenV-100}
\ds \frac{1}{2} \int_s^t & |(\widehat{\eta}_n)_t(\tau)|_{L^2(\Omega)}^2 \, d \tau +\int_s^t \bigl| {\textstyle \sqrt{\alpha_0(\overline{\eta}_n(\tau))}} (\widehat{\theta}_n)_t(\tau) \bigr|_{L^2(\Omega)}^2 \, d \tau
\nonumber
\\
& {+ \mathcal{F}^{\nu_n}_{\gamma}(\overline{\eta}_n(t), \overline{\theta}_n(t)) \le \mathcal{F}^{\nu_n}_{\gamma}(\underline{\eta}_n(s), \underline{\theta}_n(s)) }
\\
& \mbox{for all $ 0 \leq s \leq t < \infty $.}
\nonumber
\end{align}
Similarly, from (\ref{lem2-inq}) in Theorem \ref{thAp1} and (\ref{thAp2-00}) in Theorem \ref{thAp2}, we can deduce that:
\begin{align}\label{kenV-04}
\ds \frac{1}{2} \int_0^t & \tau |(\widehat{\eta}_n)_t(\tau)|_{L^2(\Omega)}^2 \, d \tau + \int_0^t \tau \bigl|{\textstyle \sqrt{\alpha_0(\overline{\eta}_n(\tau))}}(\widehat{\theta}_n)_t(\tau) \bigr|_{L^2(\Omega)}^2 \, d \tau
\nonumber
\\
& \ds {+t \mathcal{F}^{\nu_n}_{\gamma}(\overline{\eta}_n(t), \overline{\theta}_n(t)) \leq \int_0^{t +h_n} \mathcal{F}^{\nu_n}_{\gamma}(\underline{\eta}_n(\tau), \underline{\theta}_n(\tau)) \, d \tau,}
\end{align}
and
\vspace{-1ex}
{
\begin{align}\label{kenV-03}
\frac{1}{2} & |\overline{\eta}_n(t) -w_0|_{L^2(\Omega)}^2 +\frac{A_1}{2} |\overline{\theta}_n(t) -v_0|_{L^2(\Omega)}^2 +A_2 \int_0^t \mathcal{F}^{\nu_n}_{\gamma}(\underline{\eta}_n(\tau), \underline{\theta}_n(t)) \, d \tau
\nonumber
\\
&
\leq \frac{1}{2}|\tilde{\eta}_0^{n}(t) -w_0|_{L^2(\Omega)}^2 +\frac{A_1}{2} |\tilde{\theta}_0^{n}(t) -v_0|_{L^2(\Omega)}^2 +
\frac{h_n}{A_{2}} \mathcal{F}^{\nu_n}_{\gamma}(\tilde{\eta}_0^{n}, \tilde{\theta}_{0}^{n})
\\
& \quad +2A_3 t \left( 1 +|w_0|_{H^1(\Omega)}^2 +|v_0|_{H^1(\Omega)}^2 +|\gamma|_{H^{\frac{1}{2}}(\Gamma)}^2 \right)
\nonumber
\\
& \qquad \mbox{for all $ 0 \leq s \leq t < \infty $,  $ [w_0, v_0] \in D_0^\mathrm{rx} $, and $ n = 1, 2, 3, \dots $.}
\nonumber
\end{align}
}

Now, taking into account (\ref{kenV-01})--(\ref{kenV-03}), we obtain the following properties:
\vspace{0.5ex}
\begin{description}
\hypertarget{1-a}{}
\item[\textmd{($\sharp$1-a)}]$ \{ [\overline{\eta}_n, \overline{\theta}_n] \}_{n = 1}^\infty $ and $ \{ [\underline{\eta}_n, \underline{\theta}_n] \}_{n = 1}^\infty $ are bounded in {$ L_{\rm loc}^\infty((0, \infty); L^2(\Omega))^2 $}, and $ \{ [\widehat{\eta}_n, \widehat{\theta}_n] \}_{n = 1}^\infty $ is bounded in {$ C_{\rm loc}((0, \infty); L^2(\Omega))^2 \cap W^{1,2}_\mathrm{loc}((0,\infty); L^2(\Omega))^2 $};
\vspace{0.5ex}
\hypertarget{1-b}{}
\item[\textmd{($\sharp$1-b)}]the function
$
t \in {(0, \infty)}  \mapsto {\mathcal{F}^{\nu_n}_{\gamma}(\overline{\eta}_n(t), \overline{\theta}_n(t))} \in \R
$
is nonincreasing on {$ (0, \infty) $}, and hence, the sequences $ \{ \mathcal{F}^{\nu_n}_{\gamma}(\overline{\eta}_n, \overline{\theta}_n) \}_{n = 1}^\infty $ and $ \{ \mathcal{F}^{\nu_n}_{\gamma}(\underline{\eta}_n, \underline{\theta}_n) \}_{n = 1}^\infty $ are bounded in {$  L_{\rm loc}^1([0, \infty)) \cap  BV_{\rm loc}((0, \infty)) $};
        \hypertarget{1-c}{}
\item[\textmd{($\sharp$1-c)}] {$h_{n} \mathcal{F}_{\gamma}^{\nu_{n}}(\tilde{\eta}_{0}^{n}, \tilde{\theta}_{0}^{n}) \to 0$} as $n \to \infty$.
\end{description}
Therefore, by compactness (see \cite{MR1857292} and \cite[Corollary 4]{MR0916688}), we can find a pair of functions $ [\eta, \theta] \in {C_\mathrm{loc}((0, \infty); L^2(\Omega))^2} $  together with subsequences (not relabeled) of
$ \{ [\overline{\eta}_n, \overline{\theta}_n] \}_{n = 1}^\infty $, $ \{ [\underline{\eta}_n, \underline{\theta}_n] \}_{n = 1}^\infty $ and
$ \{ [\widehat{\eta}_n, \widehat{\theta}_n] \}_{n = 1}^\infty $, such that

\begin{equation}\label{kenV-05}
\left\{ \hspace{-2ex} \parbox{15,5cm}{
\vspace{-2ex}
\begin{itemize}
\item $ \eta \in  { W_{\rm loc}^{1,2}((0,\infty);L^{2}(\Omega)) \cap L_{\rm loc}^\infty((0, \infty); H^1(\Omega))} $ and $ 0 \leq \eta \leq 1 $ a.e. in $ Q $,
\vspace{-1ex}
\item $\theta \in { W_{\rm loc}^{1,2}((0,\infty);L^{2}(\Omega))}$, $ |D \theta(\cdot)|(\Omega) \in {L_{\rm loc}^\infty((0, \infty))} $ and $ |\theta| \leq {|\gamma|_{L^\infty(\Gamma)}} $ a.e. in $ Q $;
\vspace{-1ex}
\end{itemize}
} \right.
\end{equation}
{
\begin{equation}\label{sm-07}
\left\{ \hspace{-2ex} \parbox{12cm}{
\vspace{-2ex}
\begin{itemize}
\item $ \widehat{\theta}_n \to \theta $ in $ {C_{\rm loc}((0, \infty); L^2(\Omega))} $, weakly in $ { W^{1,2}_{\rm loc}((0, \infty), L^2(\Omega))} $ and weakly-$*$ in $ L^\infty(Q) $,
\vspace{-1ex}
\item $ \widehat{\theta}_n(t) \to \theta(t) $ in $ L^2(\Omega) $, and weakly-$ * $ in $ BV(\Omega) $, for any $ t > 0 $,
\vspace{-3ex}
\item $ \widehat{\eta}_n \to \eta $ in $ {C_{\rm loc}((0, \infty); L^2(\Omega))} $, weakly-$ * $ in $ {L_{\rm loc}^\infty((0, \infty); H^1(\Omega))} $, and weakly-$*$ in $ L^\infty(Q) $,
\vspace{-1ex}
\item $ \widehat{\eta}_n(t) \to \eta(t) $ in $ L^2(\Omega) $, and weakly in $ H^1(\Omega)$, {for any $ t > 0 $,}
\vspace{-2ex}\end{itemize}
} \right. \mbox{ \ \ \ \ as $ n \to \infty $.}
\end{equation}}
Meanwhile, for any bounded open interval $ I \subset (0, \infty) $, it is checked that
\begin{align*}
|\overline{\eta}_n & -\widehat{\eta}_n|_{C(\overline{I}; L^2(\Omega))} +|\underline{\eta}_n -\widehat{\eta}_n|_{C(\overline{I}; L^2(\Omega))}
\nonumber
\\
& \leq 2 |\overline{\eta}_n -\underline{\eta}_n|_{L^\infty(I; L^2(\Omega))} \leq 2 \sup_{{\tau, \sigma \in I}\atop{0 \leq \tau -\sigma \leq h_n}} |\widehat{\eta}_n(\tau) -\widehat{\eta}_n(\sigma)|_{L^2(\Omega)}
\\
& \leq 4 |\widehat{\eta}_n -\eta|_{C(\overline{I}; L^2(\Omega))} +2 \sup_{{\tau, \sigma \in I}\atop{0 \leq \tau -\sigma \leq h_n}} |\eta(\tau) -\eta(\sigma)|_{L^2(\Omega)} \to 0 \mbox{ as $ n \to \infty $,}
\nonumber
\end{align*}
as well as,
\begin{equation*}
|\overline{\theta}_n -\widehat{\theta}_n|_{C(\overline{I}; L^2(\Omega))} +|\underline{\theta}_n -\widehat{\theta}_n|_{C(\overline{I}; L^2(\Omega))} \to 0 \mbox{ as $ n \to \infty $.}
\end{equation*}
Hence, we also have the following convergences
\begin{equation}\label{kenV-06}
\left\{ \hspace{-2ex} \parbox{10cm}{
\vspace{-2ex}
\begin{itemize}
\item $ \overline{\eta}_n \to \eta $ and $ \underline{\eta}_n \to \eta $ in {$ L_{\rm loc}^\infty((0, \infty); L^2(\Omega)) $}, weakly-$ * $ in {$ L_{\rm loc}^\infty((0, \infty); H^1(\Omega)) $}, and weakly-$*$ in $ L^\infty(Q) $,
\item $ \overline{\eta}_n(t) \to \eta(t) $, $ \underline{\eta}_n(t) \to \eta(t) $ in $ L^2(\Omega) $, weakly in $ H^1(\Omega) $, for any {$ t > 0 $},
\vspace{-2ex}
\end{itemize}
} \right. \mbox{ \ \ \ \ as $ n \to \infty $;}
\end{equation}
\begin{equation}\label{kenV-07}
\left\{ \hspace{-2ex} \parbox{10cm}{
\vspace{-2ex}
\begin{itemize}
\item $ \overline{\theta}_n \to \theta $ and $ \underline{\theta}_n \to \theta $ in {$ L_{\rm loc}^\infty((0, \infty); L^2(\Omega)) $}, and weakly-$*$ in $ L^\infty(Q) $,
\item $ \overline{\theta}_n(t) \to \theta(t)  $, $ \underline{\theta}_n(t) \to \theta(t) $ in $ L^2(\Omega) $, and weakly-$ * $ in $ BV(\Omega) $, for any {$ t > 0 $},
\vspace{-2ex}
\end{itemize}
} \right. \mbox{ \ \ \ \ as $ n \to \infty $.}
\end{equation}
Moreover, under subsequences if necessary, we can find a function {$ \mathcal{J}_{\gamma} \in BV_{\rm loc}((0, \infty)) $} such that:
\begin{equation}\label{kenV-08}
\begin{array}{c}
\ds {\mathcal{F}^{\nu_n}_{\gamma}}(\underline{\eta}_n, \underline{\theta}_n) \to {\mathcal{J}}_\gamma
\mbox{ \ weakly-$*$ in $ BV_{\rm loc}((0, \infty)) $, weakly-$*$ in $ L_{\rm loc}^\infty((0, \infty)) $,}
\\[1ex]
\mbox{and a.e. in $ (0, \infty) $, \ as $ n \to \infty $.}
\end{array}
\end{equation}

Next, we will show that the limit pair $ [\eta, \theta] $ satisfies the variational inequalities {in (\hyperlink{S1}{S1}) and (\hyperlink{S2}{S2})}. Let $ I $ be any bounded open interval such that  $ I \subset\subset (0, \infty) $. Then, by (\ref{kenApp01}) and (\ref{kenApp02}), the sequences in (\ref{kenV-05})--(\ref{kenV-07}) satisfy the following two variational formulas:
\begin{equation}\label{kenV-09}
\begin{array}{c}
\ds \int_I \bigl( (\widehat{\eta}_n)_t(t) +g(\overline{\eta}_n(t)), \omega(t) \bigr)_{L^2(\Omega)} \, dt +\int_I \bigl( \nabla \overline{\eta}_n(t), \nabla \omega(t) \bigr)_{L^2(\Omega)^N} \, dt
\\[1ex]
\ds +\int_I \int_{\Omega} \omega(t) \alpha'(\overline{\eta}_n(t)) {|\nabla \underline{\theta}_n(t)|_{\nu_n}} \, dt = 0,
\\[2ex]
\mbox{for any $ \omega \in L^2(I; H^1(\Omega)) \cap L^\infty(I \times \Omega) $ and any $ n \in \N $,}
\end{array}
\end{equation}
and
\begin{equation}\label{kenV-10}
\begin{array}{l}
\ds \int_I \bigl( \alpha_0(\overline{\eta}_n(t)) (\widehat{\theta}_n)_t(t), \overline{\theta}_n(t) -v(t)  \bigr)_{L^2(\Omega)} \, dt + {\Phi_{\nu_n, \gamma}^I}(\alpha(\overline{\eta}_n); \overline{\theta}_n)
\\[1ex]
\qquad \ds \leq {\Phi_{\nu_n,\gamma}^I}(\alpha(\overline{\eta}_n); v),
\mbox{ \ for any $ v \in L^2(I; H^1(\Omega)) $ and any $ n \in \N $.}
\end{array}
\end{equation}

{
    Let us take any $ v \in C(\overline{I}; L^2(\Omega)) $ with $ |Dv({}\cdot{})|(\Omega) \in L^1(I) $. Owing to (\hyperlink{Fact4}{Fact\,4}), (\ref{2.4-1}) and (\ref{kenV-05}), it is deduced that $ v \in D({\Phi_{\gamma}^I}(\alpha(\eta));{}\cdot{})) $. So, by Corollary \ref{Cor.Gamma-conv_nu}, there exists a sequence $ \{ v_n \}_{n = 1}^\infty \subset L^2(I; H^1(\Omega)) $, such that
\begin{equation*}
v_n \to v \mbox{ in $ L^2(I; L^2(\Omega)) $ and $ {\Phi_{\nu_n,\gamma}^I}(\alpha(\overline{\eta}_n); v_n) \to {\Phi_{\gamma}^I}(\alpha(\eta); v) $, \ as $ n \to \infty $.}
\end{equation*}
Then, by (\ref{kenV-05})--(\ref{kenV-07}) and Corollary \ref{Cor.Gamma-conv_nu}, letting $ n \to \infty $ in (\ref{kenV-10}), it follows that
{
\begin{equation}\label{kenV-11}
\begin{array}{ll}
\multicolumn{2}{l}{\ds \int_{I} (\alpha_{0}(\eta(t)) \theta_{t}(t), \theta(t) -v(t))_{L^{2}(\Omega)} dt + {\Phi_{\gamma}^{I}}(\alpha(\eta);\theta)}
\\
~~~~ & \ds \le \lim_{n \to \infty} \int_{I} (\alpha_{0}(\overline{\eta}_{n})(\widehat{\theta}_{n})_{t}(t), \overline{\theta}_{n}(t)-v_n(t))_{L^{2}(\Omega)} dt + \liminf_{n\to\infty} {\Phi_{{\nu}_{n},\gamma}^{I}}(\alpha(\overline{\eta}_{n}); \overline{\theta}_{n})
\\
& \ds \le \int_{I} (\alpha_{0}(\eta) \theta_{t}(t), \theta(t)-v(t))_{L^{2}(\Omega)} dt + \limsup_{n\to\infty} {\Phi_{{\nu}_{n},\gamma}^{I}}(\alpha(\overline{\eta}_{n}); \overline{\theta}_{n})
\\
& \ds \le \lim_{n \to \infty} {\Phi_{{\nu}_{n},\gamma}^{I}}(\alpha(\overline{\eta}_{n});v_n)= {\Phi_{\gamma}^{I}}(\alpha(\eta);v),
\\
\multicolumn{2}{l}{\mbox{ for any $ v \in C(\overline{I}; L^2(\Omega)) $ with $ |Dv({}\cdot{})|(\Omega) \in L^1(I) $.}}
\end{array}
\end{equation}
}
Since $ I$ was arbitrarily chosen, from (\ref{kenV-11}) we obtain that $ [\eta, \theta] $ satisfy the variational inequality {in (\hyperlink{S2}{S2})}.
\medskip

{
Next, letting $ v = \theta $ in (\ref{kenV-11}), we obtain
\begin{equation}\label{kenV-17}
\lim_{n \to \infty} {\Phi_{\nu_n,\gamma}^I}(\alpha(\overline{\eta}_n); \overline{\theta}_n) = \Phi_\gamma^I(\alpha(\eta); \theta),
\end{equation}
and hence,
{
\begin{equation}\label{kenV-14}
\begin{array}{rcl}
0 & \leq & \ds \liminf_{n \to \infty} \frac{\nu_n^2}{2} \int_I \int_\Omega |\nabla (\overline{\theta}_n(t) - [\gamma]^{\rm hm} )|^2 \, dx dt \leq \limsup_{n \to \infty} \frac{\nu_n^{2}}{2} \int_I \int_\Omega |\nabla (\overline{\theta}_n - [\gamma]^{\rm hm})|^2 \, dx dt
\\[2ex]
& \leq & \ds \lim_{n \to \infty} \Phi_{\nu_n,\gamma}^I(\alpha(\overline{\eta}_n); \overline{\theta}_n) -\liminf_{n \to \infty} \Phi_\gamma^I(\alpha(\overline{\eta}_n); \overline{\theta}_n) \leq 0.
\end{array}
\end{equation}
}
}
In particular, (\ref{kenV-17}) and (\ref{kenV-14}) imply that
\begin{equation}\label{kenV-18}
\int_I \int_\Omega \alpha(\overline{\eta}_n(t)) {|\nabla \overline{\theta}_n(t)|_{\nu_n}} \, dx dt \to { \int_I \int_{\overline{\Omega}} d \bigl[ \alpha(\eta(t)) |D \theta(t)| \bigr]_{\gamma} \, dt} \mbox{ \ as $ n \to \infty $.}
\end{equation}

Having in mind (\ref{kenV-05})--(\ref{kenV-07}), and (\ref{kenV-18}), we apply Lemma \ref{Lem.core} with
$ \beta = \alpha(\eta) $, $ \{ \beta_n \}_{n = 1}^\infty = \{ \alpha(\overline{\eta}_n) \}_{n = 1}^\infty $, $ v = \theta $, $ \{ v_n \}_{n = 1}^\infty = \{ \overline{\theta}_n \}_{n = 1}^\infty $, $ \varrho = 1 $, $ \{ \varrho_n \}_{n = 1}^\infty = \{ 1 \} $ and { $\{\gamma_{n}\}_{n=1}^{\infty} = \{\gamma\}$}. Then, we see that
{
\begin{equation}\label{kenV-34}
\int_I \int_\Omega |\nabla \overline{\theta}_n(t)|_{\nu_n} \, dx dt \to {\int_I \int_{\overline{\Omega}} [|D \theta(t)|]_{\gamma} \, dt} \mbox{ \ as $ n \to \infty $.}
\end{equation}
}
On the other hand, let us take $I=(t_0,t_1)$. From {(\hyperlink{A2}{A2})}, \eqref{ene-inq}, (\ref{kenV-01})--(\ref{kenV-20}), we deduce that
{
\begin{equation}\label{kenV-35}\begin{array}{c}
\ds \left| \int_{I}\int_{\Omega} \bigl( |\nabla \underline{\theta}_{n}|_{\nu_n} -|\nabla \overline{\theta}_{n}|_{\nu_n} \bigr) \, dxdt \right| = {h_{n} {\int_\Omega (|\nabla \underline{\theta}(t_0)_{n}|_{\nu_n}-|\nabla \overline{\theta}(t_1)_{n}|_{\nu_n})\,dx}}
\\[2ex] \ds \leq
\frac{2}{\delta_\alpha} h_n \mathcal{F}_{\nu_n}({\tilde{\eta}_{0}^{n}}, {\tilde{\theta}_{0}^{n}}) \to 0,
\mbox{ as $ n \to \infty $.}
\end{array}
\end{equation}
}
Hence, taking any $w  \in H^{1}(\Omega) \cap L^{\infty}(\Omega) $, and applying Lemma \ref{Lem.core} with $ \beta = 1 $, $ \{ \beta_{n} \}_{n = 1}^\infty = \{ 1 \} $, $ v = \theta  $, $ \{ v_n \}_{n = 1}^\infty = \{ \underline{\theta}_n \}_{n = 1}^\infty $, $ \varrho = w \alpha'(\eta) $, $ \{ \varrho_{n} \}_{n = 1}^\infty = \{ w \alpha'(\overline{\eta}_{n}) \}_{n = 1}^\infty $ and {$\{\gamma_{n}\}_{n=1}^{\infty} = \{\gamma\}$}, we have
{
\begin{equation} \label{1steq-meas}
\begin{array}{c}
\ds \int_{I} \int_{\Omega} w \alpha'(\overline{\eta}_{n}(t))|\nabla \underline{\theta}_{n}(t)|_{\nu_n} dxdt \to {\int_{I} \int_{\overline{\Omega}} d \bigl[ w \alpha'(\eta(t)) |D \theta(t)| \bigr]_{\gamma} \, dt}
\\[2ex]
\mbox{as $ n \to \infty $, \ for any $ w \in H^1(\Omega) \cap L^\infty(\Omega) $.}
\end{array}
\end{equation}
}

Due to (\ref{kenV-05})--(\ref{kenV-07}), and (\ref{1steq-meas}), letting $n \to \infty$ in (\ref{kenV-09}), it follows that
\begin{equation*}
\begin{array}{c}
\ds \int_{I} (\eta_{t}(t) + g(\eta(t)), w)_{L^{2}(\Omega)} dt + \int_{I}\bigl( \nabla \eta(t), \nabla w \bigr)_{L^2(\Omega)^N} \, dxdt
\\[2ex]
\ds + \int_{I} {\int_{\overline{\Omega}} d \bigl[ w \alpha'(\eta(t))|D\theta(t)| \bigr]_{\gamma}} \, dt = 0, \mbox{ \ for any $w \in H^{1}(\Omega) \cap L^{\infty}(\Omega)$. }
\end{array}
\end{equation*}
By the arbitrariness of $I$, $ [\eta, \theta] $ satisfies the variational inequality {in (\hyperlink{S1}{S1})}.
\bigskip

In order to finish the proof of (S4), it remains to prove the following three items:

\begin{description}
\hypertarget{2-a}{}
\item[\textmd{($\sharp$2-a)}]$ \eta \in L_{\rm loc}^2([0, \infty); H^1(\Omega)) $ and $ |D \theta({}\cdot{})|(\Omega) \in L_{\rm loc}^1([0, \infty)) $;
\vspace{0.5ex}
\hypertarget{2-b}{}
\item[\textmd{($\sharp$2-b)}]$ [\eta, \theta] \in C([0, \infty); L^2(\Omega))^2 $ and $ [\eta(0), \theta(0)] = [\eta_0, \theta_0] $ in $ L^2(\Omega)^2 $;
\hypertarget{2-c}{}
\item[\textmd{($\sharp$2-c)}]$ \mathcal{J}_{\gamma} $ obtained in (\ref{kenV-08}) is nonincreasing, and $ \mathcal{J}_{\gamma} = \mathcal{F}_{\gamma}(\eta, \theta) $ a.e on $ (0, \infty) ${, i.e.:
\begin{equation*}
\mathcal{J}_{\gamma}(t) = \frac{1}{2} \int_\Omega |\nabla \eta(t)|^2 \, dx +\int_\Omega {G}(\eta(t)) \, dx + { \int_{\overline{\Omega}} d[\alpha(\eta(t))|D \theta(t)|]_{\gamma}}, \mbox{ for a.e. $ t > 0 $.}
\end{equation*}
}
\end{description}

We fix $ t \in (0, \infty) $, $ n,\ell\in \N $ and consider (\ref{kenV-03}) with $ [h, \nu] = [h_n, \nu_n] $ and $ [w_0, v_0] = {[\tilde{\eta}_{0}^{\ell}, \tilde{\theta}_{0}^{\ell}]} $. Then, with (\hyperlink{(A2)}{A2}), (\ref{rxFE}), (\ref{kenV-01}), (\ref{6-1})--(\ref{kenV-07}) and Corollary \ref{Cor.Gamma-conv_nu} in mind, letting $ n \to \infty $, we deduce that
\begin{align}
\frac{1}{2} \bigl( |\eta(t) -\tilde{\eta}_{0}^{\ell}|_{L^2(\Omega)}^2 +A_1 |\theta(t) - \tilde{\theta}_{0}^{\ell}|_{L^2(\Omega)}^2 \bigr) +A_{2} \left( \frac{1}{2} |\nabla \eta|_{L^2(0, t; L^2(\Omega)^N)}^2 +\delta_\alpha \bigl| |D \theta({}\cdot{})|(\Omega) \bigr|_{L^1(0, t)} \right)
\nonumber
\\
 \leq \frac{1}{2} \lim_{n \to \infty} \bigl( |\overline{\eta}_n(t) -\tilde{\eta}_{0}^{\ell}|_{L^2(\Omega)}^2 +A_1 |\overline{\theta}_n(t) -\tilde{\theta}_{0}^{\ell}|_{L^2(\Omega)}^2 \bigr) +A_2 \liminf_{n \to \infty} \int_0^t \mathcal{F}_{\nu_n}(\underline{\eta}_n(\tau), \underline{\theta}_n(\tau)) \, d \tau
\nonumber
\\
 \leq \frac{1}{2} \bigl( |\tilde{\eta}_{0}^{\ell} -\eta_{0}|_{L^2(\Omega)}^2 +A_1 |\tilde{\theta}_{0}^{\ell} -\theta_{0}|_{L^2(\Omega)}^2 \bigr)
+2t A_3 \bigl( 1 +|\tilde{\eta}_{0}^{\ell}|_{H^1(\Omega)}^2 +|\tilde{\theta}_{0}^{\ell}|_{H^1(\Omega)}^2 + |\gamma|^{2}_{H^{\frac{1}{2}}(\Gamma)}\bigr),\label{kenV-30}
\end{align}
which, together with (\ref{kenV-05}), yields (\hyperlink{2-a}{$\sharp$2-a}).
\medskip

In the meantime, the following inequality
\begin{equation*}
\begin{array}{c}
\ds |\eta(t) -\eta_0|_{L^2(\Omega)}^2 + {A_1} |\theta(t) -\theta_0|_{L^2(\Omega)}^2 \leq 2 \bigl( |\eta(t) -{\tilde{\eta}_{0}^{\ell}}|_{L^2(\Omega)}^2 + {A_1} |\theta(t) - {\tilde{\theta}_{0}^{\ell}}|_{L^2(\Omega)}^2 \bigr)
\\[1ex]
\ds +2 \bigl( |{\tilde{\eta}_{0}^{\ell}} -\eta_0|_{L^2(\Omega)}^2 +{A_1} |{\tilde{\theta}_{0}^{\ell}} -\theta_0|_{L^2(\Omega)}^2 \bigr), \mbox{ \ for any $ t \in (0, \infty) $ and $ \ell \in \N $,}
\end{array}
\end{equation*}
and (\ref{kenV-30}) imply that
\begin{equation*}
\begin{array}{c}
\ds \limsup_{t \downarrow 0} \, \bigl( |\eta(t) -\eta_0|_{L^2(\Omega)}^2 + {A_1} |\theta(t) -\theta_0|_{L^2(\Omega)}^2 \bigr)
\\[1ex]
\ds \leq 4 \bigl( |{\tilde{\eta}_{0}^{\ell}} -\eta_0|_{L^2(\Omega)}^2 + {A_1} |{\tilde{\theta}_{0}^{\ell}} -\theta_0|_{L^2(\Omega)}^2 \bigr), \mbox{ \ for any $ \ell \in \N $.}
\end{array}
\end{equation*}
{
By (\ref{kenV-01}), (\ref{kenV-20}), (\ref{kenV-06}), and (\ref{kenV-07}), the above inequality implies (\hyperlink{2-b}{$\sharp$2-b}).
}
\medskip

Next, given any bounded open interval $ I \subset \hspace{-0.25ex} \subset (0, \infty) $, we take a sequence $ \{ \eta_n \}_{n = 1}^\infty \subset C^\infty(\overline{I \times \Omega}) $, such that $ \eta_n \to \eta $ in $ L^2(I; H^1(\Omega)) $ as $ n \to \infty $. We choose  $ \omega = \overline{\eta}_n -\eta_n $ in  (\ref{kenV-09}). Then, taking into account (\ref{kenV-34})--{(\ref{1steq-meas})}, and applying {{Lemma \ref{Lem.core}}} with $ \beta = 1 $, $ \{ \beta_n \}_{n = 1}^\infty = \{ 1 \} $, $ v = \theta $, $ \{ v_n \}_{n = 1}^\infty = \{ \underline{\theta}_n \}_{n = 1}^\infty $, $ \varrho = 0 $, $ \{ \varrho_n \}_{n = 1}^\infty = \{ (\overline{\eta}_n -\eta_n) \alpha'(\overline{\eta}_n) \}_{n = 1}^\infty $ and {$\{\gamma_{n}\}_{n=1}^{\infty} = \{\gamma\}$}, one can see that
\begin{equation}\label{kenV-36}
\begin{array}{rl}
\multicolumn{2}{l}{\ds\hspace{-2ex} \int_I |\nabla \eta(t)|_{L^2(\Omega)^N}^2 \, dt \leq \liminf_{n \to \infty} \int_I |\nabla \overline{\eta}_n(t)|_{L^2(\Omega)^N}^2 \, dt \leq \limsup_{n \to \infty} \int_I |\nabla \overline{\eta}_n(t)|_{L^2(\Omega)^N}^2 \, dt
}
\\[2ex]
\leq & \ds \lim_{n \to \infty} \left[ \int_I |\nabla \eta_n(t)|_{L^2(\Omega)^N}^2 \, dt -2\int_I \bigl( (\widehat{\eta}_n)_t(t) +g(\overline{\eta}_n(t)), (\overline{\eta}_n -\eta_n)(t) \bigr)_{L^2(\Omega)} \, dt \right.
\\[2ex]
& \ds \quad  \left. -2\int_I \int_\Omega (\overline{\eta}_n -\eta_n)(t) \alpha'(\overline{\eta}_n(t)) {|\nabla \underline{\theta}_n(t)|_{\nu_n}} \, dx dt \right] = \int_I |\nabla \eta(t)|_{L^2(\Omega)^N}^2\, dt.
\end{array}
\end{equation}
By (\ref{kenV-06}), (\ref{kenV-07}), (\ref{kenV-18}), (\ref{kenV-36}), and the uniform convexities of the $ L^2 $-type topologies, we obtain
{
\begin{equation}\label{kenV-36_1}
\left\{ \parbox{8.5cm}{
$ \overline{\eta}_n \to \eta $ in  $ L^2(I; H^1(\Omega)) $,
\\[1ex]
$ \ds \int_I {\mathcal{F}^{\nu_n}_{\gamma}}(\overline{\eta}_n(t), \overline{\theta}_n(t)) \, dt \to \int_I {\mathcal{F}_{\gamma}}(\eta(t), \theta(t)) \, dt $,
} \right. \mbox{as $ n \to \infty $}
\end{equation}
}
and,  by (\ref{6-1}),
\begin{equation}\label{kenV-37}
\begin{array}{rl}
\multicolumn{2}{l}{\ds \hspace{-4ex}\left| \int_I {\mathcal{F}^{\nu_n}_{\gamma}}(\underline{\eta}_n(t), \underline{\theta}_n(t)) \, dt -\int_I {\mathcal{F}_{\gamma}} (\eta(t), \theta(t)) \, dt \right|}
\\[2ex]
\leq & \ds \left| \int_I {\mathcal{F}^{\nu_n}_{\gamma}}(\overline{\eta}_n(t), \overline{\theta}_n(t)) \, dt -\int_I {\mathcal{F}^{\nu_n}_{\gamma}}(\underline{\eta}_n(t), \underline{\theta}_n(t)) \, dt \right|
\\[2ex]
& \ds \hspace{17.8ex} +\left| \int_I {\mathcal{F}^{\nu_n}_{\gamma}}(\overline{\eta}_n(t), \overline{\theta}_n(t)) \, dt -\int_I {\mathcal{F}_{\gamma}}(\eta(t), \theta(t)) \, dt \right|
\\[2ex]
\leq & \ds 2 h_n {\mathcal{F}^{\nu_n}_{\gamma}}(\eta_{0, n}, \theta_{0, n}) +\left| \int_I {\mathcal{F}^{\nu_n}_{\gamma}}(\overline{\eta}_n(t), \overline{\theta}_n(t)) \, dt -\int_I {\mathcal{F}_{\gamma}}(\eta(t), \theta(t)) \, dt \right|
\\[2ex]
\to & 0, \mbox{ as $ n \to \infty $, \ for any bounded open interval $ I \subset \hspace{-0.25ex} \subset (0, \infty) $.}
\end{array}
\end{equation}

Given now any bounded open set $ A \subset (0, \infty) $, we denote by $ \mathcal{I}_A $ the (at most countable) class of pairwise-disjoint open intervals, such that \ $ A = \bigcup_{\tilde{I} \in \mathcal{I}_A} \tilde{I} $. Here, from (\ref{kenV-37}), it can be seen that
\begin{equation*}
\begin{array}{c}
\ds \sum_{\tilde{I} \in \tilde{\mathcal{I}}} \int_{\tilde{I}} {\mathcal{F}_{\gamma}}(\eta(t), \theta(t)) \, dt \leq \liminf_{n \to \infty} \int_{A} {\mathcal{F}^{\nu_n}_{\gamma}}(\underline{\eta}_n(t), \underline{\theta}_n(t)) \, dt
\\[2ex]
\mbox{for any finite subclass $ \tilde{\mathcal{I}} \subset \mathcal{I}_A $,}
\end{array}
\end{equation*}
and accordingly,
\begin{equation}\label{kenV-38}
\begin{array}{c}
\ds \int_A {\mathcal{F}_{\gamma}}(\eta(t), \theta(t)) \, dt \leq \liminf_{n \to \infty} \int_A {\mathcal{F}^{\nu_n}_{\gamma}}(\underline{\eta}_n(t), \underline{\theta}_n(t)) \, dt
\\[2ex]
\mbox{for any bounded open set $ A \subset \hspace{-0.25ex} \subset (0, \infty) $.}
\end{array}
\end{equation}
As an application of \cite[Proposition 1.80]{MR1857292} for (\ref{kenV-36_1})--(\ref{kenV-38}), we have
\begin{equation*}
{\mathcal{F}^{\nu_n}_{\gamma}}(\underline{\eta}_n, \underline{\theta}_n) \to {\mathcal{F}_{\gamma}}(\eta, \theta) \mbox{ in $ \mathscr{M}_{\rm loc}((0, \infty)) $, as $ n \to \infty $.}
\end{equation*}

Thus, by (\ref{kenV-08}), we observe that
\begin{equation}\label{kenV-40}
{\mathcal{F}_{\gamma}}(\eta(t), \theta(t)) = {\mathcal{J}_\gamma}(t), \mbox{ \ a.e. $ t \in (0, \infty) $.}
\end{equation}
On the other hand, having (\ref{kenV-08}) in mind and letting $ n \to \infty $ in (\ref{kenV-100}), it follows that
\begin{equation}\label{kenV-41}
\begin{array}{c}
\ds \frac{1}{2} \int_s^t |\eta_t(\tau)|_{L^2(\Omega)}^2 \, d \tau +\int_s^t |{\textstyle \sqrt{\alpha_0(\eta(\tau))}} \theta_t(\tau)|_{L^2(\Omega)}^2 \, d \tau + {\mathcal{J}_\gamma}(t) \leq {\mathcal{J}_\gamma}(s),
\\[2ex]
\mbox{for a.e. $ 0 < s < t < \infty $.}
\end{array}
\end{equation}
{
From \eqref{kenV-06} and (\hyperlink{2-a}{$\sharp$2-a}), we conclude that
}
\begin{center}
$ {\mathcal{F}_{\gamma}}(\eta, \theta) \in L_{\rm loc}^1([0, \infty)) $ $\cap L_{\rm loc}^\infty((0, \infty)) $,
\end{center}
which, together with (\ref{kenV-40}), (\ref{kenV-41}) leads to (\hyperlink{2-c}{$\sharp$2-c}).

{
    We finish the proof of Main Theorem 1 by addressing (\hyperlink{S5}{S5}). We take the initial data
\begin{equation*}
[\eta_{0}, \theta_{0}] \in D_{0} \cap (H^{1}(\Omega) \times BV(\Omega)).
\end{equation*}
Then, applying (\hyperlink{Fact7}{Fact\,7}) to the case $ \{ {\beta}_n \}_{n = 1}^\infty = \{ {\beta} \} = \{ \alpha(\eta_0) \}$, we can find a sequence $ \{{\theta}_{0}^{n} \}_{n = 1}^\infty $  such that
\begin{equation}\label{S5-1}
\left\{ \begin{array}{l}
[\eta_{0}, {\theta}_{0}^{n}] \in D_0^\mathrm{rx} \mbox{ \ for $ n = 1, 2, 3, \dots $,}
\\[1ex]
{\theta}_{0}^{n} \to \theta_0 \mbox{ in $ L^2(\Omega) $ \ as $ n \to \infty $.
}
\\[1ex] \Phi_{\gamma}^{\nu_n}(\alpha(\eta_0); {\theta}_{0}^{n}) \to \Phi_{\gamma}(\alpha(\eta_0); \theta_0)  \mbox{\rm \  as $ n \to \infty $.}
\end{array} \right.
\end{equation}
Letting
\begin{equation}\label{S5-2}
F_{\ast} := \sup_{n \in \mathbb{N}} \mathcal{F}_{\gamma}^{\nu_{n}}(\eta_{0}, {\theta}_{0}^{n}),
\end{equation}
from \eqref{ene-inq} we see that
\begin{equation}\label{S5-3}
\begin{array}{l}
\ds \frac{1}{2} \int_{0}^{T} |(\widehat{\eta}_{n})_{t}(t)|_{L^{2}(\Omega)}^{2} dt + \int_{0}^{T} |\sqrt{\alpha_{0}(\overline{\eta}_{n}(t))}(\widehat{\theta}_{n})_{t}(t)|_{L^{2}(\Omega)}^{2} dt \vspace{3mm}\\
\ds = \sum_{i \in \mathbb{Z},\ ih \in [0,T]} \left( \frac{1}{2h} |\eta_{h_{n},i}^{\nu_{n}} - \eta_{h_{n},i-1}^{\nu_{n}}|_{L^{2}(\Omega)}^{2} + \frac{1}{h} |\sqrt{\alpha_{0}(\eta_{h_{n},i}^{\nu_{n}})}(\theta_{h_{n},i}^{\nu_{n}} - \theta_{h_{n},i-1}^{\nu_{n}}) |_{L^{2}(\Omega)}^{2} \right) \le F_{\ast}, \vspace{3mm}\\
\ds \sup_{t \in [0,T]} |\mathcal{F}_{\gamma}^{\nu_{n}}(\overline{\eta}_{n}, \overline{\theta}_{n})| + \sup_{t \in [0,T]} |\mathcal{F}_{\gamma}^{\nu_{n}}(\underline{\eta}_{n}, \underline{\theta}_{n})| \le 2 \sup_{i \in \mathbb{Z},\ ih \in [0,T]} |\mathcal{F}_{\gamma}^{\nu_{n}}(\eta_{h_{n},i}^{\nu_{n}}, \theta_{h_{n},i}^{\nu_{n}})| \le 2F_{\ast},
\end{array}
\end{equation}
for all $T>0$. Taking into account (\ref{app-rist}), (\ref{kenV-100}), and (\ref{S5-1})--(\ref{S5-3}), we have the following properties:
\begin{description}
\hypertarget{3-a}{}
\item[\textmd{($\sharp$3-a)}]$ \{ [\overline{\eta}_n, \overline{\theta}_n] \}_{n = 1}^\infty $ and $ \{ [\underline{\eta}_n, \underline{\theta}_n] \}_{n = 1}^\infty $ are bounded in {$ L_{\rm loc}^\infty([0, \infty); L^2(\Omega))^2 $}, and $ \{ [\widehat{\eta}_n, \widehat{\theta}_n] \}_{n = 1}^\infty $ is bounded in {$ C_{\rm loc}([0, \infty); L^2(\Omega))^2 \cap W^{1,2}_\mathrm{loc}([0,\infty); L^2(\Omega))^2 $};
\vspace{0.5ex}
\hypertarget{3-b}{}
\item[\textmd{($\sharp$3-b)}]the function
$t \in {(0, \infty)}  \mapsto {\mathcal{F}^{\nu_n}_{\gamma}(\overline{\eta}_n(t), \overline{\theta}_n(t))} \in \R $
is nonincreasing on {$ [0, \infty) $}, and hence, the sequences $ \{ \mathcal{F}^{\nu_n}_{\gamma}(\overline{\eta}_n, \overline{\theta}_n) \}_{n = 1}^\infty $ and $ \{ \mathcal{F}^{\nu_n}_{\gamma}(\underline{\eta}_n, \underline{\theta}_n) \}_{n = 1}^\infty $ are bounded in {$  L_{\rm loc}^1([0, \infty)) \cap  BV_{\rm loc}([0, \infty)) $}.
\end{description}
By the compactness theories as in \cite{MR1857292} and \cite[Corollary 4]{MR0916688}, there exists a pair of functions $ [\eta, \theta] \in \{  W^{1,2}_\mathrm{loc} ([0, \infty); L^2(\Omega))^2 \} $  together with subsequences (not relabeled) of
$ \{ [\overline{\eta}_n, \overline{\theta}_n] \}_{n = 1}^\infty $, $ \{ [\underline{\eta}_n, \underline{\theta}_n] \}_{n = 1}^\infty $ and
$ \{ [\widehat{\eta}_n, \widehat{\theta}_n] \}_{n = 1}^\infty $, such that
\begin{equation}\label{S5-4}
\left\{ \hspace{-2ex} \parbox{14cm}{
\vspace{-2ex}
\begin{itemize}
\item $ \eta \in  { L_{\rm loc}^\infty([0, \infty); H^1(\Omega))} $, and $ 0 \leq \eta \leq 1 $ a.e. in $ Q $,
\vspace{-1ex}
\item $|D \theta(\cdot)|(\Omega) \in L_{\rm loc}^\infty([0, \infty)) $, and $ |\theta| \leq {|\gamma|_{L^\infty(\Gamma)}} $ a.e. in $ Q $.
\vspace{-1ex}
\item $ [\eta(0), \theta(0)] = [\eta_0, \theta_0] $ in $ L^2(\Omega) $;
\vspace{-2ex}
\end{itemize}
} \right.
\end{equation}
{
\begin{equation}\label{S5-5}
\left\{ \hspace{-2ex} \parbox{12cm}{
\vspace{-2ex}
\begin{itemize}
\item $ \widehat{\theta}_n \to \theta $ in $ {C_{\rm loc}([0, \infty); L^2(\Omega))} $, weakly in $ W^{1,2}_{\rm loc}([0, \infty), L^2(\Omega)) $ and weakly-$*$ in $ L^\infty(Q) $,
\vspace{-1ex}
\item $ \widehat{\theta}_n(t) \to \theta(t) $ in $ L^2(\Omega) $, and weakly-$ * $ in $ BV(\Omega) $, for any $ t > 0 $,
\vspace{-3ex}
\item $ \widehat{\eta}_n \to \eta $ in $ C_{\rm loc}([0, \infty); L^2(\Omega)) $, weakly-$ * $ in $ {L_{\rm loc}^\infty([0, \infty); H^1(\Omega))} $, and weakly-$*$ in $ L^\infty(Q) $,
\vspace{-1ex}
\item $ \widehat{\eta}_n(t) \to \eta(t) $ in $ L^2(\Omega) $, and weakly in $ H^1(\Omega)$, for any $ t > 0 $,
\vspace{-2ex}\end{itemize}
} \right. \mbox{ \ \ \ \ as $ n \to \infty $;}
\end{equation}}
Therefore,  we also have the following convergences as $ n \to \infty $:
\begin{equation}\label{S5-6}
\left\{ \hspace{-2ex} \parbox{10cm}{
\vspace{-2ex}
\begin{itemize}
\item $ \overline{\eta}_n \to \eta $ and $ \underline{\eta}_n \to \eta $ in $ L_{\rm loc}^\infty([0, \infty); L^2(\Omega)) $, weakly-$ * $ in $ L_{\rm loc}^\infty([0, \infty); H^1(\Omega)) $, and weakly-$*$ in $ L^\infty(Q) $,
\item $ \overline{\eta}_n(t) \to \eta(t) $, $ \underline{\eta}_n(t) \to \eta(t) $ in $ L^2(\Omega) $, weakly in $ H^1(\Omega) $, for any $ t \ge 0 $,
\vspace{-2ex}
\end{itemize}
} \right. \mbox{ \ \ \ \ as $ n \to \infty $;}
\end{equation}
\begin{equation}\label{S5-7}
\left\{ \hspace{-2ex} \parbox{10cm}{
\vspace{-2ex}
\begin{itemize}
\item $ \overline{\theta}_n \to \theta $ and $ \underline{\theta}_n \to \theta $ in {$ L_{\rm loc}^\infty((0, \infty); L^2(\Omega)) $}, and weakly-$*$ in $ L^\infty(Q) $,
\item $ \overline{\theta}_n(t) \to \theta(t)  $, $ \underline{\theta}_n(t) \to \theta(t) $ in $ L^2(\Omega) $, and weakly-$ * $ in $ BV(\Omega) $, for any $ t \ge 0 $,
\vspace{-2ex}
\end{itemize}
} \right. \mbox{ \ \ \ \ as $ n \to \infty $.}
\end{equation}
Moreover, under subsequences if necessary, we can find a function $ \mathcal{J}_{\gamma} \in BV_{\rm loc}([0, \infty)) $ such that:
\begin{equation}\label{S5-8}
\begin{array}{c}
\ds {\mathcal{F}^{\nu_n}_{\gamma}}(\underline{\eta}_n, \underline{\theta}_n) \to {\mathcal{J}}_\gamma
\mbox{ \ weakly-$*$ in $ BV_{\rm loc}((0, \infty)) $, weakly-$*$ in $ L_{\rm loc}^\infty((0, \infty)) $,}
\\[1ex]
\mbox{and a.e. in $ (0, \infty) $, \ as $ n \to \infty $.}
\end{array}
\end{equation}
Using (\ref{S5-4})--(\ref{S5-8}), we can get the energy inequality in (\hyperlink{S5}{S5}). \hfill $ \Box $

}

\subsection{Proof of Main Theorem 2}
\ \ \vspace{-3ex}

{
\begin{rem}\label{Rem.extra01}
\begin{em}
Observe that, due to the nonincreasing property of {$ \mathcal{J}_\gamma $},  the condition ``for a.e. $ 0 < s < t < \infty $'' in (\ref{kenV-41}) can be rephrased as ``for all $ 0 < s \leq t < \infty $''. Moreover, taking into account Remark \ref{Rem.G-conv01} and (\hyperlink{(S0)}{S0}) in Definition \ref{Def.Sol}, one can deduce from (\ref{kenV-41}) that
\begin{equation}\label{ene-inq1}
\begin{array}{c}
\ds \frac{1}{2} \int_s^t |\eta_t(\tau)|_{L^2(\Omega)}^2 \, d \tau +\int_s^t |{\textstyle \sqrt{\alpha_0(\eta(\tau))}} \theta_t(\tau)|_{L^2(\Omega)}^2 \, d \tau +{\mathcal{F}_{\gamma}}(\eta(t), \theta(t)) \leq \mathcal{J}_\gamma(s),
\\[2ex]
\mbox{for all $ 0 < s \leq t < \infty $.}
\end{array}
\end{equation}
\end{em}
\end{rem}
}
\bigskip

\noindent
\textbf{Proof of Main Theorem 2.}
Let $ [\eta, \theta] \in C([0, \infty); L^2(\Omega))^2 \cap W_{\rm loc}^{1, 2}((0, \infty); L^2(\Omega))^2 $ be a solution to (\hyperlink{S}{S}) with the property (S4). Then, from (\ref{ene-inq1}) and Remark \ref{Rem.G-conv01}, it is observed that
\begin{equation}\label{kenVI-01}
\left\{ \hspace{-2.5ex} \parbox{14cm}{
\vspace{-2ex}
\begin{itemize}
\item $ [\eta_t, \theta_t] \in L^2(1, \infty; L^2(\Omega))^2 $, and hence $ \eta_t(\cdot +s) \to 0 $ and $ \theta_t(\cdot +s) \to 0 $ in $ L^2(0, 1; L^2(\Omega)) $, as $ s \to \infty $,
\item $ \{ [\eta(t), \theta(t)] \, | \, t \geq 1 \} \subset F_1 := \left\{ \begin{array}{l|l}
[\tilde{\eta}, \tilde{\theta}] \in D_0 & {\mathcal{F}_{\gamma}}(\tilde{\eta}, \tilde{\theta}) \leq {\mathcal{J}_\gamma}(1)
\end{array} \right\} $,
\item $ F_1 $ is compact in $ L^2(\Omega)^2 $, because \eqref{freeEgy}, \eqref{initial-2} and Remark \ref{Rem.G-conv01} enable us to say that $ F_1 $ is closed in $ L^2(\Omega)^2 $, and bounded in $ (H^1(\Omega) \cap L^\infty(\Omega)) \times (BV(\Omega) \cap L^\infty(\Omega)) $.
\vspace{-2ex}
\end{itemize}
} \right.
\end{equation}
Therefore, we can find a pair $ [\eta_\infty, \theta_\infty] \in L^2(\Omega)^2 $ and a sequence of times $1 \leq t_1 < t_2 < t_3 < \cdots < t_n  \uparrow \infty$ as $ n \to \infty$
such that
\begin{equation}\label{kenVI-02}
[\eta(t_n), \theta(t_n)] \to [\eta_\infty, \theta_\infty] \mbox{ in $ L^2(\Omega)^2 $, as $ n \to \infty $.}
\end{equation}
This implies that $ \omega(\eta, \theta) \ne \emptyset $. Also, the compactness of $ \omega(\eta, \theta) $ is obtained from the compactness of $ F_1 $ and the fact that
\begin{equation*}
\omega(\eta, \theta) = \bigcap_{s \geq 0} \overline{\{ [\eta(t), \theta(t)] \, | \, t \geq s \}} \subset \overline{\{ [\eta(t), \theta(t)] \, | \, t \geq 1 \}} \subset F_1.
\end{equation*}
Thus, {it follows that $\omega_{\infty}(\eta,\theta)$ is nonempty and compact in $L^{2}(\Omega)^{2}$.}

\medskip
Next, we verify  (\hyperlink{S0infty}{S0})$_{\infty}$--(\hyperlink{S2infty}{S2})$_{\infty}$. We take any $ [\eta_\infty, \theta_\infty] \in \omega(\eta, \theta) $. Then, there exists a sequence of times $1 \leq t_1 < t_2 < t_3 < \cdots < t_n  \uparrow \infty$ as $ n \to \infty$
such that
(\ref{kenVI-02}) holds. Hence,  item (\hyperlink{S0infty}{S0})$_{\infty}$ is a straightforward consequence of (\hyperlink{S0infty}{S0})$_{\infty}$, (\ref{kenVI-01}) and (\ref{kenVI-02}).

\medskip In the meantime, it follows from (\ref{kenVI-01}) that
\begin{equation}\label{kenVI-03}
\left\{ \hspace{-2.5ex} \parbox{12cm}{
\vspace{-1ex}
\begin{itemize}
\item $ \{ \eta_n \}_{n = 1}^\infty := \{ \eta({}\cdot +\,t_n) \}_{n = 1}^\infty $ is bounded in $ W^{1, 2}(0, 1; L^2(\Omega)) \cap L^\infty(0, 1; H^1(\Omega)) $;
\vspace{-1ex}
\item $ \{ \theta_n \}_{n = 1}^\infty := \{ \theta({}\cdot + \, t_n) \}_{n = 1}^\infty $ is bounded in $ W^{1, 2}(0, 1; L^2(\Omega)) $, and $ \{ |D \theta_n({}\cdot{})|(\Omega) \}_{n = 1}^\infty $ is bounded in $ L^\infty(0, 1) $;
\vspace{-1ex}
\item $ \{ [\eta_n(t), \theta_n(t)] \, | \, t \in [0, 1], ~ n \in \N \} \subset D_0 $.
\vspace{-1ex}
\end{itemize}
} \right.
\end{equation}
Due to (\ref{kenVI-01}) and (\ref{kenVI-03}) { and the compactness theories as in \cite{MR1857292} and \cite[Corollary 4]{MR0916688}} we deduce that
\begin{equation}\label{kenVI-06}
\parbox{12cm}{
\begin{tabular}{rl}
$ (\eta_n)_t \to 0 $ & \hspace{-1ex} and \ $ (\theta_n)_t \to 0 $ \ in $ L^2(0, 1; L^2(\Omega)) $, \ as $ n \to \infty $,
\end{tabular}
}
\end{equation}
\begin{equation}\label{kenVI-04}
\parbox{12cm}{
\begin{tabular}{rl}
$ \eta_n \to \eta_\infty $ & in $ W^{1, 2}(0, 1; L^2(\Omega)) $, weakly-$*$ in $ L^\infty(0, 1; H^1(\Omega)) $,
\\[1ex]
& and weakly-$*$ in $ L^\infty((0, 1) \times \Omega) $, \ as $ n \to \infty $;
\end{tabular}
}
\end{equation}
\begin{equation}\label{kenVI-05}
\parbox{12cm}{
\begin{tabular}{rl}
$ \theta_n \to \theta_\infty $ & in $ W^{1, 2}(0, 1; L^2(\Omega)) $, \ as $ n \to \infty $;
\end{tabular}
}
\end{equation}
\begin{equation}\label{kenVI-07}
\parbox{12cm}{
\begin{tabular}{rl}
$ \theta_n(t) \to \theta_\infty $ & weakly-$ * $ in $ BV(\Omega) $, for any $ t \in (0, 1) $, \ as $ n \to \infty $,
\end{tabular}
}
\end{equation}
{
by taking subsequences (not relabeled) if necessary.}
On the other hand, owing to (\hyperlink{(S1)}{S1}), (\hyperlink{S2}{S2}) in Definition \ref{Def.Sol}, the sequence $ \{ [\eta_n, \theta_n] \}_{n = 1}^\infty $ satisfies
\begin{equation}\label{kenVI-08}
\begin{array}{c}
\ds \int_0^1 \bigl( (\eta_n)_t(t) +g(\eta_n(t)), w \bigr)_{L^2(\Omega)} \, dt +\int_0^1 (\nabla \eta_n(t), \nabla w)_{L^2(\Omega)^N} \, dt
\\[1.5ex]
\ds + {\int_0^1 \int_{\overline{\Omega}}  d \bigl[ w \alpha'(\eta_{n}(t)) |D \theta_{n}(t)| \bigr]_{\gamma} \, dt} = 0,
\\[2ex]
\mbox{ \ for any $ w \in H^1(\Omega) \cap L^\infty(\Omega) $ and any $ n \in \N $,}
\end{array}
\end{equation}
and
\begin{equation}\label{kenVI-09}
\begin{array}{c}
\ds \int_0^1 \bigl( \alpha_0(\eta_n(t)) (\theta_n)_t(t), {\theta_n(t)-v} \bigr)_{L^2(\Omega)} \, dt +\int_0^1 {\Phi_\gamma(\alpha(\eta_n(t)); \theta_n(t))} \, dt
\\[2ex]
\leq {\Phi_\gamma(\alpha(\eta_n(t)); v)}, \mbox{ \ for any $ n \in \N $ and $v\in BV(\Omega)\cap L^2(\Omega)$.}
\end{array}
\end{equation}

Taking into account (\ref{kenVI-01}), {(\ref{kenVI-06})--(\ref{kenVI-07})}, (\ref{kenVI-09}), and Lemma \ref{Lem.LB}, we see that
{
\begin{eqnarray}
0 & \leq & \int_0^1 \int_{\overline{\Omega}} d \bigl[ \alpha(\eta(t)) |D \theta(t)| \bigr]_{\gamma} dt \leq \liminf_{n \to \infty} \int_0^1 \int_{\overline{\Omega}} d \bigl[ \alpha(\eta_n(t)) |D \theta_n(t)| \bigr]_{\gamma} \, dt
\nonumber
\\
& \leq & \limsup_{n \to \infty} \int_0^1 \Phi_{\gamma}(\alpha(\eta_n(t)), \theta_n(t)) \, dt
\label{kenVI-10}
\\
& \leq & \lim_{n \to \infty}\int_0^1 \Phi_{\gamma}(\alpha(\eta_n(t)), v) \, dt - \lim_{n \to \infty} \int_0^1 \bigl( \alpha_0(\eta_n(t)) (\theta_n)_t(t), \theta_n(t) - v \bigr)_{L^2(\Omega)} \, dt
\nonumber
\\
& = & \Phi_{\gamma}(\alpha(\eta_\infty(t)), v).
\nonumber
\end{eqnarray}
}
By (\hyperlink{(A2)}{A2}), the above inequality implies the item (\hyperlink{S2infty}{S2})$_{\infty}$.
\medskip

Finally, with (\ref{kenVI-06})--(\ref{kenVI-07}) and (\ref{kenVI-10}) in mind, given any $ w \in H^1(\Omega) \cap L^\infty(\Omega) $, we apply Lemma \ref{Lem.core} with $ I = (0, 1) $, $ \beta = \alpha(\eta_\infty) $, $ \{ \beta_n \}_{n = 1}^\infty = \{ \alpha(\eta_n) \}_{n = 1}^\infty $, $ v =\theta_\infty $, $ \{ v_n \}_{n = 1}^\infty =  \{ \theta_n \}_{n = 1}^\infty $, $ \varrho = w \alpha'(\eta_\infty) $, $ \{ \varrho_n \}_{n = 1}^\infty = \{ w \alpha'(\eta_n) \}_{n = 1}^\infty $ and {$\{\gamma_{n}\}_{n=1}^{\infty} = \{\gamma\}$}. Then, we deduce that
{
\begin{equation}\label{kenVI-11}
\begin{array}{c}
\ds \int_0^1 \int_{\overline{\Omega}} d \bigl[ w \alpha'(\eta_n(t)) |D \theta_n(t)| \bigr]_{\gamma} \, dt \to \int_{\overline{\Omega}} d \bigl[ w \alpha'(\eta_\infty) |D \theta_\infty| \bigr]_{\gamma},
\\[2ex]
\mbox{as $ n \to \infty $, \ for any $ w \in H^1(\Omega) \cap L^\infty(\Omega) $.}
\end{array}
\end{equation}
}
With (\ref{kenVI-06})--(\ref{kenVI-07}) and (\ref{kenVI-11}) in mind, letting $ n \to \infty $ in (\ref{kenVI-08}) yields
\begin{equation*}
\begin{array}{c}
\ds (\nabla \eta_\infty, \nabla w)_{L^2(\Omega)^N} + (g(\eta_\infty), w)_{L^2(\Omega)} {+ \int_{\overline{\Omega}} d \bigl[ w \alpha'(\eta_\infty) |D \theta_\infty| \bigr]_{\gamma}} = 0, \vspace{3mm}\\
\ds \mbox{ \ for any $ w \in H^1(\Omega) \cap L^\infty(\Omega) $.}
\end{array}
\end{equation*}
{Thus, we conclude (\hyperlink{S1infty}{S1})$_\infty$.
\hfill $ \Box $}

\section{Asymptotic behavior}

In this Section we study some structural observations of the steady state solutions in the precise setting of \eqref{concrete-setting}. The qualitative properties of the steady states differ in the case of a one dimensional domain or in the two dimensional case.
\subsection{1-D case}
First of all, we consider the case that $\Omega=(0,1)$. We will show below that orientations in steady state solutions are SBV functions. We assume without loss of generality  that $0=\gamma(0)<\gamma(1)$. The Euler-Lagrange equations derived from (\hyperlink{S1infty}{S1})$_\infty$ and (\hyperlink{S2infty}{S2})$_{\infty}$ are the following ones (see \cite{MR2139257} for the equation derived from (\hyperlink{S2}{S2})):
\begin{equation}
  \label{asymp-1-d}
  \left\{\begin{array}{cc} \eta''=\eta-1+\eta|\theta'| & {\rm in \ } (0,1) \\ \mp\eta'(x)+\eta(x)|\theta(x)-\gamma(x)|=0 & {\rm at \ } \{0,1\} \\ (\alpha(\eta) w)'(x)=0 & {\rm in \ } (0,1) \\ w(x)\geq 0 & {\rm at \ } \{0,1\}\end{array}\right.,
\end{equation}
with $\|w\|_\infty\leq 1$ such that \begin{equation}\label{eq-measures-1D}\alpha(\eta)|\theta'|=-\theta(\alpha(\eta)w)'+(\alpha(\eta)\theta w)' \quad\mbox{ as measures}, \end{equation}

\begin{thm}
 For any $\gamma(1)>0$ and $\{a_k,b_k\}_k\subset [0,1]$ with $a_k< b_k<a_{k+1}$ there exists $0<d<1$ such that $$\eta(x)=\sum_{k\in\N} \left(1+(d-1)\frac{\cosh\left(x-\frac{a_k+b_k}{2}\right)}{\cosh(b_k-a_k)}\right)\chi_{\{(a_k,b_k)\}}+d\chi_{[0,1]\setminus\bigcup_{k\in\N} (a_k,b_k)},$$

and $\theta$  recovered from the Dirichlet constraints via integration by parts from
$$\theta'=\frac{1-d}{d}\left(1+\sum_{k\in\N} \tanh\left(\frac{b_k-a_k}{2}\right)\delta_{\{a_k,b_k\}}\right)$$ is a solution to  \eqref{asymp-1-d}. Moreover, these are the only solutions to \eqref{asymp-1-d}.
\end{thm}

\begin{proof}

We begin by noting that \eqref{eq-measures-1D}  implies that $|w|=1, |\theta'|$-a.e in $(0,1)$, which, in view of the boundary conditions and third equation imply that  $w=1$ $|\theta'|$-a.e in $(0,1)$ ). Note that, since $\eta\in H^1(0,1)$ and $\theta\in BV(0,1)$, \eqref{asymp-1-d} implies that $\eta'\in BV(\Omega)$ and that $w\in H^1(0,1)$. Then, $w$ is continuous in $(0,1)$.

Thus, $\{w<1\}$ is an open set. Therefore, there exists a sequence of disjoint intervals such that
    $$
    \{w<1\}=\bigcup_{k=1}^{+\infty} ]a_k,b_k[.
    $$
It follows that $\theta=C_k$ in $]a_k,b_k[$. Then, the only solution to \eqref{asymp-1-d} is $$\eta_k(x)=1+(d-1)\frac{\cosh\left(x-\frac{a_k+b_k}{2}\right)}{\cosh(b_k-a_k)},$$ with $d=\min_{[0,1]} \eta$.   Observe that $(\eta_k')^+(a_k)>0$ while  $(\eta_k')^-(b_k)<0$.

Our next aim is to prove that  $(\theta')^c=0$ in $(0,1)$.

Suppose that $x_0$ is an approximate continuity point for $\theta$ (and therefore for $\eta'$). Then, since $\eta$ is differentiable a.e. and the weak derivative coincides with the classical one, and since $\eta(x_0)$ is a minimum of $\eta$, $(\eta')^+(x_0)=(\eta')^-(x_0)=0$. In particular, $a_k,b_k$ cannot be approximate continuity points. Suppose now that $x_0\in \{w=1\}\setminus\{a_k,b_k:  k\in\N\}$. Then, we have the following two alternatives:

\begin{itemize}
    \item[(a)] There exists $\rho_0>0$ such that $(x_0-\rho_0,x_0+\rho_0)\cap \{w<1\}=\emptyset$.
    \item[(b)] For all $\rho>0$, $\exists k_n\in\N$ such that $I_{k_n}\cap (x_0,x_0+\rho)\neq \emptyset$ with $I_{k_n}:=]a_{k_n},b_{k_n}[$.
\end{itemize}

    For possibility ({a}), we observe that $\eta=d$ in $(x_0-\rho_0,x_0+\rho_0)$, therefore, we get that $|\theta'|=\frac{1-d}{d}$ in $(x_0-\rho_0,x_0+\rho_0)$ and then $x_0$ cannot be a Cantor point of $\theta'$.

\medskip
    If ({b}) is true, then letting $\Lambda_\rho:=\{\lambda\in N: I_\lambda\cap (x_0,x_0+\rho)\neq\emptyset\}$ we obtain that $$\eta'(y)=\left\{\begin{array}
  {cc}\displaystyle \frac{(d-1)\sinh\left(y-\frac{a_\lambda+b_\lambda}{2}\right)}{\cosh\left(\frac{b_\lambda-a_\lambda}{2}\right)} & {\rm if \ } y\in I_\lambda \\ 0 & {\rm if \ } y\notin \bigcup_{\lambda\in\Lambda_\rho} I_\lambda \, .
\end{array}\right.$$
Therefore, its distributional derivative is $$D\eta'=\sum_{\lambda_\in \Lambda_\rho} \frac{(d-1)\cosh\left(\cdot-\frac{a_\lambda+b_\lambda}{2}\right)}{\cosh\left(\frac{b_\lambda-a_\lambda}{2}\right)}\chi_{I_\lambda}+\sum_{\lambda\in \Lambda_\rho}(1-d)\tanh\left(\frac{b_\lambda-a_\lambda}{2}\right)\delta_{\{a_\lambda,b_\lambda\}}.$$ Then, $$\frac{|D\eta'|(x_0,x_0+\rho)}{\rho}=\frac{4(1-d)}{\rho}\sum_{\lambda\in \Lambda_\rho}\tanh\left(\frac{b_\lambda-a_\lambda}{2}\right)\leq 2(1-d),$$ which implies that $x_0$ does not belong to the Cantor set of $D\eta'$ (equivalently, of $\theta'$).

Therefore, $(\theta')^{c}=0$ and we obtain the description of $\eta$, $\theta$ as stated.

\end{proof}

\subsection{2-D case: Radially symmetric solutions}
Here, we consider $N=2$, $\Omega=B(0,R)\setminus B(0,r_0)$ and $\theta=\theta(|x|)$, $\eta=\eta(|x|)$. Then, the system in radial coordinates reads as:
\begin{equation*}
    \left\{\begin{array}{cc} \displaystyle -\eta''(r)-\frac{\eta'(r)}{r}+\eta(r)-1+\eta(r)|\theta'(r)|=0 & {\rm in \ } ]r_0,R[ \\ \mp\eta'+\alpha'(\eta)|\theta-\gamma|=0 & {\rm on \ } \{r_0,R\}  \\ \displaystyle (\alpha(\eta(r))\omega(r))'+\frac{\alpha(\eta(r))\omega(r)}{r}=0 & {\rm in \ } ]r_0,R[ \end{array}\right.,
\end{equation*}
where $\|w\|\leq 1$ satisfies \eqref{eq-measures-1D}. We restrict our analysis to the case of
piecewise constant solutions for $\theta$ and give sufficient conditions for the solutions having one or two jumps to exist. The first example shows that precisely one jump at the interior boundary is always possible.

\begin{exmp}
  Suppose that $\theta=\gamma(R)>0=\gamma(0)$. Then, the solution for $\eta$ is given by  $$\eta(r)=1+A I_0(r)+B K_0(r),$$ with $I_0$ being the modified Bessel function of the first kind and $K_0$ the modified Bessel function of the second kind. Coupled with the boundary conditions, this yields $$\eta(r)=1-\frac{{\gamma(R)}(I_0(r)+K_0(r)T_1(R))}{K_1(r_0)(T_1(R)-T_1(r_0))+{\gamma(R)}(I_0(r)+K_0(r)T_1(R))}, $$where $T_j(r):=\frac{I_j(r)}{K_j(r)}$. In fact, it suffices to take $$w(r):=\frac{r_0 \alpha(\eta(r_0))}{r\alpha(\eta(r))}.$$
\end{exmp}

Next, we show that solutions with exactly one jump at the exterior boundary may exist, but extra conditions on the radius (and therefore on curvature) for this to happen are needed. This phenomenon is clearly explained by the fact that, formally, the evolution of the solutions is a combination of mean curvature motion for the jumps of rotation, coupled with total variation evolution for the values of the rotation (see \cite{kobayashi2000continuum}).
\begin{exmp}
  Suppose that $\theta=0=\gamma(0)<{\gamma(R)}$. Then, as before,  $$\eta(r)= 1+\frac{(d-1)K_0(r)(T_0(r)+T_1(r_0))}{K_0(R)(T_0(R)+T_1(r_0))} ,$$ where $d=\eta(R)$.  Note that, $d$ is easily obtained by the relation:
{
\begin{equation*}
     \frac{(d-1)K_1(R)}{K_0(R)}\cdot\frac{T_1(r_0)-T_1(R)}{T_0(R)+T_1(r_0)}=d \gamma(R).
 \end{equation*}
}
We observe that, defining $w(r):=\frac{R\alpha(d)}{r\alpha(\eta(r))}$ we will have to show that $|w|\leq 1$. For that, we point out that the function $\sigma\mapsto \sigma\alpha(\eta(\sigma))$ is increasing at $r_0$ and at $R$ (here we need $R{\gamma(R)}\geq \frac{1}{2}$. Moreover, $w$ does not have an interior maximum since
$$w'(r)= -\frac{R\alpha(d)}{r^2\alpha^2(\eta(r))}(\alpha+r\eta(r)\eta'(r))=0\leftrightarrow \eta(r)\eta'(r)=-\frac{\alpha(\eta(r))}{r}.$$Then,
at a critical point,
$$w''(r)=-\frac{R\alpha(d)}{r^3\alpha^2(\eta(r))\eta^2(r)\left(-\frac{\eta^4(r)}{2}+\delta^2+r^2\eta(r)(\eta(r)-1)\right)}>0\,,\quad \mbox{ for $\delta$ small enough}$$

  Therefore, if we show that $w(r_0)\leq 1$ then $w(r)\leq 1$ in $[r_0,R]$.

    We note that for {$w(r_0)=\frac{R(\frac{d^2}{2}+\delta)}{r_0(\frac{\eta^2(r_0)}{2}+\delta)}\leq 1$} to hold, then, necessarily $Rd^2\leq r_0\eta^2(r_0)$. Moreover, by continuity, once we show that $Rd^2< r_0\eta^2(r_0)$, there will be $\delta$ sufficiently small such that $w(r_0)\leq 1$ still holds. Then, we restrict the rest of the discussion to  the case $\delta=0$.

    After some computations, we obtain that $w(r_0)\leq 1$ holds if and only if the following function is nonnegative:
\begin{equation}\label{condexistence1jump}
    f(r_0,R)={\gamma(R) (r_0 b(r_0,R)-1)}-\sqrt{r_0}(\sqrt{R}-\sqrt{r_0})\frac{\partial b}{\partial R}(r_0,R),
\end{equation}
 with
{
\begin{equation}\label{expressionforb}
    b(x,y):=I_0(y) K_1(x)+ I_1(x) K_0(y).
\end{equation}
}

 \begin{lem}
   Given $r_0>0$ such that $r_0{\gamma(R)}\geq 1$ there exists a unique $R^*$ such that $f(r_0,R)\geq 0$ for all $R\in ]r_0,R^*]$.
 \end{lem}
\begin{proof}
  We fix $r_0$ and we let $g(x):=f(r_0,x)$ and $a(x):=b(r_0,x)$. We will show that, if $r_0{\gamma(R)}\geq 1$, then there is only one critical point (and therefore a global positive maximum since $g(x)\to -\infty$ as $x\to +\infty$) of $g$. In fact, (noting that $a''=a-\frac{a'}{x}$)
  $$g'(x)=\sqrt{r_0}\left(a'(x)\left({\gamma(R)}\sqrt{r_0}+\frac{1}{2\sqrt{x}}-\sqrt{\frac{r}{x}}\right)+(\sqrt{r}-\sqrt{x})a(x)\right).$$
  Observe that $g'(r_0)=0$ and that $g''(r_0)\geq 0$ if, and only if $r_0{\gamma(R)}\geq1$. Moreover, in the case $r_0{\gamma(R)}=1$, $g'(r_0)=g''(r_0)=0$ and $g'''(r_0)>0$.

  At a critical point,
  $$\frac{a'(x)}{a(x)}=\frac{2x(\sqrt{x}-\sqrt{r_0})}{2\sqrt{r_0}(\gamma(R) x-1)+\sqrt{x}}.$$ However, $y=\frac{a'}{a}$ satisfies $$y'=1-\frac{y}{x}-y^2\,,\quad y(r_0)=1,$$ and therefore it is a concave increasing function from $0$ to $1$ in $[r_0,+\infty[$ with
  initial slope $1$. On the other hand: $$h(x):=\frac{2x(\sqrt{x}-\sqrt{r_0})}{2\sqrt{r_0}(\gamma(R) x-1)+\sqrt{x}}$$ is easily seen to be increasing, with initial slope (at $r_0$)  less than or equal to $1$ (strictly less if $r_0{\gamma(R)}>1$ and tending to $\infty$. Moreover,
    it is concave at $r_0$ and at $+\infty$ and it only has two inflexion points in $]r_0,+\infty[$. Therefore, both functions $y$ and $h$ intersect only once {(see (a) and (b) in Figure 1)}.
\end{proof}
    {\small
    \begin{center}
        \textsc{\normalsize Figure 1:} numerics in the case when $ r_0 = 1 $
        \\[0ex]
        \includegraphics[width=12.5cm]{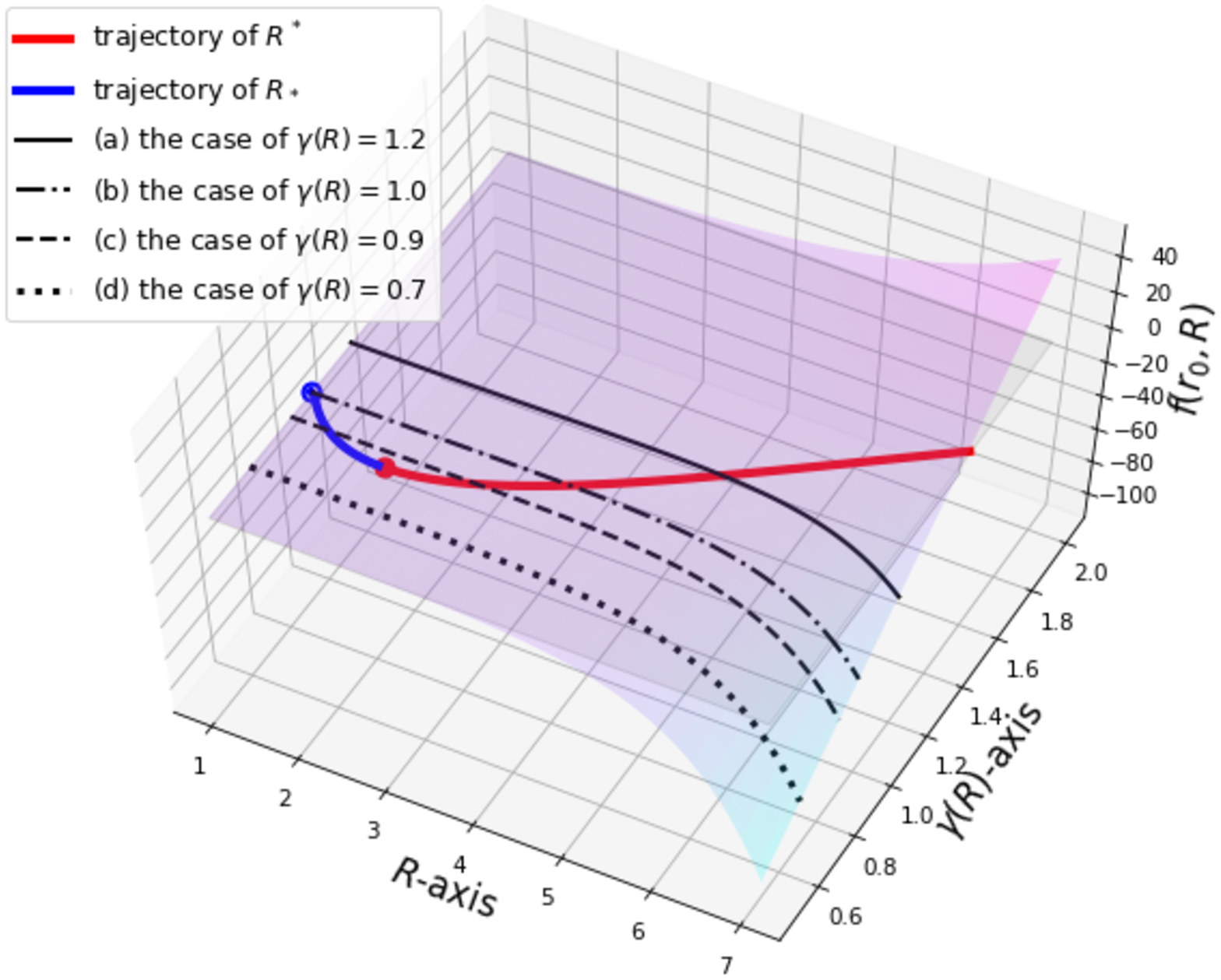}
    \end{center}
}
\noindent
{\bf Comment:} In case that $r{\gamma(R)}<1$ it might happen that $f$ has two critical points (and therefore there exists an interval $[R_*,R^*]$ such that if $r_0<R_*<R<R^*$ the problem has a solution {
    (see (c) in Figure 1)} or that there is no solution at all
    {
    (see (d) in Figure 1)}.
    \medskip

Altogether, we have obtained that, if $r_0{\gamma(R)}>1$, there exists an interval of possible radii $R\in ]r_0,R^*]$ such that the problem has a solution. In the case that $R>R^*$ the problem admits no solution.
\end{exmp}

We address now the case of exactly one jump on the interior of the domain, and we give sufficient conditions for this to happen.
\begin{exmp}
  Suppose that $r_0<r_1\leq R$ and that $\theta=\gamma(R)\chi_{[r_1,R]}$. Then, as before,  boundary conditions, boundary conditions at the jump point $r_1$ and continuity yield that the solution is exactly $$\eta(r)=\left\{\begin{array}
    {cc} \displaystyle 1+\frac{(d-1)K_0(r)(T_0(r)+T_1(r_0))}{K_0(r_1)(T_0(r_1)+T_1(r_0))} & {\rm in \ } [r_0,r_1] \\ \displaystyle 1+\frac{(d-1)K_0(r)(T_0(r)+T_1(R))}{K_0(r_1)(T_0(r_1)+T_1(R))} & {\rm in \ } [r_1,R]
  \end{array}\right. ,$$ where $d=\eta(r_1)$. Observe that, $d$ is easily by the relation:
    \begin{equation}\label{solutionford}
 \frac{(d-1)K_1(r_1)}{K_0(r_1)}\frac{(T_1(r_0)-T_1(R))(T_1(r_1)+T_0(r_1))}{(T_0(r_1)+T_1(R))(T_0(r_1)+T_1(r_0))}=d{\gamma(R)}.
 \end{equation}
  We observe that, defining $w(r):=\frac{r_1\alpha(\eta(r_1))}{r\alpha(\eta(r))}$ we will have to show that $|w|\leq 1$. For that, we observe that , as in the previous example, it suffices to show that $w(r_0)\leq 1$.

  First of all,
\begin{equation}\label{conditionford}
    w(r_0)=\frac{r_1 d^2}{r_0{\eta(r_0)^2}}\leq 1 \Leftrightarrow \eta(r_0)\sqrt{\frac{r_0}{r_1}}{\geq}d \Leftrightarrow d\leq \frac{r_0 b(r_0,r_1)-1}{\sqrt{r_1 r_0}b(r_0,r_1)-1},
\end{equation}
with $b$ defined as in \eqref{expressionforb}.
Then, using \eqref{solutionford}, we get
$$
    w(r_0)\leq 1 \Leftrightarrow F(r_0,r_1,R):=\frac{\sqrt{r_0}(\sqrt{r_1}-\sqrt{r_0})\frac{\partial b}{\partial R}(r_0,R)}{r_1 b(R,r_1)(r_0b(r_0,r_1)-1)}\leq \gamma(R).
$$
Observe that
{
$$
    \frac{\partial F}{\partial R}(r_0,r_1,R)=
     \frac{T_1'(R) (T_0(r_1) +T_1(r_0))F(r_0, r_1, R)}{(T_1(R) -T_1(r_0))(T_0(r_1) +T_1(R))} \geq 0.
$$
}
    Then, for $w(r_0)\leq 1$ to hold, it is necessary that $$F(r_0,r_1,r_1)=\frac{\sqrt{r_0}(\sqrt{r_1}-\sqrt{r_0})\frac{\partial b}{\partial r_1}(r_0,r_1)}{r_0 b(r_0,r_1)-1)}\leq \gamma(R); $$
    i.e. $f(r_0,r_1)\geq 0$ with $f$ defined in \eqref{condexistence1jump}, which implies that {$r_1\in ]r_0,R^*] $ with $R^*$} the one obtained in last example.
However, in order for $w(r_0)\leq 1$ to be satisfied, we observe that defining
\begin{equation*}
    G(r_0,r_1,R):=
     \gamma(R)r_1 b(R,r_1)(r_0 b(r_0,r_1)-1)-\sqrt{r_0}(\sqrt{r_1}-\sqrt{r_0})\frac{\partial b}{\partial R}(r_0,R),
\end{equation*}
in order that $G(r_0,r_1,R)\geq 0$, since $G(r_0,r_0,R)=0$ and $\frac{\partial G}{\partial r_1}(r_0,r_0,R)<0$, $r_1$ cannot be too close to $r_0$.

One sufficient condition is that
{
\begin{align*}
    \lim_{R\to\infty} G(r_0,r_1,R) &~ = \lim_{R \to \infty} I_1(R) \bigl( \gamma(R) r_1  K_0(r_1)(r_0 b(r_0,r_1)-1) -\sqrt{r_0}(\sqrt{r_1} -\sqrt{r_0})K_1(r_0) \bigr)
    \\
    &~ = \infty,
\end{align*}
}
i.e. \begin{equation}\label{sufficient condition 1 interior jump} \gamma(R) r_1  K_0(r_1)(r_0 b(r_0,r_1)-1) -    \sqrt{r_0}(\sqrt{r_1} -\sqrt{r_0})K_1(r_0)\geq 0,\end{equation}
since $\frac{\partial G}{\partial R}(r_0,r_1,R)<0$ for $G\leq 0$.
We illustrate this condition into two examples:
  \begin{enumerate}
      \item[{(a)}] $r_0=1$, $\gamma(R)=2$, the sufficient condition \eqref{sufficient condition 1 interior jump} is not verified for any $r_1$. However, as it can be seen in Figure 2, there is still a range of $r_1,R$ for which \eqref{conditionford} is satisfied, thus implying existence of solutions.
      \item[{(b)}] $r_0=4$, $\gamma(R)=2$, the sufficient condition is verified for a range of $r_1$ (and then there is a solution for all $R>r_1$.
  \end{enumerate}
    \begin{center}
        {\small
        \textsc{\normalsize Figure 2:} contour maps of $ G(r_0, r_1, R) $ under fixed $ r_0 > 0 $.
        \vspace{-1ex}
        $$
        \left\{ \hspace{-3ex} \parbox{8cm}{
            \vspace{-2ex}
            \begin{itemize}
                \item blue zone: the existence area of solution
                \vspace{-1.5ex}
                \item gray zone: non-effective area $ \{ r_1 > R \} $
            \vspace{-2ex}
            \end{itemize}
        } \right.
        \vspace{-1ex}
        $$
        \hspace{-2ex}\begin{tabular}{cc}
            \begin{minipage}[t]{8cm}
                \begin{center}
                    \includegraphics[height=5.85cm]{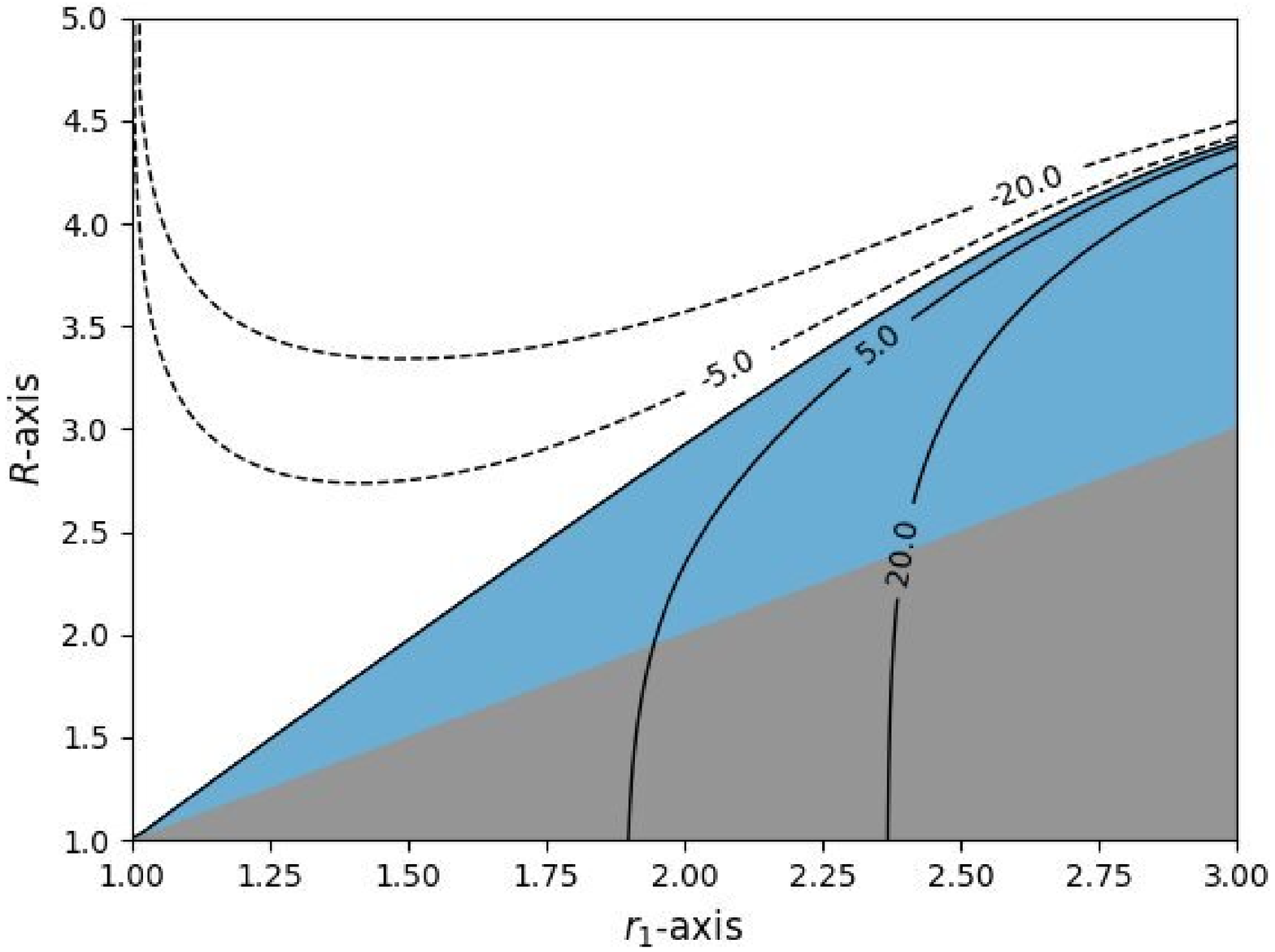}
                    \\
                    (a) the case of $ r_0 = 1.0 $ and $ \gamma(R) = 2.0 $
                \end{center}
            \end{minipage}
            &
            \begin{minipage}[t]{8cm}
                \begin{center}
                    \includegraphics[height=5.85cm]{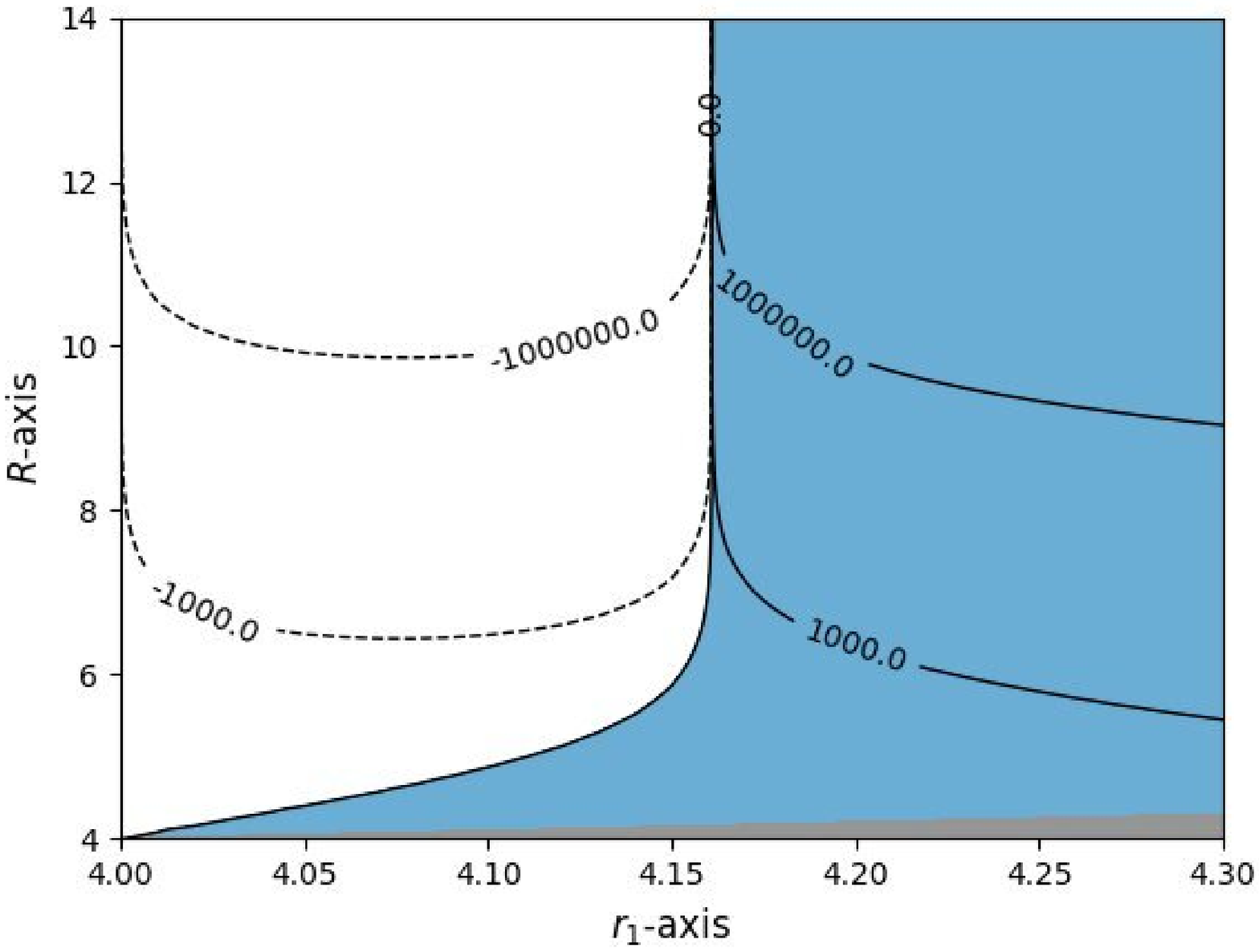}
                    \\
                    (b) the case of $ r_0 = 4.0 $ and $ \gamma(R) = 2.0 $
                \end{center}
            \end{minipage}
        \end{tabular}
        }
    \end{center}
\end{exmp}

The last example deals with the case of two possible jumps. We show that solutions may exist but now the casuistic depending on the radius of where solutions may jump and on both internal and external radius are much complicated than in previous cases.
\begin{exmp}
  Suppose that $\theta=\theta_0\chi_{[r_1,r_2]}+\gamma(R)\chi_{]r_2,R]}$. Then, letting $d_1:=\alpha(r_1)$ and $d_2:=\eta(r_2)$, we get  $d_2=\alpha^{-1}(\alpha(d_1)\frac{r_1}{r_2})$. Moreover, the solution for $\eta$ is given by $$\eta(r)=\left\{\begin{array}
    {cc} \displaystyle 1+\frac{(d_1-1)K_0(r)(T_0(r)+T_1(r_0))}{K_0(r_1)(T_0(r_1)+T_1(r_0))} & {\rm in \ } [r_0,r_1] \\ \\ 1+K_0(r)\left(\frac{\frac{1-d_1}{K_0(r_1)}-\frac{1-d_2}{K_0(r_2)}}{T_0(r_2)-T_0(r_1)} (T_0(r)-T_0(r_1))+\frac{d_1-1}{K_0(r_1)}\right) & {\rm in \ } [r_1,r_2]\\ \\ \displaystyle 1+\frac{(d_2-1)K_0(r)(T_0(r)+T_1(R))}{K_0(r_2)(T_0(r_2)+T_1(R))} & {\rm in \ } [r_2,R]
  \end{array}\right. ,$$ coupled with some compatibility conditions imposed from $$(\eta'(r_1))^+-(\eta'(r_1))^-=d_1\theta_0\,,\quad (\eta'(r_2))^+-(\eta'(r_2))^-=d_2(f(R)-\theta_0);$$which read as $$\left\{\begin{array}{ccc} C_1(d_1-1)+C_2(d_2-1) & = & (\theta_0)d_1 \\ C_3(d_1-1)+C_4(d_2-1)& =& (\gamma(R)-\theta_0)d_2\end{array}\right.,$$ with
$$\begin{array}{ccc}
     \displaystyle C_1 & = & \displaystyle-\frac{T_1(r_0)+T_0(r_2)}{r_1 K_0^2(r_1)(T_1(r_0)+T_0(r_1))(T_0(r_2)-T_0(r_1))}\\ \\  C_2 & = & \displaystyle \frac{1}{r_1 K_0(r_1) K_0(r_2)(T_0(r_2)-T_0(r_1))}  \\ \\ C_3 & = & \displaystyle \frac{1}{r_2 K_0(r_1) K_0(r_2)(T_0(r_2)-T_0(r_1))} \\ \\ C_4 & = & \displaystyle-\frac{T_1(R)+T_0(r_1)}{r_2 K_0^2 (r_2)(T_1(R)+T_0(r_2))(T_0(r_2)-T_0(r_1))}
  \end{array}$$
  In the case that $\alpha(\sigma)=\frac{\sigma^2}{2}$, then $d_2=d_1\sqrt{\frac{r_1}{r_2}}$ and therefore, fixing $r_1,r_2$, the compatibility conditions turn into
 \begin{equation}\label{condition2_jumps_positivity}
     \left\{\begin{array}{ccc} d_1(C_1+{\sqrt{\frac{r_1}{r_2}}C_2}-(C_1+C_2) & = & {\theta_0 d_1}
 \\
 d_1(C_3+\sqrt{\frac{r_1}{r_2}}C_4)-(C_3+C_4)& =& (\gamma(R)-\theta_0)d_2
 \end{array} \right..
 \end{equation}
 In particular, the following inequality need to be verified (since $C_1+\sqrt{\frac{r_1}{r_2}}C_2\leq C_1+C_2\leq 0$),
 $$
 0\leq d_1\leq \frac{C_1+C_2}{C_1+\sqrt{\frac{r_1}{r_2}}C_2}.
 $$
 Moreover, it is easy to see that
 $$
     \frac{C_1+C_2}{C_1+\sqrt{\frac{r_1}{r_2}}C_2}\leq 1.
 $$
 On the other hand, if $C_3+\sqrt{\frac{r_1}{r_2}}C_4\geq 0$, then, {the following necessary condition for the second equation in \eqref{condition2_jumps_positivity}:
 \begin{equation*}
     d_1 \geq \frac{C_3 +C_4}{C_3 +\sqrt{\frac{r_1}{r_2}} C_4} \geq 0,
 \end{equation*}
 }
is always satisfied. Otherwise, $$\frac{C_3+C_4}{
C_3+\sqrt{\frac{r_1}{r_2}}C_4}\geq 1.$$Therefore, for \eqref{condition2_jumps_positivity} to hold it suffices to show that
$$
    d_1\leq \frac{C_1+C_2}{C_1+\sqrt{\frac{r_1}{r_2}}C_2}.
$$

Adding the two equalities in \eqref{condition2_jumps_positivity}, we obtain
{
\begin{align}\label{condition2jumps}
    0<d_1= ~& -\frac{C_1+C_2+\sqrt{\frac{r_2}{r_1}}(C_3+C_4)}{\gamma(R)-(C_1+\sqrt{\frac{r_1}{r_2}}C_2+\sqrt{\frac{r_2}{r_1}}C_3+C_4)}
    \nonumber
    \\
    \leq &~ \bar{d}_1 := \frac{C_1+C_2}{C_1+\sqrt{\frac{r_1}{r_2}}C_2}.
\end{align}
}
We observe that this is not always the case since the following inequality is always true (and therefore {$\gamma(R)$} needs to be big enough):
\begin{equation}\label{not_always}
    \frac{C_1+C_2+\sqrt{\frac{r_2}{r_1}}(C_3+C_4)}{C_1+\sqrt{\frac{r_1}{r_2}}C_2+\sqrt{\frac{r_2}{r_1}}C_3+C_4}\geq \frac{C_1+C_2}{C_1+\sqrt{\frac{r_1}{r_2}}C_2}.
\end{equation}
To show this last inequality we use the fact that
$$
C_1+\sqrt{\frac{r_1}{r_2}}C_2+\sqrt{\frac{r_2}{r_1}}C_3+C_4=C_1+C_2+{\sqrt{\frac{r_2}{r_1}}} \left(2-\sqrt{\frac{r_2}{r_1}}\right)C_3+C_4\leq C_1+C_2+C_3+C_4\leq 0.
$$
Then, \eqref{not_always} is seen to be equivalent to $$C_2 C_3\leq C_1 C_4,$$ which is easy to show to hold.

  Observe that we have not imposed that $r_1>r_0$. In case $r_1=r_0$, then \eqref{condition2jumps} will suffice for existence of solution. However, in the case that $r_0<r_1$, then there is an extra compatibility condition as in the previous examples:
  Defining $w(r):=\frac{r_1\alpha(d)}{r\alpha(\eta(r))}$, the new compatibility condition is
    {
    \begin{equation}\label{condition3}
        1 -w(r_0) \geq 0.
    \end{equation}
}

  We conclude by showing some examples of the difficulty of the casuistic. We first show an example (Fig.3) in which \eqref{condition2jumps} is violated for any admissible range of boundary data and therefore no solution exists. In the second example (Fig.4), we obtain two different threshold values for $\gamma(R)$ for which \eqref{condition2jumps}, resp. \eqref{condition3}, is satisfied. We point out that, in the particular case of Fig. 4, the only possible situation for an admissible (big enough) boundary datum is $r_0=r_1$.
\begin{center}
\begin{tabular}{cc}
    \begin{minipage}[t]{7.77cm}
        {\small
        \begin{center}
            \textsc{\normalsize Figure 3:} numerics in the case when
            \\[1ex]
            $ r_0 = 0.5 $, $ r_1 = 1.0 $, $ r_2 = 9.0 $, $ R = 10.0 $
            \\[1ex]
        \includegraphics[width=7.5cm]{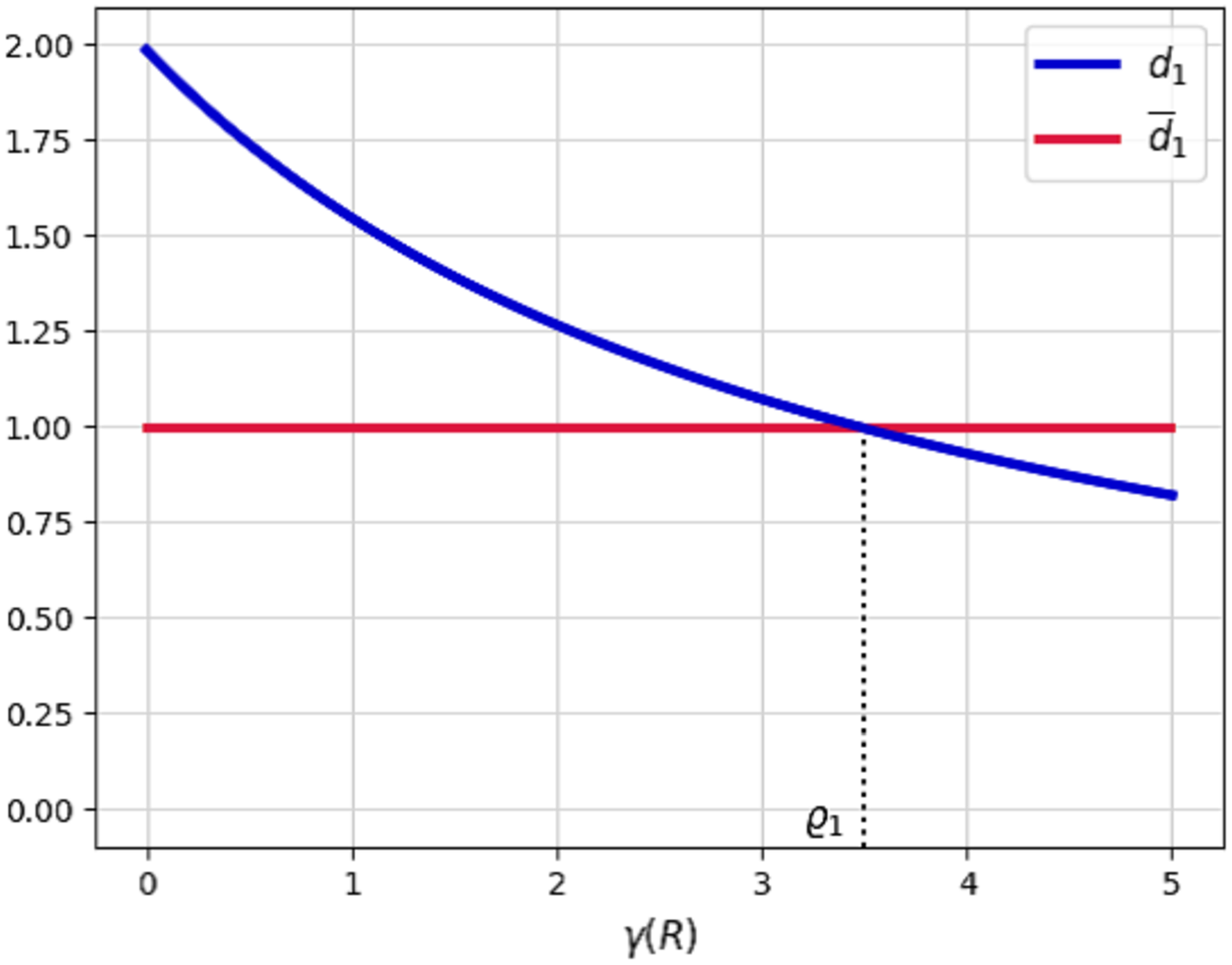}
        \\[1ex]
        \parbox{7.25cm}{In this case, we find a threshold point $ \varrho_1 $ $ (\approx 3.5) $ such that \eqref{condition2jumps} does not hold for all $ \gamma(R) \in [0, \varrho_1) $, and the range $ [0, \varrho_1) $ contains the physically admissible range $ [0, \pi) $ of the crystalline orientation angle.
            }
        \end{center}
        }
    \end{minipage}
    & \begin{minipage}[t]{7.77cm}
    {\small
        \begin{center}
            \textsc{\normalsize Figure 4:} numerics in the case when
            \\[1ex]
            $ r_0 = 1.0 $, $ r_1 = 2.5 $, $ r_2 = 9.0 $, $ R = 10.0 $
            \\[1ex]
        \includegraphics[width=7.5cm]{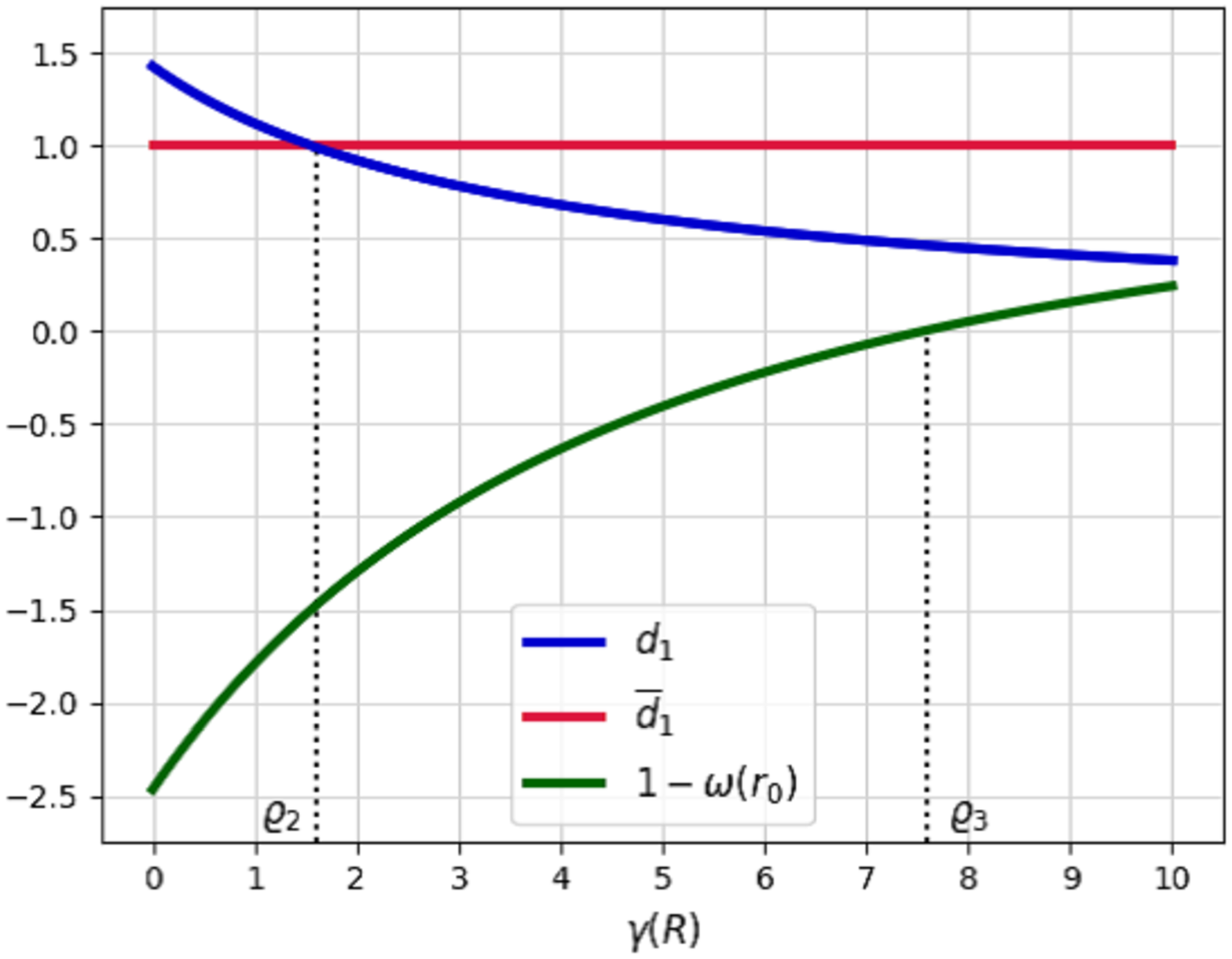}
        \\[1ex]
            \parbox{7.25cm}{In this case, we find two threshold points $ \varrho_2 \in (1, 2) $ and $ \varrho_3 \in (7, 8) $ such that \eqref{condition2jumps} (resp. \eqref{condition3}) holds iff. $ \gamma(R) \geq \varrho_2 $ (resp. $ \gamma(R) \geq \varrho_3 $).
            }
        \end{center}
    }
    \end{minipage}
\end{tabular}
\end{center}
\end{exmp}


\begin{thebibliography}{10}

\bibitem{MR2306643}
Amar, M.; De~Cicco, V.; Fusco, N.
\newblock A relaxation result in {BV} for integral functionals with
  discontinuous integrands.
\newblock {\em ESAIM Control Optim. Calc. Var.}, {\bfseries 13}(2):396--412,
  2007.

\bibitem{MR1857292}
Ambrosio, L.; Fusco, N.; Pallara, D.
\newblock {\em Functions of Bounded Variation and Free Discontinuity Problems}.
\newblock Oxford Mathematical Monographs. The Clarendon Press, Oxford
  University Press, New York, 2000.

\bibitem{MR1814993}
Andreu, F.; Ballester, C.; Caselles, V.; Maz\'on, J.~M.
\newblock The {D}irichlet problem for the total variation flow.
\newblock {\em J. Funct. Anal.}, {\bfseries 180}(2):347--403, 2001.

\bibitem{MR3288271}
Attouch, H.; Buttazzo, G.; Michaille, G.
\newblock {\em Variational analysis in {S}obolev and {BV} spaces}, volume~17 of
  {\em MOS-SIAM Series on Optimization}.
\newblock Society for Industrial and Applied Mathematics (SIAM), Philadelphia,
  PA; Mathematical Optimization Society, Philadelphia, PA, second edition,
  2014.
\newblock Applications to PDEs and optimization.

\bibitem{MR0348562}
Br\'ezis, H.
\newblock {\em Op\'erateurs Maximaux Monotones et Semi-groupes de Contractions
  dans les Espaces de {H}ilbert}.
\newblock North-Holland Publishing Co., Amsterdam-London; American Elsevier
  Publishing Co., Inc., New York, 1973.
\newblock North-Holland Mathematics Studies, No. 5. Notas de Matem\'atica (50).

\bibitem{MR1201152}
Dal~Maso, G.
\newblock {\em An Introduction to {$\Gamma$}-convergence}, volume~8 of {\em
  Progress in Nonlinear Differential Equations and their Applications}.
\newblock Birkh\"auser Boston, Inc., Boston, MA, 1993.

\bibitem{MR0117523}
Dunford, N.; Schwartz, J.~T.
\newblock {\em Linear {O}perators. {I}. {G}eneral {T}heory}.
\newblock With the assistance of W. G. Bade and R. G. Bartle. Pure and Applied
  Mathematics, Vol. 7. Interscience Publishers, Inc., New York; Interscience
  Publishers, Ltd., London, 1958.

\bibitem{MR3409135}
Evans, L.~C.; Gariepy, R.~F.
\newblock {\em Measure Theory and Fine Properties of Functions}.
\newblock Textbooks in Mathematics. CRC Press, Boca Raton, FL, revised edition,
  2015.

\bibitem{MR0102739}
Gagliardo, E.
\newblock Caratterizzazioni delle tracce sulla frontiera relative ad alcune
  classi di funzioni in {$n$} variabili.
\newblock {\em Rend. Sem. Mat. Univ. Padova}, {\bfseries 27}:284--305, 1957.

\bibitem{MR0775682}
Giusti, E.
\newblock {\em Minimal Surfaces and Functions of Bounded Variation}, volume~80
  of {\em Monographs in Mathematics}.
\newblock Birkh\"auser Verlag, Basel, 1984.

\bibitem{MR2469586}
Ito, A.; Kenmochi, N.; Yamazaki, N.
\newblock A phase-field model of grain boundary motion.
\newblock {\em Appl. Math.}, {\bfseries 53}(5):433--454, 2008.

\bibitem{MR2548486}
Ito, A.; Kenmochi, N.; Yamazaki, N.
\newblock Weak solutions of grain boundary motion model with singularity.
\newblock {\em Rend. Mat. Appl. (7)}, {\bfseries 29}(1):51--63, 2009.

\bibitem{MR2836555}
Ito, A.; Kenmochi, N.; Yamazaki, N.
\newblock Global solvability of a model for grain boundary motion with
  constraint.
\newblock {\em Discrete Contin. Dyn. Syst. Ser. S}, {\bfseries 5}(1):127--146,
  2012.

\bibitem{Kenmochi81}
Kenmochi, N.
\newblock Solvability of nonlinear evolution equations with time-dependent
  constraints and applications.
\newblock {\em Bull. Fac. Education, Chiba Univ.
  (\url{http://ci.nii.ac.jp/naid/110004715232})}, {\bfseries 30}:1--87, 1981.

\bibitem{MR0372419}
Kenmochi, N.
\newblock Pseudomonotone operators and nonlinear elliptic boundary value
  problems.
\newblock {\em J. Math. Soc. Japan}, {\bfseries 27}:121--149, 1975.

\bibitem{MR2668289}
Kenmochi, N.; Yamazaki, N.
\newblock Large-time behavior of solutions to a phase-field model of grain
  boundary motion with constraint.
\newblock In {\em Current advances in nonlinear analysis and related topics},
  volume~32 of {\em GAKUTO Internat. Ser. Math. Sci. Appl.}, pages 389--403.
  Gakk\=otosho, Tokyo, 2010.

\bibitem{kobayashi2000continuum}
Kobayashi, R.; Warren, J.~A.; Carter, W.~C.
\newblock A continuum model of grain boundaries.
\newblock {\em Physica D: Nonlinear Phenomena}, {\bfseries 140}(1):141--150,
  2000.

\bibitem{MR1794359}
Kobayashi, R.; Warren, J.~A.; Carter, W.~C.
\newblock Grain boundary model and singular diffusivity.
\newblock In {\em Free boundary problems: theory and applications, {II}
  ({C}hiba, 1999)}, volume~14 of {\em GAKUTO Internat. Ser. Math. Sci. Appl.},
  pages 283--294. Gakk\=otosho, Tokyo, 2000.

\bibitem{MR0241822}
Lady{\v z}enskaja, O.~A.; Solonnikov, V.~A.; Ural'ceva, N.~N.
\newblock {\em Linear and Quasilinear Equations of Parabolic Type}, volume~23
  of {\em Translations of Mathematical Monographs}.
\newblock American Mathematical Society, Providence, R.I., 1968.

\bibitem{MR0244627}
Ladyzhenskaya, O.~A.; Ural'tseva, N.~N.
\newblock {\em Linear and quasilinear elliptic equations}.
\newblock Translated from the Russian by Scripta Technica, Inc. Translation
  editor: Leon Ehrenpreis. Academic Press, New York-London, 1968.

\bibitem{MR0350177}
Lions, J.-L.; Magenes, E.
\newblock {\em Non-homogeneous boundary value problems and applications. {V}ol.
  {I}}.
\newblock Springer-Verlag, New York-Heidelberg, 1972.
\newblock Translated from the French by P. Kenneth, Die Grundlehren der
  mathematischen Wissenschaften, Band 181.

\bibitem{MR2072104}
Loeb, P.~A.; Talvila, E.
\newblock Lusin's theorem and {B}ochner integration.
\newblock {\em Sci. Math. Jpn.}, {\bfseries 60}(1):113--120, 2004.

\bibitem{MR2139257}
Moll, J.~S.
\newblock The anisotropic total variation flow.
\newblock {\em Math. Ann.}, {\bfseries 332}(1):177--218, 2005.

\bibitem{MR3268865}
Moll, S.; Shirakawa, K.
\newblock Existence of solutions to the {K}obayashi--{W}arren--{C}arter system.
\newblock {\em Calc. Var. Partial Differential Equations}, {\bfseries
  51}(3-4):621--656, 2014.

\bibitem{MR3670006}
Moll, S.; Shirakawa, K.; Watanabe, H.
\newblock Energy dissipative solutions to the {K}obayashi--{W}arren--{C}arter
  system.
\newblock {\em Nonlinearity}, {\bfseries 30}(7):2752--2784, 2017.

\bibitem{MR3751650}
Nakayashiki, R.; Shirakawa, K.
\newblock Weak formulation for singular diffusion equation with dynamic
  boundary condition.
\newblock In {\em Solvability, regularity, and optimal control of boundary
  value problems for {PDE}s}, volume~22 of {\em Springer INdAM Ser.}, pages
  405--429. Springer, Cham, 2017.

\bibitem{MR3038131}
Shirakawa, K.; Watanabe, H.; Yamazaki, N.
\newblock Solvability of one-dimensional phase field systems associated with
  grain boundary motion.
\newblock {\em Math. Ann.}, {\bfseries 356}(1):301--330, 2013.

\bibitem{MR3082861}
Shirakawa, K.; Watanabe, H.
\newblock Energy-dissipative solution to a one-dimensional phase field model of
  grain boundary motion.
\newblock {\em Discrete Contin. Dyn. Syst. Ser. S}, {\bfseries 7}(1):139--159,
  2014.

\bibitem{MR3462536}
Shirakawa, K.; Watanabe, H.
\newblock Large-time behavior for a {PDE} model of isothermal grain boundary
  motion with a constraint.
\newblock {\em Discrete Contin. Dyn. Syst.}, {\bfseries 1}(Dynamical systems,
  differential equations and applications. 10th AIMS Conference.
  Suppl.):1009--1018, 2015.

\bibitem{MR3362773}
Shirakawa, K.; Watanabe, H.; Yamazaki, N.
\newblock Phase-field systems for grain boundary motions under isothermal
  solidifications.
\newblock {\em Adv. Math. Sci. Appl.}, {\bfseries 24}(2):353--400, 2014.

\bibitem{MR0916688}
Simon, J.
\newblock Compact sets in the space {$L^p(0,T;B)$}.
\newblock {\em Ann. Mat. Pura Appl. (4)}, {\bfseries 146}:65--96, 1987.

\bibitem{MR3203495}
Watanabe, H.; Shirakawa, K.
\newblock Qualitative properties of a one-dimensional phase-field system
  associated with grain boundary.
\newblock In {\em Nonlinear analysis in interdisciplinary
  sciences---modellings, theory and simulations}, volume~36 of {\em GAKUTO
  Internat. Ser. Math. Sci. Appl.}, pages 301--328. Gakk\=otosho, Tokyo, 2013.

\bibitem{MR3238848}
Watanabe, H.; Shirakawa, K.
\newblock Stability for approximation methods of the one-dimensional
  {K}obayashi-{W}arren-{C}arter system.
\newblock {\em Math. Bohem.}, {\bfseries 139}(2):381--389, 2014.

\end{thebibliography}
\end{document}